\documentclass[10pt]{amsart}
\usepackage{amsmath, amsfonts, amsthm, amssymb}
\usepackage{comment}
\usepackage{verbatim} 
\usepackage[hidelinks]{hyperref}
\usepackage{caption}

\usepackage{mathtools}
\usepackage{tikz-cd}
\usepackage{color}
\usepackage{xcolor}

\usepackage{cite}
\usepackage{enumerate}

\usepackage{cite}
\usepackage[margin=1in]{geometry}
\usepackage{mathtools}    
\DeclarePairedDelimiterX\setc[2]{\{}{\}}{\,#1 \;\delimsize\vert\; #2\,}
\usepackage{float}
\usepackage{url}

\makeatletter
\newtheorem*{rep@theorem}{\rep@title}
\newcommand{\newreptheorem}[2]{%
\newenvironment{rep#1}[1]{%
 \def\rep@title{#2 \ref{##1}}%
 \begin{rep@theorem}}%
 {\end{rep@theorem}}}
\makeatother

\usepackage{MnSymbol,wasysym}
\usepackage{multicol}
\usepackage{caption}
\usepackage{subcaption}
\usepackage{array}
\renewcommand{\epsilon}{\varepsilon}
\graphicspath{ {./Images/} }

\theoremstyle{definition}\newtheorem{defn}{Definition}[subsection]
\theoremstyle{theorem}\newtheorem{thm}{Theorem}[subsection]
\theoremstyle{definition}\newtheorem{eg}{Example}[subsection]
\theoremstyle{definition}
\theoremstyle{definition}\newtheorem*{thmu}{Theorem}
\theoremstyle{theorem}\newtheorem{prop}[thm]{Proposition}
\theoremstyle{theorem}\newtheorem{conj}[thm]{Conjecture}
\theoremstyle{definition}
\theoremstyle{definition}
\theoremstyle{definition}
\theoremstyle{definition}
\theoremstyle{theorem}\newtheorem{lemma}[thm]{Lemma}
\theoremstyle{definition}\newtheorem{remark}{Remark}[subsection]
\theoremstyle{definition}
\theoremstyle{definition}
\theoremstyle{definition}
\theoremstyle{theorem}\newtheorem{cor}[thm]{Corollary}

\newcommand{\A}{\mathbf{A}}
\newcommand{\J}{\mathcal{J}}

\newcommand{\M}{\mathcal M}

\newcommand{\F}{\mathcal F}
\renewcommand{\L}{\mathcal{L}}
\renewcommand{\O}{\mathcal{O}}
\newcommand{\Gabcd}{\mathcal{G}\left(\mathbf{A}\right)}

\newcommand{\Facbd}{\mathcal{F}\left(\mathbf{A}\right)}

\newcommand{\G}{\mathcal{G}}

\newcommand{\N}{\mathbb N}
\newcommand{\Z}{\mathbb Z}
\newcommand{\C}{\mathbb C}

\renewcommand{\P}{\mathbb P}
\newcommand{\Q}{\mathbb Q}

\newcommand{\R}{\mathbb R}

\newcommand{\D}{\mathbb D}

\newcommand{\mC}{\mathcal{C}}

\newcommand{\Ob}{\mathrm{Ob}}

\newcommand{\val}{\mathfrak{v}}
\newcommand{\e}{\hat{e}}

\newcommand{\Hom}{\mathrm{Hom}}
\newcommand{\hol}{\mathrm{hol}}
\newcommand{\Ext}{\mathrm{Ext}}

\newcommand{\out}{\mathrm{out}}

\renewcommand{\phi}{\varphi}
\newcommand{\hn}{{\hat{n}}}

\newcommand{\End}{\mathrm{End}}

\newcommand{\into} {\hookrightarrow}

\newcommand{\p}{\partial}

\newcommand{\im}{\operatorname{Im}}

\newcommand{\GL}{\operatorname{GL}}

\newcommand{\ind}{\operatorname{ind}}
\newcommand{\lin}{\operatorname{lin}}

\newcommand{\Pic}{\operatorname{Pic}}

\newcommand{\W}{{\mathcal{W}r}}
\newcommand{\Wr}{{\mathcal{W}r}}
\newcommand{\coker}{\text{coker}}

\newcommand{\supp}{\operatorname{Supp}}

\newcommand{\Id}{\operatorname{Id}}

\newcommand{\FS}{\mathcal{FS}}
\newcommand{\Coh}{\textrm{Coh}}

\newcommand{\ev}{\operatorname{ev}}
\newcommand{\an}{\operatorname{an}}
\renewcommand{\mod}{\mathrm{mod}}

\newcommand{\inv}{^{-1}}

\newcommand{\wt}{\widetilde}
\newcommand{\wh}{\widehat}
\newcommand{\whCF}{\widehat{CF}}
\newcommand{\whHF}{\widehat{HF}}

\begin{document}

\begin{abstract} We show homological mirror symmetry results relating coherent analytic sheaves on some complex elliptic surfaces and objects of certain Fukaya categories. We first define the notion of a non-algebraic Landau-Ginzburg model on $\R \times \left(S^1\right)^3$ and its associated Fukaya category, and show that non-K\"{a}hler surfaces obtained by performing two logarithmic transformations to the product of the projective plane and an elliptic curve have non-algebraic Landau-Ginzburg models as their mirror spaces; this class of surface includes the classical Hopf surface $S^1 \times S^3$ and other elliptic primary and secondary Hopf surfaces. We also define localization maps from the Fukaya categories associated to the Landau-Ginzburg models to partially wrapped and fully wrapped categories. We show mirror symmetry results that relate the partially wrapped and fully wrapped categories to spaces of coherent analytic sheaves on open submanifolds of the compact complex surfaces in question, and we use these results to sketch a proof of a full HMS result.

 \end{abstract}

\title{Homological mirror symmetry for elliptic Hopf Surfaces}
\author{Abigail Ward
}

\maketitle

\tableofcontents

\section{Introduction}
For $q \in \C^*$ with $|q| < 1$, there is a free action of $\Z$ on $\C^n \setminus \{0\}$ given by scaling by $q$. When $n = 1$, the quotient space $\C^*/ \left(z \sim q z\right)$ is the elliptic curve $E_\tau$ with modular parameter $\tau$ satisfying $q=e^{2 \pi i \tau}.$ When $n=2$, the quotient is the  \emph{Hopf surface}: 
 \[ S_{\tau} = \C^2 \setminus \{0\} / \left(\mathbf{z} \sim q \mathbf{z}\right). \]
The map 
 \[ \mathbf{z}  \mapsto \left( \log |\mathbf{z}|, \frac{\mathbf{z}}{|\mathbf{z}|}\right) \in \left(\R/ |2 \pi i \tau| \Z, \C^2\right) \]
is a diffeomorphism $S_{\tau} \to S^1 \times S^3$. Since $H^2\left(S_{\tau}, \Z\right) = 0$, the Hopf surface is non-K\"{a}hler and non-algebraic. The purpose of this work is to demonstrate that despite its lack of algebraic or symplectic structure, the Hopf surface still exhibits homological mirror symmetry behavior; that is, there exists a mirror space $Y$ to $S_{\tau}$ and a Fukaya category $\F_{\tau}$ associated to $Y$ such that there is a correspondence between coherent analytic sheaves on $S_{\tau}$ and objects of $\F_{\tau}$.

 The starting point of our mirror symmetry correspondence is the observation that there exists an anti-canonical divisor $D$ on $S_{\tau}$ such that $S_{\tau}\setminus D$ is a more familiar complex surface to which we can associate a predicted mirror space. Indeed,  let $D \subset S_{\tau}$ denote the image of the two coordinate axes $\{z_1 = 0 \} \sqcup \{z_2 = 0 \}$ in $\C^2 \setminus \{0\}$ under the quotient map $\C^2 \setminus \{0\} \to S_{\tau}$; $D$ consists of the disjoint union of two copies of the elliptic curve $E_\tau$. There is a meromorphic volume form on $S_{\tau}$ given by 
\[ \frac{d z_1 \wedge d z_2}{z_1 z_2} \]
that has simple poles on the components of $D$. Moreover, $S_{\tau} \setminus D \cong \C^* \times E_{\tau}$ is an open Calabi-Yau manifold with a symplectic form which is compatible with its complex structure, and a fibration by special Lagrangian tori. Following the Strominger-Yau-Zaslow philosophy (\hspace{1sp}\cite{SYZ}), one predicts that the mirror to $\C^* \times E_{\tau}$ is the moduli space of these tori equipped with a unitary flat connection, which is parametrized by $\C^* \times E^{\tau}$, where $E^{\tau}$ is the elliptic curve mirror to $E_{\tau}$.

We then recall the following heuristic for performing mirror symmetry relative to an anti-canonical divisor (see \cite{AuMSTD}). Let $\left(X,D_X\right)$ be a pair consisting of a K\"{a}hler manifold $X$ and an anti-canonical divisor $D_X$ with the property that $X \setminus D_X$ is open Calabi-Yau and admits a fibration by Lagrangian tori. For such a pair $\left(X, D_X\right)$, the mirror operation to compactifying $X \setminus D_X$ by $D_X$ is introducing a Landau-Ginzburg superpotential $W$ on the open SYZ mirror $\left(X\setminus D\right)^\vee$ which encodes the data of Maslov index 2 discs that pass through $D$. To this Landau-Ginzburg model one associates a \emph{Fukaya-Seidel category} denoted $\FS\left(\left(X \setminus D\right)^\vee, W\right)$; the objects of this $A_\infty$ category are Lagrangians in $\left(X \setminus D\right)^\vee$ that obey boundary conditions specified by $W$. This method gives a mirror space to various non-Calabi-Yau $X$, e.g., for toric varieties (in \cite{Ab1} and later work), and one can then verify the mirror correspondence by checking the predictions of homological mirror symmetry.

Inspired by this philosophy, we construct a \emph{non-algebraic Landau-Ginzburg model} associated to the pair $\left(S,D\right)$ and other pairs consisting of a complex surface $S_{\tau}\left(\mathbf{A}\right)$ and an anti-canonical divisor $D_{\mathbf{A}}$ such that $S_{\tau} \left(\mathbf{A}\right) \setminus D_{\mathbf{A}} = \C^* \times E_{\tau}$. Each such compactification of $\C^* \times E_{\tau}$ that we consider is obtained by applying two logarithmic transformations to $\P^1 \times E$ over two points in $\P^1$, and the index $\mathbf{A}$ records which compactification we obtain in this way. Recall that the logarithmic transformation is an operation on elliptic complex surfaces which takes an elliptic surface $X$ with a fibration $f:X \to B$ to a new elliptic surface $f': X' \to B$, with the property that there exists a point $b \in B$ such that there is a biholomorphism
\[ X\setminus f\inv(b) \to X' \setminus \left(f'\right)\inv(b). \]
Topologically, this transformation amounts to performing a surgery in which one removes a neighborhood $U=D^2 \times f\inv(b)$ of the fiber $f\inv(b)$ and glues the neighborhood back in via a diffeomorphism $\varphi: \p U = T^3 \to T^3$. One imposes conditions on $\phi$ so that the resulting space has a complex structure. The logarithmic transformation has two parameters $\left(m,k\right)$ which are a pair of integers such that $\left(m,k\right)$ is a primitive element in $\Z^2$ and $m >0$; the parameters $\left(1,0\right)$ give the trivial case where the diffeomorphism $\varphi$ is the identity. We let $D\left(m,k\right)$ denote the result of performing a logarithmic transformation with parameters $\left(m,k\right)$ at the origin at $\C \times E$. By performing two logarithmic transformations to $\P^1 \times E$, one over each pole, with parameters $\left(m_0, k_0\right)$ and $\left(m_\infty,k_{\infty}\right)$, we obtain an elliptic surface $S_{\tau}\left(\left(m_0,k_0\right),\left(m_\infty,k_\infty\right)\right)$; we let $\mathbf{A}$ denote the pair $\left(\left(m_0,k_0\right),\left(m_\infty,k_\infty\right)\right)$.  In this notation, we may write the Hopf surface $S_{\tau}$ as
\[ S_{\tau} = S_{\tau}\left(\left(1,0\right), \left(1,1\right)\right). \]

 Fix the anti-canonical divisor $D_{\P^1 \times E} =\left( \{0\} \times E_\tau \right) \cup\left( \{\infty\} \times E_{\tau}\right) \subset \P^1 \times E_{\tau}$; let $D_{\mathbf{A}}$ denote the union of the fibers over $0$ and $\infty$ in $S\left(\mathbf{A}\right)$. By design we have
\[ S_{\tau} \left(\mathbf{A}\right) \setminus D_{\mathbf{A}}= \left(\P^1 \times E_{\tau}\right) \setminus D_{\P^1 \times E_{\tau}} = \C^* \times E_\tau. \]
When $\left(m_0, k_0\right)=\left(m_\infty,-k_\infty\right)$, the resulting surface is a quotient of $\P^1 \times E$ by the action of a finite group; otherwise, the resulting surface is non-algebraic. 

We then associate to each surface $S_{\tau}\left(\mathbf{A}\right)$ a non-algebraic Landau-Ginzburg model denoted by $\left(Y, \mathbf{A}\right)$ and an associated Fukaya category $\F\left(\mathbf{A}\right)$. Here $Y=\left(\R \times \left(S^1\right)^3, d t \wedge d \phi_t + d s \wedge d \phi_s\right)$ should be thought of as the mirror of $\C^* \times E_{\tau}$. These categories are defined over the Novikov ring with formal variable $T$; we show that there are ``evaluation maps" (determined by $T \mapsto e^{2 \pi i \tau}$) from  $\F\left(\mathbf{A}\right)$ to $\C$-linear categories denoted by $H^0 \F_\tau$. The construction bypasses the process of defining of a Landau-Ginzburg superpotential by counting holomorphic discs in $S_\tau\left(\mathbf{A}\right)$ as we would be unable to define the energy (symplectic area) of any such disc. Instead, we specify boundary conditions that Lagrangians will need to satisfy at each end of $Y$.
 
 \begin{defn} A type-$\mathbf{A}$ Lagrangian in $Y$ satisfies the following condition: there exists a $T > 0$ such that the function 
\begin{equation}\label{ac} m_0 \phi_t + k_0 \phi_s  \end{equation}
vanishes identically when restricted to $L \cap \left(-\infty, -T\right] \times \left(S^1\right)^3$, and
\begin{equation}\label{bd} -m_\infty \phi_t + k_\infty \phi_s \end{equation}  
vanishes identically when restricted to $L \cap \left[T, \infty\right) \times \left(S^1\right)^3$.
\end{defn}

The objects of $\F\left(\mathbf{A}\right)$ are (roughly speaking)  type-$\mathbf{A}$ Lagrangians which are linear at infinity and have well-defined Floer theory. The construction of the category is similar to that in \cite{GPS2} which associates to a Landau-Ginzburg model a stopped Fukaya category. 

We conjecture the following:
\begin{conj}\label{expected} There is a full and faithful functor
\[ D^b \Coh_{\an} S_{\tau} \left(\mathbf{A}\right) \to D^\pi \F_{\tau} \left(\mathbf{A}\right). \] \end{conj}

 Our first theorem provides evidence of this conjecture:

\begin{thm}\label{picequiv}For each $\mathbf{A}$ such that $S_{\tau}(\A)$ is not an algebraic surface, there is a full and faithful functor
\begin{equation}\Phi_{\mathbf{A}}: \Pic_{\an} S_\tau \left(\mathbf{A}\right) \to H^0 \mathcal{F}_{\tau}\left(\mathbf{A}\right) \end{equation} 
\end{thm}
The image of each line bundle is a Lagrangian section. This mirror correspondence confirms the expectation from previous instances of SYZ mirror symmetry (e.g., in the family Floer program of Abouzaid, see \cite{AbFFC}) that holomorphic vector bundles of rank $n$ are mirror to Lagrangian multi-sections that have intersection number $n$ with a fiber. 

However, the derived category of $S_{\tau}(\A)$ is not (split-)generated by line bundles and torsion sheaves, so one cannot prove Conjecture \ref{expected} using this collection of objects. This motivates us to work towards a more structural understanding of the objects of the categories of the conjecture. To this end, we first describe sheaves on $S_{\tau}\left(\mathbf{A}\right)$ using a cover of $S_{\tau}\left(\mathbf{A}\right)$ that reflects its construction in terms of the logarithmic transformations. Let $D_0$ to be the reduced fiber lying over $0$ and $D_\infty$ the reduced fiber lying over $0$.  There are restriction functors \[ \Coh_{\an} S_{\tau}(\A) \to \Coh_{\an} \left(S_\tau \setminus D_0\right) =  \Coh_{\an} D_{\tau}\left(m_\infty, k_{\infty}\right)\] and \[\Coh_{\an} S_\tau(\A)  \to \Coh_{\an} \left(S_{\tau} \setminus D_\infty\right) = \Coh _{\an} D_{\tau} \left(m_0, k_0\right) .\] 
We  define subcategories $\mathcal{C}_{\tau}\left(m, k\right) \subset \Pic_{\an} D_{\tau} \left(m, k\right)$ and $\mathcal{C}_{\tau} \subset \Pic_{\an} \C^* \times E$; the first consists of line bundles which can extend from $D_{\tau}\left(m, k\right)$ to some $S_{\tau}\left(\mathbf{A}\right)$, and the second consists of line bundles that can extend from $\C^*_{\tau} \times E$ to some compactification $S_{\tau}\left(\mathbf{A}\right)$. Thus a line bundle on $S_{\tau}\left(\mathbf{A}\right)$ can be described as a pair of objects of $\mathcal{C}_{\tau}\left(m_0, k_0\right)$  and $\mathcal{C}_{\tau}\left(m_\infty, k_\infty\right)$ that restrict to the same object of $\mathcal{C}_{\tau}$, along with gluing data.

In parallel, we decompose the boundary conditions that a Lagrangian $L \in \F\left(\mathbf{A}\right)$ satisfies. This leads us to define two Fukaya categories with relaxed boundary conditions: the  ``partially wrapped" category $\F\left(m, k\right)$ is defined by relaxing the boundary conditions at one end, and the ``fully wrapped" category $\F_{\W}$ by relaxing them at both ends: 
\[ \mathrm{Ob}\, \F\left(m, k\right) = \bigcup_{\left(m_\infty,k_\infty\right)} \Ob\, \, \F\left(\left(m, k\right), \left(m_\infty,k_\infty\right)\right) \]
and
\[ \mathrm{Ob}\, \F_{\W} = \bigcup_{\left(\left(m_0,k_0\right),\left(m_\infty,k_\infty\right)\right)} \Ob\, \, \F\left(\left(m_0, k_0\right), \left(m_\infty,k_\infty\right)\right). \] 

There are localization maps
\begin{equation} \mathcal{F}\left(\left(m_0, k_0\right),\left(m_\infty, k_\infty\right)\right) \to \mathcal{F}\left(m_0, k_0\right) \to \mathcal{F}_{\W}; \end{equation}
\begin{equation}  \mathcal{F}\left(\left(m_0, k_0\right),\left(m_\infty, k_\infty\right)\right)\to \mathcal{F}\left(m_\infty, k_\infty\right) \to \mathcal{F}_{\W} \end{equation}
and 
\begin{equation} \mathcal{F}\left(m,k\right) \to \mathcal{F}_{\W}. \end{equation}
On objects, the first and third maps are inclusions; the second differs by an interchange of the ends of $Y$ and can be expressed as the composition of maps
  \[ \mathcal{F}\left(\left(m_0, k_0\right),\left(m_\infty, k_\infty\right)\right) \to  \mathcal{F}\left(\left(m_\infty, k_\infty\right), \left(m_0, k_0\right)\right) \to \F\left(m_\infty, k_\infty\right).\] 

One expects objects of $\F\left(m,k\right)$ and $\F_{\W}$ to correspond to coherent analytic sheaves on $D_{\tau}\left(m,k\right)$ and $\C^* \times E_{\tau}$, respectively. One also expects the localization maps of Fukaya categories to be mirror to localization maps of sheaves. These expectations are confirmed with the following theorems:

 \begin{thm}\label{openfunctors}

\begin{enumerate}[(a)] \item
For each $\left(m, k\right)$ there is a full and faithful functor
\begin{equation*} \Phi_{\left(m,k\right)}: \mathcal{C}_{\tau}\left(m, k\right) \to H^0 \F_\tau\left(m, k\right). \end{equation*}

\item There is a full and faithful functor
\begin{equation*} \Phi_{\W}: \mathcal{C}_{\tau} \to H^0 \F_{\tau, \Wr}\end{equation*}\end{enumerate}

\end{thm}

\begin{thm}\label{diagram} The following diagram of categories commutes:
\begin{equation}\label{thediag}
    \begin{tikzcd}[row sep=1.15 em, column sep = 1.5em]
   \Pic  S_{\tau} \left(\mathbf{A}\right)  \arrow[rr] \arrow[dr, swap] \arrow[dd,swap] &&
    \mathcal{C}_{\tau} \left(m_0, k_0\right) \arrow[dd] \arrow[dr] \\
    & H^0 \F_\tau\left(\mathbf{A}\right) \arrow[rr] \arrow[dd] &&
    H^0 \F_\tau \left(m_0, k_0\right) \arrow[dd] \\
    \mathcal{C}_{\tau} \left(m_\infty , k_\infty\right) \arrow[rr,] \arrow[dr] && \mathcal{C}_{\tau} \arrow[dr] \\
    & H^0 \F_\tau \left(m_{\infty},k_\infty\right) \arrow[rr] && H^0 \F_{\tau, \Wr}
    \end{tikzcd}
\end{equation}
where each arrow is given by the functor between the corresponding categories explained above.
\end{thm}

These theorems provide an explicit and geometric example of a homological mirror symmetry correspondence in the case when the complex side is not K\"{a}hler and admits no symplectic structure. Mirror symmetry for non-K\"{a}hler space has been investigated in some other settings. Most recently, in \cite{aldega}, \'{A}lvarez-C\'{o}nsul, de Arriba de la Hera, and Garcia-Fernandez  show that certain homogeneous surfaces, including the Hopf surface, admit $(0,2)$ mirrors with isomorphic half-twisted models.  In \cite{LTY} Lau, Tseng, and Yau use a Fourier-Mukai transform to obtain a correspondence between different Ramond-Ramond flux terms for pairs of Type IIA and Type IIB supersymmetric systems, and in \cite{Po} Popovici  shows that there is a classical mirror symmetry correspondence for the Iwasawa surface in the sense that there is an isomorphism between the Gauduchon cone (parametrizing the space of symplectic structures) and a space of complex deformations of the same surface.  In \cite{AbFFC}, Abouzaid sketches the construction of a candidate SYZ mirror to the Kodaira-Thurston surface using the existence of a fibration of the surface by smooth Lagrangian tori, and states that the Lagrangian sections of the SYZ fibration should correspond to coherent sheaves twisted by a gerbe on an abelian surface.  Our construction bears most resemblance to this last work.

In the remainder of the introduction, we give more details of the construction of the Fukaya categories, as well as a proposed proof of Conjecture \ref{expected}.

\subsection{Mirror symmetry for compact surfaces} 
The Fukaya-Seidel category associated to a Landau-Ginzburg model was constructed by Seidel (\hspace{1sp}\cite{seidelbook}, \cite{SeLF1}) in the case where the the superpotential is a Lefschetz fibration. There have been recent formulations of this category in terms of stops on a Liouville domain in \cite{SyPW} and in terms of Liouville sectors in \cite{GPS1}. Our construction is closest in spirit to the formulation using a Liouville sector with stops. To the symplectic manifold $Y = \R \times \left(S^1\right)^3$ equipped with the standard symplectic form $\omega = d t \wedge d \phi_t + d s \wedge d \phi_s$, we associate an $A_\infty$ category denoted $\F\left(\mathbf{A}\right)$ with objects linear type-$\mathbf{A}$ Lagrangians with unitary flat connection. This category is morally similar to the monomially admissible Fukaya category introduced in work of Hanlon (\hspace{1sp}\cite{Ha}), in which one demands that Lagrangians have specified radial behavior in certain regions. The actual construction of the category proceeds through a localization procedure similar to that in \cite{Ha} or \cite{GPS1} and inspired by work of Abouzaid and Seidel. There is a quasi-isomorphism
\[ \Hom_{\F\left(\mathbf{A}\right)}^*\left(L,L'\right) \to C F^* \left(\psi_H\left(L\right), \psi_{H'}\left(L'\right)\right) \]
where $H$ and $H'$ are two Hamiltonians chosen so that  $\psi_H\left(L\right)$ and $\psi_{H'}\left(L'\right)$ intersect transversely, such that away from a compact set each Hamiltonian is equal to $\epsilon |t|$ for some $\epsilon$, and such that $\psi_H\left(L\right)$ and $\psi_{H'}\left(L'\right)$ are in \emph{correct position}. Note that at infinity we always perturb in the $t$ direction, so the direction we perturb in is independent of the choice of boundary condition.

The conditions on $H$ and $H'$ ensure that   $\psi_H\left(L\right)$ and $\psi_{H'}\left(L'\right)$ are disjoint away from a compact set, so $\Hom_{\F\left(\mathbf{A}\right)}^*\left(L,L'\right)$ is necessarily finite-dimensional. This agrees with the expectation from the mirror symmetry correspondence: for any two coherent analytic sheaves $\mathcal{F}$ and $\mathcal{G}$ on a compact complex space $X$, $\Hom_{D^b \Coh X} \left(\F, \G[i]\right)$ is finite-dimensional by Cartan-Serre finiteness (\hspace{1sp}\cite{CaSe}). This is in contrast to the situation when $X$ is non-compact.

\subsection{Mirror symmetry for $D_{\tau}\left(m,k\right)$ and $\C^* \times E_{\tau}$}
The definition of $\F\left(m,k\right)$ and $\F_{\W}$ originates from the requirement that the diagram of Theorem \ref{diagram} accommodate any pair $\mathbf{A}=\left(\left(m_0,k_0\right),\left(m_\infty, k_\infty\right)\right)$. The category $\F_{\W}$ is an enlargement of the wrapped category  which is the expected mirror to $\C^* \times E_{\tau}$. This enlargement does not respect any notion of a``cylindrical Lagrangian;"  the objects that make up $\F(\W)$ cannot be simultaneously invariant under the flow of any ``Liouville vector field". Rather, the objects of $\F_{\W}$ are those  which are invariant under \emph{some} linear outward pointing vector field.

The categories $\F\left(m,k\right)$ and $\F_{\W}$ are computed by a categorical localization process. The morphisms can be computed by ``wrapping at an end;" this amounts to perturbing by Hamiltonians which are close to quadratic functions of $t$ at the end where the boundary conditions were relaxed.

\subsubsection{Convergence}

There are several technical challenges that arise in this process. Since  $\R \times \left(S^1\right)^3$ is not an exact symplectic manifold, we cannot directly apply results from \cite{GPS1}. We rectify some of the problems that arise by keeping track of the areas of holomorphic discs and by imposing strong geometric conditions on our Lagrangians (for example, we demand that the Lagrangians are exact when lifted to the universal cover $\R^4$). However, because of the lack of exactness (and the resulting lack of \emph{a priori} energy bounds on J-holomorphic curves), we can no longer bound the image of discs which are counted in the $A_\infty$ products for a fixed set of generators of the Floer complexes; thus the output of a $\mu^n$ product can (and will) have nonzero coefficients on infinitely many generators. We accommodate this possibility by completing the Floer complex to allow infinite series. We first define all Floer complexes over the Novikov ring. In this case, the completion is equal to the following:

\[ \widehat{CF}^*\left(L, K\right) =\left\{ \sum_{ y_i \in L \cap K}  c_i y_i \, \mid \, c_i \in \Lambda, \lim_{t\left(y_i\right) \to \pm \infty}  \val\left(c_i\right) \to \infty \right\}.\]
In Section 4 we prove that this completion is closed under taking $A_{\infty}$ products.

The mirror symmetry statements in Theorem \ref{picequiv}  are equivalences of $\C$-linear categories. These categories require a slightly different completion of the Floer complex, which is described in Section \ref{comps}. From the point of view of mirror symmetry, the completion process is a natural operation which is mirror to completing the space of algebraic functions on $D_{\tau}\left(m, k\right)$ and $\C^* \times E_\tau$ to the space of holomorphic functions. We use the basic fact from complex analysis that one can express the domain of convergence of a Laurent series on $\C^*$ using the limit of the $n$th roots of the absolute value of its coefficients. We also show that we have necessary convergence results for the $\mu^2$ products.

\subsection{Proposed proof of mirror symmetry}\label{pproof}
One hopes that the diagram of Theorem \ref{diagram} can be upgraded to a proof of Conjecture \ref{expected}. We expect that $\F\left(\mathbf{A}\right)$ is the pull back of the diagram of gluing $\F\left(m_0, k_0\right)$ and  $\F\left(m_\infty,k_\infty\right)$ by $\F_{\W}.$ To obtain a proof, we would then express $D^b \Coh_{\an} \, S_{\tau}\left(\mathbf{A}\right)$ as the pull back of a diagram that glues the subcategories of $D^b \Coh_{\an}\, D_\tau\left(m_0, k_0\right)$ and $D^b \Coh_{\an}\, D_\tau\left(m_\infty, k_\infty\right)$  generated by objects in $\mathcal{C}_{m_0,k_0}$  (denoted $\mathcal{D}_\tau\left(m_0, k_0\right)$) and $\mathcal{C}\left(m_\infty,k_\infty\right)$ (denoted  $\mathcal{D}_\tau \left(m_\infty,k_\infty\right)$) via the category $\mathcal{D}_\tau$ which is the subcategory of  $D^b \Coh_{\an} \C^* \times E_\tau$ generated by $\mathcal{C}$. The maps $\mathcal{C}_\tau\left(m_0, k_0\right) \to H^0 \F_q^{s}\left(m_0, k_0\right)$ and $\mathcal{C}_\tau \to H^0 \F_{\W,q}^{s}$ could be upgraded to maps $\mathcal{D}_\tau \left(m_0, k_0\right) \to D^\pi \F_\tau \left(m_0, k_0\right)$ and $\mathcal{D} \to D^\pi \left(\F_{\tau, \Wr}\right)$.
Gluing derived categories is in general a process requiring some care, but in this case the fact that there are only two categories being glued simplifies the computations significantly (\hspace{1sp}\cite{KuLu}). 

In this work, we only compute on the level of cohomology categories; this allows us to compute only $\mu^1$ and $\mu^2$. The proposed proof would require calculating the higher products $\mu^n$ in $\F_{\W}$, which would require significantly more computation. It would also require proving that the categories $\mathcal{D}_{m,k}$ are large enough -- that is, that $\mathcal{D}_\tau \left(m, k\right)$ contains the image of the map
\[ D^b \Coh_{\an}\, S_\tau \left(\left(m, k\right),\left(m_\infty,k_\infty\right)\right) \to D^b \Coh_{\an} \, D_\tau \left(m, k\right) \]
induced by restriction for all pairs $\left(\left(m, k\right),\left(m_\infty,k_\infty\right)\right)$.

The idea of proving mirror symmetry diagrammatically is not new; it is found in, e.g., work of Lee \cite{Lee}  (although in this work is in the ``other direction" -- the Landau Ginzburg model is found on the complex side of the mirror pair).

\subsection{Organization} The paper is organized as follows. In Section \ref{classi} we classify the complex surfaces that we consider. In Section \ref{LG}, we give a more complete description of the SYZ heuristics that guide our construction. Section \ref{adlags} is dedicated to the description of the Lagrangians that make up the objects of the Fukaya categories that we describe; Section \ref{ainfs} describes Floer complexes that will be used to define the $\Hom$ spaces in these categories. In Section \ref{fukcatsec} we define the Fukaya categories. Section \ref{lbcx} describes the categories which will make up the complex side of the mirror symmetry correspondence.  Sections \ref{mscalcs} and \ref{cmpct} are dedicated to the proof of Theorems \ref{picequiv} and \ref{openfunctors}.   In the final chapter of the main text, Section \ref{diagsec}, we relate all the categories via Theorem \ref{diagram}. Appendix \ref{appendix} sketches how the mirror correspondence could be expanded to include some torsion sheaves.

The reader who wishes to understand the flavor of the mirror symmetry equivalences is advised to read Section \ref{LG} for intuition and Sections  \ref{mscalcs} and \ref{cmpct} to see the mirror symmetry correspondence. Example \ref{HS} treats mirror symmetry for the Hopf surface, and Example \ref{example1} is a simple example of how calculations in the wrapped category reflect the behavior of analytic line bundles on the surfaces; these two examples capture most aspects of the mirror phenomena.

\subsection{Acknowledgements} This paper was completed as part of the author's Stanford University doctoral thesis supervised by Denis Auroux, and she would like first and foremost to thank him for his guidance, support, and generosity, as well as the many ideas he contributed to this work. She was co-advised by Ravi Vakil, who provided invaluable insight. She would also like to thank Sheel Ganatra, Fran\c{c}ois Greer, and Andrew Hanlon for helpful conversations and suggestions. The author was supported by NSF grants 1147470 and 2002183, by the Stanford DARE and EDGE fellowships, and by the Simons Foundation grant ``Homological Mirror Symmetry and Applications."

\section{Classification of candidate $S_\tau\left(\mathbf{A}\right)$}\label{classi}
\subsection{Logarithmic transformations}
Let $\left(m,k\right) \in \Z^2$ be primitive with $m > 0$. The logarithmic transformation with parameters $\left(m,k\right)$  replaces a smooth fiber $E_b$ on an elliptically fibered surface $S \to B$ by a fiber with multiplicity $m$. Around a fiber $E_b \subset S$, there exists a neighborhood $U$ which we may write as $U=\Delta_v \times \C^*/ \left(\left(v',w\right) \sim \left(v', q\left(v'\right)w\right)\right)$ where $\Delta_v$ is a disc in $\C$ centered around zero and $0<|q\left(v'\right)|<1$. We base change by writing $v' = v^m$ and write 
\[ U'= U_{v}/\left(\left(v,w\right) \sim \left(e^{2 \pi i/m} v, e^{-2 \pi i k/m} w\right)\right). \] The fibration $U' \to \Delta_v$  has a multiple fiber $E'_b$ at 0. There is a biholomorphic map  \[U' \setminus E_b' \to U \setminus E_b \]
given by 
\begin{equation}\label{coc} \left(v,w\right) \mapsto \left(v^m, v^{k} w\right)\end{equation}
and gluing in $U'$ via this map produces a new elliptic surface $S'$ with the property that there is a biholomorphism $S \setminus E_b \to S' \setminus E_b$. Note that we allow the case where $m=1$. For an elliptic surface $S \to B$ with a multiple fiber $E_b$, there exists a unique $k\, \mod\, m$ for which we can perform an inverse logarithmic transformation to produce a new surface $S'$ with no multiple fiber over $b$.

\begin{remark} If $k \equiv k' \, \mod\, m$, then the surfaces obtained by performing logarithmic transformations with parameters $\left(m,k\right)$ and $\left(m,k'\right)$ are locally biholomorphic near the log-transformed fiber but will in general not be globally biholomorphic. \end{remark}

Let  $e^{2\pi i \tau}= q \in \C^*$ with $|q| <1$, and let $E_\tau$ denote the elliptic curve with modular parameter $\tau$. Let 
\begin{align*} D_\tau\left(m,k\right) &=  \C \times E_\tau/ \left(\left(v,w\right) \sim \left(e^{2 \pi i/m} v, e^{-2 \pi i k/m} w\right)\right);        \end{align*}             
when we write $D_\tau\left(m,k\right)$ we will always mean this complex space along with this choice of coordinates. Let $i_{m,k}: \C^* \times E_\tau \to D_\tau\left(m,k\right)$ denote the inclusion. In coordinates $i_{m,k}$ is given by \[\left(z,x\right) \mapsto \left(z^{1/m}, z^{-k/m} x\right).\] In writing this map, and similar maps throughout this section, we make a consistent choice of the $m$-th root in both factors and then note that the map is invariant under such a consistent choice.

 \subsection{Constructing $S_\tau\left(\mathbf{A}\right)$}
Consider the surface $S_\tau\left(\mathbf{A}\right)$ obtained by performing log transforms with parameters $\left(m_0, k_0\right)$ and $\left(m_\infty,k_\infty\right)$ to $\P^1 \times E_\tau$ over $0$ and $\infty$ respectively. (In this subsection we will frequently omit the $\tau$.) It will be useful to consider $S_\tau \left(\mathbf{A}\right)$ as the union of $D_\tau\left(m_0, k_0\right)$ and $D_\tau \left(m_\infty, k_\infty\right)$ glued over their common open set $\C^* \times E$, via the transition map specified by \eqref{coc}. The transition map $D_\tau \left(m_0, k_0\right)\setminus \{v =0\} \to D_\tau \left(m_\infty,k_\infty\right)$ is given by 
 \[ \label{transmap} \left(v,w\right) \mapsto \left(v^{-m_0/m_\infty}, v^{k_0+\left(m_0k_\infty/m_\infty\right)} w\right). \] By construction, each logarithmically transformed surface is elliptic and fibers over $\P^1$; denote the map $S_\tau\left(\mathbf{A}\right) \to \P^1$ by $\pi$.  We now investigate the properties of the surfaces $S(\A)$.

\begin{remark}\label{dualizing} Each $S\left(\mathbf{A}\right)$ has Kodaira dimension $-\infty$. The canonical bundle is given by
\[ \mathcal{K}_{S\left(\mathbf{A}\right)}= \mathcal{O}_{S\left(\mathbf{A}\right)}\left(-\left(\frac{1}{m_0}+\frac{1}{m_\infty}\right)[F]\right)\]
where $[F]$ is the class of the smooth fiber (see \cite{BPV}). In other words, the reduced preimage $D_{\mathbf{A}}= \pi\inv\left(\{0\} \cup \{\infty\}\right)$ satisfies $[D_{\mathbf{A}}] \in |-K_{S\left(\mathbf{A}\right)}|$. \end{remark}

\begin{remark}  The surfaces of the form $S\left(\left(m, k\right),\left(m, -k\right)\right)$ are quotients of $\P^1 \times E$ by $\Z/ m\Z$ and as such are algebraic. In this work we do not prove mirror symmetry results for these algebraic surfaces, but the predicted mirror spaces do agree with the mirror spaces we obtain. \end{remark}

\begin{remark}The map $S\left(\mathbf{A}\right) \to \P^1$ is more properly thought of as a map to the weighted projective space $\P^1\left(m_0,m_\infty\right)$. \end{remark}

\begin{prop} \label{fundgpprop} For each $\mathbf{A}$,
\[ \pi_1\left(S(\A)\right) = \mathrm{coker} \begin{pmatrix}  m_0 & - m_\infty & 0 \\   k_0 & k_\infty & 0 \\ 0 & 0 & 0\end{pmatrix} : \Z^3 \to \Z^3. \] \end{prop}
\begin{proof} Denote the matrix in the statement of the proposition by $M$. We apply the Seifert-Van Kampen theorem to the cover of $S(\A)$ by $D\left(m_0,k_0\right)$ and $D\left(m_\infty, k_{\infty}\right)$ glued over  $\C^* \times E$. Choose the following basis for $\pi_1\left(\C^* \times E\right) = \Z^3$: 
\begin{align} \left(1,0,0\right) &= \{ \left(e^{2 \pi i t},1\right) \mid t \in [0, 1]\};\\
		     \left(0,1,0\right) &= \{\left(1,e^{2 \pi i t}\right) \mid t \in [0, 1] \};\\	
		     \left(0,0,1\right) & = \{ \left(1, 1+t\left(q-1\right)\right) \mid t \in [0,1] \}. \end{align}
The inclusion $\C^* \times E \into D\left(m_0,k_0\right)$ amounts topologically to attaching a two-cell to the class $\left(m_0,k_0,0\right)$; the inclusion $\C^* \times E \into D\left(m_\infty,k_\infty\right)$ attaches a two-cell to the class $\left(-m_{\infty},k_{\infty},0\right)$ (the negative sign originates from the change in coordinates $z \mapsto z\inv$). Thus we may write
\[ \pi_1\left(S(\A)\right) = \Z^3 /\left(m_0,k_0,0\right) \star_{\Z^3} \Z^3 /\left(-m_\infty,k_\infty,0\right) = \coker M. \]

\end{proof}

\begin{prop} For $\A$ with 
\[ \det \begin{pmatrix} m_0 & - m_\infty \\ k_0 & k_\infty \end{pmatrix} = n \neq 0, \]
the universal cover of $S_{\tau}(\A)$ is $\C^2 \setminus \{0\}$. The fundamental group $\pi_1\left(S_{\tau}(\A)\right)$, written as a quotient of $\Z^3$, acts freely on the universal cover by 
\[ \left(a,b,c\right) \cdot \left(z_1, z_2\right) = \left(\exp\left(\frac{2 \pi i\left(k_{\infty} a + m_{\infty} b+m_\infty c \tau \right)}{n}\right) z_1,  \exp\left(\frac{2 \pi i \left(-k_0 a + m_0 b+c m_0\tau\right)}{n}\right) z_2 \right) , \] 
and $S_{\tau}(\A)$ is the quotient of $\C^2 \setminus \{0\}$ by this action.
\end{prop} 
\begin{proof}Note that the kernel of the map $\rho: \Z^3 \to \left(\C^*\right)^2 \subset \GL_2\left(\C\right)$ given by 
\[ \rho\left(a,b,c\right) = \left(\exp\left(\frac{2 \pi i\left(k_{\infty} a + m_{\infty} b+m_\infty c \tau \right)}{n}\right),  \exp\left(\frac{2 \pi i \left( -k_0 a + m_0 b+ c m_0 \tau\right)}{n}\right) \right)\]
is equal to the image of the map $M: \Z^3 \to \Z^3$. Moreover, 
\[\im\left(\rho\right) \cap \left( \{1 \} \times \C^* \cup \C^* \times \{1\}\right) = \left(1,1\right). \]
Thus $\rho$  descends to the quotient to give a free action of $\coker M$ on $\C^2 \setminus \{0\}$.

Let $S_0 = \{z_2 \neq 0\}$ and $S_{\infty}= \{z_1 \neq 0\}$. Define a map $f: S_0 \to D_{\tau }\left(m_0,k_0\right)$ by
\[ f\left(z_1, z_2\right) = \left(z_1 z_2^{-m_{\infty}/m_0}, z_2^{n/m_0}\right). \]
We first verify that this map is independent of the choice of branch cut of the $m_0$th root. Indeed, note that $k_0 m_{\infty} \equiv n\,\,\, \mod \, m_0$, so
\begin{align*}& \left(\exp\left(\frac{-2 \pi i m_{\infty}}{m_0}\right) z_1 z_2^{-m_{\infty}/m_0}, \exp\left(\frac{2 \pi i n}{m_0}\right) z_2^{n/m_0}\right)\\
& =\left(\exp\left(\frac{-2 \pi i m_{\infty}}{m_0}\right) z_1 z_2^{-m_{\infty}/m_0},  \exp\left(\frac{2 \pi i k_0 m_\infty}{m_0}\right) z_2^{n/m_0}\right)\\
& \sim \left(z_1 z_2^{-m_{\infty}/m_0}, z_2^{n/m_0}\right). \end{align*}
We also note that
\begin{align*}& f\left(\left(a,b,c\right) \cdot \left(z_1,z_2\right)\right)\\ 
&= \left(\exp \left( \frac{2 \pi i a }{m_0}\right) z_1 z_2^{-m_{\infty}/m_0},  \exp^{2 \pi i \tau c n} \exp \left(\frac{-2 \pi i  k_0 a}{m_0} \right) z_2^{n/m_0}  \right); \\
& \sim f\left(z_1,z_2\right);
\end{align*}
so the map is well-defined. $f$ has an inverse given by
\begin{equation} \label{coc2} f^{\inv} \left(v, w\right) = \left(v w^{m_{\infty}/n}, w^{m_0/n}\right) \end{equation}
which one can also check is well-defined. We similarly define a map $g: S_1 \to D_{\tau}\left(m_\infty,k_\infty\right)$ by 
\[ g\left(z_1, z_2\right) = \left(z_1^{-m_{0}/m_\infty} z_2, z_1^{n/m_\infty}\right) \]
with inverse
\begin{equation} \label{coc3} g\inv\left(v,w\right) = \left(w^{m_{\infty}/n}, v w^{m_0/n}\right). \end{equation}
It remains to show that the two maps agree over $S_0 \cap S_{\infty}$, i.e., that $ f\inv \circ i_{m_0,k_0}\left(z,x\right)= g\inv \circ i_{m_{\infty},k_{\infty}}\left(z\inv,x\right)$. We check that
\begin{align} \label{map2} f\inv \circ \left(i_{m_0,k_0}\right)\left(z,x\right)&= f\inv\left(z^{1/m_0}, z^{-k_0/m_0} x\right) =\left(z^{k_{\infty}/n} x^{m_{\infty}/n}, z^{-k_0/n} x^{m_0/n}\right) \\
\label{map3} g\inv \circ \left(i_{m_\infty,k_\infty}\right)\left(z\inv,x\right)&= g\inv\left(z^{-1/m_{\infty}}, z^{k_{\infty}/m_\infty} x\right)  =\left(z^{k_{\infty}/n} x^{m_{\infty}/n}, z^{-k_0/n} x^{m_0/n}\right).
 \end{align}
\end{proof}

Denote  $f\inv \circ \left(i_{m_0,k_0}\right): \C^* \times E_{\tau} \to S_{\tau}(\A)$ by $i_{\A}$.

\begin{remark}[Terminology] Any surface that is the quotient of $\C^2 \setminus \{0\}$ by a free action of a discrete group $G$ is called a Hopf surface. A \emph{primary} Hopf surface is one where $G = \Z$; topologically these surfaces are all $S^1 \times S^3$. A \emph{secondary} Hopf surface is one where $G$ is not $\Z$; equivalently, one which is not topologically $S^1 \times S^3$; equivalently, one with torsion in its fundamental group. For more background, see \cite{KatoTo}.  Thus all the surfaces we consider are Hopf surfaces which are also elliptic. \end{remark}

\begin{remark} In this notation, the classical Hopf surface $S$ is $S(\A)$ where $\A = \left(\left(1,0\right),\left(1,1\right)\right)$. \end{remark}

\section{Landau-Ginzburg models}\label{LG}
In this section, we explain in more detail the idea that lead us to define the Fukaya category based on the philosophy of performing mirror symmetry ``relative to an anti-canonical divisor," given in work of Auroux (\hspace{1sp}\cite{AuMSTD}). Although these heuristics provide the motivation for the definition of the category, they do not enter into the proofs of any of the theorems in this paper.

Let $\pi: S\left(\mathbf{A}\right) \to \P^1$ denote the map which furnishes the fibration. Let $D_0$ denote $\pi\inv \left(\{0\}\right)$ with its reduced scheme structure and  $D_\infty=\pi\inv\left( \{\infty\}\right)$ with its reduced scheme structure, so $D_{\mathbf{A}} = D_0 \sqcup D_\infty$. By design, $S_{\mathbf{A}} \setminus D_{\mathbf{A}} = \C^* \times E$, and $D_{\mathbf{A}}$ is anti-canonical. 
 The standard symplectic form on $\C^* \times E$ is compatible with the complex structure on $S_{\mathbf{A}} \setminus D_{\mathbf{A}} = \C^* \times E$; moreover, $\C^* \times E$ admits an SYZ fibration $\C^* \times E \to \R \times S^1$ given by
\[ \left(z, x\right) \mapsto \left(\log |z|, \log |x|\right) \in \R \times \left(\R /|2 \pi i \tau| \Z\right). \]
Each fiber $L_{|z|=e^t, |x|=e^s}=L_{t,s}$ bounds no holomorphic discs in $\C^* \times E$. Guided by the hope that we can construct a mirror space to $S\left(\mathbf{A}\right)$ by correcting the mirror to $S\left(\mathbf{A}\right)\setminus D_{\mathbf{A}}$, we define a SYZ mirror $Y$ consisting of the dual fibration over $\C^* \times E$. This consists of the points
\[ Y = \{ \left( L_{t,s}, \nabla\right) \} \]
where $\nabla$ is a unitary flat connection up to gauge equivalence on the rank-one trivial bundle on $L_{t,s}$. Fixing the symplectic structure and complex structure on $Y$ described in \cite{AuMSTD}, one obtains that 
\[ Y = \R \times S^1 \times S^1 \times S^1. \]
The coordinates on $Y$ are $\left(t,s, \phi_t, \phi_s\right)$; the $\phi_t$ and $\phi_s$ encode the holonomy of $\nabla$ around the loops  $\gamma_1=\{ e^{i \theta_t} e^{t}, e^s  \mid \theta_t \in \left[0,2\pi\right)  \}$ and $\gamma_2= \{ e^t,  e^{i \theta_s} e^{s}\mid \theta_s \in \left[0,2\pi\right) \}$, respectively. The symplectic form on $Y$ is given by $\omega=d t \wedge d \phi_t + d s \wedge d \phi_s$, and the complex structure $J$ is given by $J\left(\p/\p t\right)=\p/\p \phi_t$, $J\left(\p/\p s\right)=\p/\p \phi_s$.

Note that each $L_{t, s}$ is a totally real submanifold of $S\left(\mathbf{A}\right)$ which bounds the following holomorphic discs passing through $D_\A$:
\begin{align*} u_{0}^{t,s}: \D \to D\left(m_0, k_0\right) \subset S\left(\mathbf{A}\right); \\
		      u_{\infty}^{t,s}: \D \to D \left(m_\infty,k_\infty\right) \subset S\left(\mathbf{A}\right).\end{align*}
$u_0^{t,s}$  is given by 
\[ z \mapsto \left(e^{t/m_0} z, e^{s-k_0 t/m_0}\right); \] $u_{\infty}^{t,s}$ is given by the analogous map to $D\left(m_\infty,k_\infty\right)$. 
The boundaries of the discs give elements $[\p u_{0}], [\p u_{\infty}] \in \pi_1\left(L\right) \subset \pi_1\left(\C^* \times E\right)$. In the basis for $ \pi_1\left(\C^* \times E\right)$ of Proposition \ref{fundgpprop},  $[\p u_0] = \left(m_0, k_0,0\right)$ and $[\p u_{\infty}] = \left(-m_\infty, k_{\infty},0\right)$ (one notes that these discs represent the cells we attached in calculating the fundamental group). 

From a dimension counting argument similar to that of Lemma 3.1 of \cite{AuMSTD}, we expect that these are the only rigid holomorphic discs passing through the point $\left(e^t,e^s\right)$ on $L_{t,s}$. If $S\left(\mathbf{A}\right)$ were K\"{a}hler and the symplectic form $\omega$ on $S\left(\mathbf{A}\right)\setminus D_{\mathbf{A}}$ were inherited from a symplectic form $\omega$ on $S\left(\mathbf{A}\right)$, we would define a superpotential recording each disc:
\[ W\left(L, \nabla\right) = e^{-\omega\left(u_0^{t,s}[\D]\right)} \hol_{\nabla}  [\left(m_0, k_0\right)] + e^{-\omega\left(u_\infty^{t,s} [\D]\right)} \hol_{\nabla} [\left(-m_\infty,k_\infty\right)].  \]
We would then define a Fukaya-Seidel category associated to the Landau Ginzburg model $\left(Y,W\right)$ with objects the Lagrangians which fiber over paths asymptotic to the positive real axis under $W$. This would give conditions on the Lagrangians at infinity; in the region where each monomial corresponding to a term recording a disc $u$ in the superpotential dominated, the Lagrangians would be restrained to the region where $\hol_{\nabla} \left([\p u]\right)$ is close to the identity.

Of course we cannot define such a $W$ because there is no symplectic form $\omega$! Thus the failure of $S\left(\mathbf{A}\right)$ to be a symplectic manifold is reflected in the failure of the superpotential $W$ to be well-defined, and in particular in the failure of each monomial term to have a well-defined norm. However, the standard symplectic form on the universal cover endows $S(\A)$ with a natural conformal symplectic structure -- that is, $S(\A)$ is locally a symplectic manifold where the symplectic form is defined only up to multiplication by a constant -- and the tori $L_{t,s}$ are Lagrangian for the conformal symplectic form.  We can use this structure to define a ``relative energy" of the two discs via the universal cover; this will reflect the intuition that the $e^{-\omega\left(u_0[\D]\right)} \hol_{\nabla}  [\left(m_0, k_0\right)]$ term should dominate as $t \to -\infty$, and the   $e^{-\omega\left(u_0[\D]\right)} \hol_{\nabla}  [\left(-m_\infty, k_\infty\right)]$ term should dominate as $t \to \infty$.

Let $\tilde{u}^{t,s}_0$ and $\tilde{u}^{t,s}_\infty$ be lifts of the discs $u^{t,s}_0$ and $u^{t,s}_{\infty}$ to $\C^2 \setminus \{0\}$ so that $\tilde{u}^{t,s}_0\left(1\right)=\tilde{u}^{t,s}_{\infty}\left(1\right)$. After applying the change of coordinates of Equations \eqref{coc2}, \eqref{coc3}, \eqref{map2}, \eqref{map3}, we obtain: 
\begin{align*} \tilde{u}^{t,s}_0 \left(z\right)&= \left(\exp \left( \frac{2 \pi i \tau\ell_1 m_{\infty}}{n}\right)  \exp \left(\frac{2 \pi i \ell_2 m_{\infty}}{n}+\frac{t k_{\infty}+s m_{\infty}}{n}\right) z,\right. \\ 
& \quad\quad\quad \left. \exp\left( \frac{2 \pi i \tau \ell_1 m_{0}}{n}\right) \exp \left(\frac{2 \pi i \ell_2 m_{0}}{n}+\frac{-t k_{0}+s m_{0}}{n}\right)\right) ;  \\
 \tilde{u}^{t,s}_\infty \left(z\right)&=  \left(\exp \left(\frac{2 \pi i \tau \ell_1 m_{\infty}}{n}\right)\exp \left(\frac{2 \pi i \ell_2 m_{\infty}}{n}+\frac{t k_{\infty}+s m_{\infty}}{n}\right) ,\right.\\
 & \quad \quad \quad  \left. \exp\left(\frac{2 \pi i \tau \ell_1 m_{0}}{n} \right) \exp \left(\frac{2 \pi i \ell_2 m_{0}}{n}+\frac{-t k_{0}+s m_{0}}{n}\right)z\right), \end{align*}
where $\ell_1, \ell_2$ are integers depending on the choice of lift. We can then calculate the energy of these lifts of the discs using the standard symplectic form 
\[ \omega_{\C^2}=\frac{i}{2} \left(d z_1 \wedge d \overline{z_1}+  d z_2 \wedge d \overline{z_2}\right) \]
on $\C^2 \setminus \{0\}$, which is the form which provides the conformal symplectic structure on $S(\A)$: we obtain

\begin{align*} A_{0}\left(t,s\right) &= \int_{\D}\left( \tilde{u}_{0}^{t,s}\right)^* \omega_{\C^2} = \pi \exp\left(\frac{2 \pi i \tau \ell_1 m_{\infty}}{n} \right) \exp\left(\frac{2\left(t k_{\infty} + s m_{\infty}\right)}{n} \right); \\
 A_{\infty}\left(t,s\right) &= \int_{\D}\left( \tilde{u}_{\infty}^{t,s}\right)^* \omega_{\C^2} =\pi \exp\left(\frac{2 \pi i \tau \ell_1 m_{0}}{n}\right) \exp \left(\frac{2\left(-t k_{0}+s m_{0}\right)}{n}\right). \end{align*}
We observe that the ratio
\[ \frac{A_0 \left(t,s\right)^{m_0}}{A_{\infty}\left(t,s\right)^{m_{\infty}}} = \exp \left(2 t\right); \]
is independent of the choice of lift and of $s$. Thus as $t \to -\infty$, one can say that $A_0\left(t,s\right)$ is becoming small relative to $A_{\infty}$; as $t \to \infty$,  $A_{\infty}\left(t,s\right)$ is becoming small relative to $A_{0}$. 

We thus define a Fukaya category, denoted $\F\left(\mathbf{A}\right)$, with objects Lagrangians $L$ in $Y$ so that for $t << 0$, \begin{equation}\label{condac} \arg \left(e^{2 \pi i \left(m_0 \phi_t + k_0 \phi_s\right)}\right)  = 0 \in \R/2 \pi \Z  \end{equation} when restricted to $L$ and for $t >> 0$, \begin{equation}\label{condbd} \arg\left(e^{2 \pi i \left(-m_{\infty} \phi_t + k_\infty \phi_s\right)}\right)  = 0 \in \R/2 \pi \Z  \end{equation} when restricted to $L$, with the expectation that objects of such a category will correspond to coherent sheaves on $S\left(\mathbf{A}\right)$.  A formal definition of this $A_\infty$ category will be given in Section \ref{adlags}.

\subsection{Stop removal and (partially-)wrapped categories}
There are restriction maps
\begin{align*} \Coh_{\an} S\left(\mathbf{A}\right) \to \Coh_{\an} \left( S\left(\mathbf{A}\right)  \setminus D_{\infty} \right)= \Coh_{\an} \left(D\left(m_0, k_0\right)\right)\\
	              \Coh_{\an} S\left(\mathbf{A}\right) \to  \Coh_{\an} \left(S\left(\mathbf{A}\right) \setminus D_{0}\right) = \Coh_{\an} D\left(m_\infty, k_\infty\right); \end{align*}
Following the ``stop removal" philosophy of \cite{SyPW} or \cite{GPS1}, we expect that the first map corresponds to a localization map $\F\left(\mathbf{A}\right) \to \F\left(m_0, k_0\right)$, where $\F\left(m_0, k_0\right)$ is a partially wrapped category with objects which are Lagrangians which satisfy $m_0 \phi_t + k_0 \phi_s \to 0$ when $t << 0$; and similarly for the second map after an interchange of the ends $Y$. Roughly speaking, $\F\left(m_0, k_0\right)$ should be the partially wrapped category associated to a superpotential which counts discs that pass through $D_{\mathbf{A}} \setminus D_{\infty} = D_0$. Similarly, we expect that the inclusion $ \C^* \times E \subset D\left(m_0, k_0\right)$ corresponds to a localization map  $\F\left(m_0, k_0\right) \to \F_{\W}$ from a partially wrapped category to a fully wrapped category.  We will define such localizations and show that there is a diagram of the form

 \begin{equation}\label{gluingfuk}
\begin{tikzcd}
& \F\left(m_0, k_0\right) \arrow[dr] & \\
\F\left(\mathbf{A}\right) \arrow[ur]\arrow[dr] && \F_{\W} \\ 
& \F\left(m_\infty,k_\infty\right) \arrow[ur]
\end{tikzcd} \end{equation}

These maps differ from the  stop removal maps found in \cite{GPS2} in that they are not essentially surjective; we will show that the boundary conditions impose conditions on the homology classes of the Lagrangians in each category.

\section{Admissible Lagrangians and perturbing Hamiltonians}\label{adlags}
Let $Y=\left(\R \times (S^1)^3, d t \wedge d \phi_t+ d s \wedge d \phi_s\right)$, with coordinates $\left(t,s, \phi_t, \phi_s\right)$.  Let $\pi: \R^4 \to Y$ denote the covering map. The map $Y \to \R \times S^1$ given by the projection to the first two coordinates provides the SYZ fibration seen in Section \ref{LG} on $\C^* \times E$; we refer to the $\phi_t$ and $\phi_s$ coordinates as the angular coordinates. The term \emph{Lagrangian section} refers to a section of this SYZ fibration. We will define three Fukaya categories on $Y$ with increasingly relaxed boundary conditions on their objects. These boundary conditions will control the behavior of the angular coordinates as $|t|$ grows large.

Fix the following notation for the two ends of $Y:$ for $T >0$ let
\[ Y^{-}_T = \{ \left(t,s, \phi_t, \phi_s\right) \in Y \mid t < -T \} ;\]
\[ Y^{+}_T = \{ \left(t,s, \phi_t, \phi_s\right) \in Y \mid t > T \}. \]
 Let $Y_T = Y^+_T \cup Y^-_T$. For a subset $N \subset Y$,  let $N_T^{\pm}= N \cap Y^{\pm}_T$ and let $N_T = N \cap Y_T$.  Let $t,s,\phi_t,\phi_s: Y \to \R$ denote the corresponding coordinate functions.
 
Define basis elements of $\R^4$: $\e_t = \left(1,0,0,0\right), \e_s = \left(0,1,0,0\right), \e_{\phi_t}=\left(0,0,1,0\right), \e_{\phi_s}=\left(0,0,0,1\right)$. Within this section, let $d: \R^4 \times \R^4 \to \R$ denote the Euclidean distance function. 

\subsection{Admissible planar Lagrangians}

\begin{defn}[Admissible Lagrangians] Let $L$ be a Lagrangian submanifold of $Y$ with following properties: \begin{itemize}
 \item A lift of $L$ to $\R^4$ is a section of the fibration $\R^4 \to \R^2$ given by $\left(t, \tilde{s}, \wt{\phi_t}, \wt{\phi_s}\right) \to \left(t, \tilde{s}\right)$.
 \item $\mu^1\left(L\right)=0  \in H^1\left(L,\Z\right)$.
 \item (\emph{planarity}) $L$ is planar at the ends of $Y$: that is, there exists $T > 0$ such that the lift of $L_T^+$ to $\R^4$ lies on the points of an affine plane $P^+$, and similarly for $L_T^-$.
 \end{itemize} 
Then $L$ is an \emph{admissible Lagrangian}. \end{defn}

 \begin{defn} Let $\left(m_0, k_0\right)$ and $\left(m_\infty,k_\infty\right)$ be two primitive elements in $\Z^2$ such that $m_0, m_\infty>0$, and let $\mathbf{A}$ denote the pair $\left(\left(m_0, k_0\right), \left(m_\infty,k_\infty\right)\right)$.
$L$ is a \emph{type-$\mathbf{A}$} Lagrangian if $L$ is an admissible Lagrangian such that there  exists $T > 0$ such that 
 \begin{align}\label{raycond1}  m_0 \phi_t + k_0 \phi_s &\equiv 0 \in \R/\Z \quad \textrm{on\,} L_-^T; \\
   -m_\infty \phi_t + k_\infty \phi_s &\equiv 0 \in \R/\Z \quad \textrm{on\,} L_+^T. \end{align} 
We will refer to this property -- i.e., the property that the projection of the ends of $L$ to the $\left(\phi_t, \phi_s\right)$ plane lie in a line -- as \emph{linearity}. We call the bounded region in which $L$ does not satisfy the linearity conditions the \emph{non-linear region of L}, and its complement the \emph{linear region of L}.
  \end{defn}

 Let $\mathcal{G}\left(\A\right)$ denote the set of type-$\A$ planar Lagrangians equipped with a grading (i.e., a lift of the phase map $L \to S^1$ from the complex volume form on $Y$) and a choice of  $\mathrm{Pin}$ structures as in \cite{seidelbook}, with each Lagrangian remembering the choice of $\A$. We will frequently abuse notation by suppressing the choice of grading and Pin structures, and denote an element of $\mathcal{G}$ by $L$. Let
\begin{align*} \mathcal{G}\left(m, k\right) &= \bigcup_{\left(m',k'\right)} \G\left(\left(m,k\right),\left(m',k'\right)\right) \\
 \mathcal{G}& = \bigcup_{\left(\mathbf{A}\right)} \Gabcd. \end{align*}
 Thus $\mathcal{G}$ is the set of linear admissible Lagrangians.  
 
\subsection{Topology of Lagrangians}
The conditions above give control over the behavior of the Lagrangians at the ends of $Y$.  Let $L$ be an admissible planar Lagrangian of type $\mathbf{A}$. By the linearity and planarity of $L$, as well as the fact that $L$ is Lagrangian, there exists $\frac{d_-}{r_-} \in \Q$, $E_{-} \in \R^4$, $T > 0$ such that a lift of $L_{T}^-$ to $\R^4$ lies on the points of the affine plane
\[ P^- = \left\{ t \begin{pmatrix} 1 \\ 0 \\ - \frac{d_-}{r_-} \frac{k_0^2}{m_0} \\  \frac{d_-}{r_-} k_0 \end{pmatrix} + s \begin{pmatrix} 0 \\ 1 \\  \frac{d_-}{r_-} k_0 \\ - \frac{d_-}{r_-} m_0 \end{pmatrix} + E_- \Bigm\vert t,s \in \R \right\} \] 
and $L_T^+$ similarly lies on an affine plane $P^+$. 
By considering the set $L \cap \{t = t_0\}$ for any fixed $t_0$, we obtain an element $[\gamma_L] \in H_1\left(T^3\right)$. Using the natural basis for $H_1\left(T^3\right)$ from the coordinates $\left(s,\phi_t, \phi_s\right)$, for a Lagrangian with $P^+$ and $P^-$ of the above form and $\frac{d_\pm}{r_\pm}$ reduced with $r_\pm > 0$, we obtain that \[[\gamma_L] = \left(r_-, d_{-} k_{0}, -d_{-} m_{0}\right) =  \left(r_+, -d_{+} k_{\infty}, -d_{+} m_{\infty}\right).\] We conclude that when $m_{\infty} k_0 + m_0 k_{\infty} \neq 0$,  $d_{+}=d_{-}=0$ and $r_{+}=r_-=1$. In all cases we can write $[\gamma_L]$ as $\left(r, d k_0, -d m_0\right)$. This homology class is an invariant of the Lagrangian under any perturbation .

\begin{defn} With $L$ as above, let the quantity $\frac{d}{r}$ denote the \emph{slope} of the Lagrangian, and let $\frac{m d}{r}$ denote the \emph{total slope} of the Lagrangian $L$. \end{defn}

\begin{defn}Let $L, L' \in \G$ be in correct position. Fix lifts $\tilde{L}$ and $\tilde{L}'$ of $L$ and $\tilde{L}'$;  define a subgroup $\Gamma\left(L, L'\right) \subset \Z^3$ by 
\[ \Gamma\left(L,L'\right)=\left\{\hat{n} \in \Z^3 \mid \pi\left(\tilde{L} \cap \tilde{L}'\right) = \pi\left(\tilde{L} \cap\left( \tilde{L}'-\hat{n}\right)\right)\right\}.  \] 
Note that  $\Gamma\left(L,L'\right)$ is a topological invariant of $L$ and $L'$. 
  \end{defn}
  
\begin{defn}[Interpolating Functions]\label{intfun}
We define the following functions, which will be necessary in the subsequent lemma and throughout. \begin{enumerate}[(a)]
\item For $a$ and $b$ in $\R$, choose monotonic functions $\rho_{a}^b: t \to \R$ such that 
\begin{align*}& \rho_{a}^b = \begin{cases} a & t \leq -1 \\
					b & t \geq 1, \end{cases}  ; \\ 
         & \left|\frac{d \rho_a^b\left(t\right)}{dt}\right| \leq |b-a|; \\
          & \rho_{a}^{b}+\rho_{a'}^{b'}=\rho_{a+a'}^{b+b'} \\
& \rho_{a}^b\left(0\right)=\frac{a+b}{2}; \end{align*}
and $\rho_{a}^b$ is linear with slope $b-a$ on $\left(-1/4, 1/4\right)$;

\item Let $\alpha_+, \alpha_-, \epsilon_+, \epsilon_- > 0$ be such that $\alpha_- + \epsilon_- = \alpha_+ + \epsilon_+$. Let $\sigma_{\alpha_-, \epsilon_-}^{\alpha_+, \epsilon_+}: \R \to \R$ be a smooth function with 
 \[  \sigma_{\alpha_-, \epsilon_-}^{\alpha_+, \epsilon_+}\left(t\right) = \begin{cases} \left(\alpha_++\epsilon_+\right) t^2 & t \in [-1/2, 1/2]\\
													 \alpha_\pm t^2 + \epsilon_\pm |t| & \pm t \in \left[1, \infty\right) \end{cases}; \]
such that											 
\[  \left| \frac{d \sigma_{\alpha_-, \epsilon_-}^{\alpha_+, \epsilon_+}\left(t\right)}{dt} \right| > \frac{2 \alpha_+ t + \epsilon_+}{2} \]
for $t \geq 1/2$, and similarly for $t < -1/2$; and 
\[ \sigma_{\alpha_-, \epsilon_-}^{\alpha_+, \epsilon_+}\left(t\right) + \sigma_{\alpha_-', \epsilon_-'}^{\alpha_+', \epsilon_+'}\left(t\right) =  \sigma_{\alpha_-+\alpha_-', \epsilon_-+\epsilon_-'}^{\alpha_++\alpha_+', \epsilon_++\epsilon_+'}\left(t\right). \]

	
\item Let $\chi: \R \to [0,1]$ be a symmetric function with 
\[ \chi\left(t\right) = \begin{cases} 1 & |t| < \frac{1}{2} \\						
					0 & |t| > 2 \end{cases}. \]
					
so that for $t > 0$, $0 \leq \chi\left(t\right) \leq 1$ and such that $ |d \chi(t)/dt| <1.$	For $R \in \R^*$, let $\chi_R\left(t\right)=\chi\left(t/R\right)$.		\end{enumerate} \end{defn}

\begin{defn}[Cylic covers]\label{cycliccovers} For $r \in \Z$, let $Y_r \to Y$ denote the $r$-fold cover where the deck transformation group $\Z/r \Z$ acts as $\left(t,s, \phi_{t}, \phi_{s}\right) \to \left(t, s +1, \phi_t, \phi_s\right)$. Let \begin{align*} &C_r = \R \times \R/r \Z; \\ &S^1_r=\R/r \Z, \end{align*} so $S^1_r$ includes into $C_r$ and $Y_r$ fibers over $C_r$. \end{defn}

\begin{lemma} \label{formofh} Let $L \in \G\left(\A\right)$ with $\gamma_{L}=[r, d k, -d m]$,  let $\tilde{L}$ be a lift of $L$ to $\R^4 = T^* \R^2$, and let $h: \R^2 \to \R$ be a function so that $\tilde{L}$ is the graph of the differential $d h$ in $\R^4 = T^* \R^2$. If $d=0$, then we may write $h$ as a function of the following form:
\begin{equation} \label{diff2} h =  a s  + t \rho_{b^-}^{b^+} \left(t\right) + \rho_{0}^c\left(t\right) + G\left(t,\tilde{s}\right)\end{equation}
where $a, b^-, b^+, c \in \R$, and $G: \R^4 \to \R$ is a function which vanishes for $|t| >> 0$ with the property that $G\left(t,\tilde{s}+1\right)=G\left(t,s\right)$.

If $d \neq 0$, we can write $h$ in the following form:
\begin{equation} \label{diff1} h = \frac{-d m \left(s- \left(k/m\right)t \right)^2}{2 r} + a\left(s+ \left(k/m\right) t\right)+ t \rho_{b^-}^{b^+} \left(t\right) + \rho_0^c\left(t\right) + G\left(t,\tilde{s}\right) \end{equation}
where $b^-, b^+ \in \frac{1}{m} \Z$, $c \in \R$, and $G: \R^4 \to \R$ is a function which vanishes for $|t| >> 0$ with the property that $G\left(t,\tilde{s}+r\right)=G\left(t,s\right)$. \end{lemma}
\begin{proof} Assume first that $d=0$. Then necessarily $r=1$, and $L$ is a Lagrangian section. 
Write
\[ \tilde{L} = \{ t,s, g_1\left(t,s\right), g_2\left(t,s\right) \mid \left(t,s\right) \in \R \times \R \}. \] We have seen that for $|t| >> 0$, $g_1$ and $g_2$ take constant values; let $g_1 = b^-$ for $t << 0$, $g_1= b^+$ for $t>> 0$, and let $g_2 = a^-$ for $t << 0$, $g_2= a^+$ for $t>>0$. We can assume that $h= b^- t +a^- s $ for $t<< 0$.

Consider the function $I: \R \to \R$ defined by 
\[ I\left(t\right) = \int_{0}^1 g_2 \left(t, s\right) \, ds. \]
Note that  $I\left(t\right)=a^-$ for $t << 0$, and  $I\left(t\right)=a^+$ for $t >> 0$. 
Also note that 
\begin{align} \frac{d I}{d t} &=  \int_{0}^1  \frac{\p g_2}{\p t}\left(t, s\right) \, ds \\
			   & = \int_{0}^1 \frac{\p g_1}{\p s}\left(t,s\right)  \, d s \\
			   & =0. \end{align}
Thus $I\left(t\right)$ is constant with value $a = a^- = a^+$.
Let  \[ c = \int_{-\infty}^\infty  g_1\left(t,0\right) - \frac{d}{d t} \left(t \rho_{b^-}^{b^+} \left(t\right)\right) dt;  \] 
this quantity is well-defined since the integrand is bounded and vanishes outside of a compact set. 
Now consider the real-valued function $G$ defined by
\[ G\left(t,s\right) = h\left(t,s\right) -  \left(a s+ t \rho_{b^-}^{b^+} \left(t\right) +\rho_{0}^c\left(t\right)\right). \]
We wish to show that $G$ is a compactly supported function on $\R \times \R/\Z$. By construction, $G$ vanishes identically for $t<<0$. For $t >> 0$, note that 
\begin{align*} G\left(t,s\right) = \int_{-\infty}^t \left(g_1\left(\tau, 0\right) - \frac{d}{d \tau} \left( \tau \rho_{b^-}^{b^+} \left(\tau\right)+\rho_0^c\left(\tau\right) \right)  \right)\, d \tau + \int_{0}^s \left( g_2\left(t, \sigma\right) \, - a \right)\,d \sigma  = 0. \end{align*} Finally note that
\[ G\left(t,s+1\right) -G\left(t,s\right)=  -a + \int_{s}^{s+1} g_2\left(t, \sigma\right) \, d \sigma  = -a +I\left(t\right)= 0. \] Thus $G$ is of the desired form. 

Now assume that $\tilde{L}$ is a Lagrangian section and $d \neq 0$. Then necessarily $\left(m_0, -k_0\right)=\left(m_{\infty}, k_{\infty}\right)=\left(m,-k\right)$. We can write
\[ \tilde{L} = \{ t,s, g_1\left(t,s\right), g_2\left(t,s\right) \mid \left(t,s\right) \in \R \times \R \} \]
where for $|t| >> 0$, 
\begin{align*} g_1= \frac{d k}{r} \left(s- ({k}/{m}) t\right) + b^{\pm} \\
		      g_2 =  -\frac{d m}{r} \left(s- ({k}/{m}) t\right) + a^{\pm}. \end{align*}
We can assume that for $t << 0$, \[ h= -\frac{d m}{2 r}\left(s- ({k}/{m}) t\right)^2+ b^- t + a^- s.\] Let 
\begin{align*} h' &= h +  \frac{d m}{2 r} \left(s- ({k}/{m}) t\right)^2\\
		     g_1' &=  g_1-  \frac{d k}{r} \left(s- ({k}/{m}) t\right)\\
		    g_2'& =g_2+ \frac{d m}{r}\left(s- ({k}/{m}) t\right). \end{align*}		    
		     Then we may apply the above arguments to $h', g_1', g_2'$ to conclude again that 
\[ h'  =  a s  + t \rho_{\left(b^-\right)'}^{\left(b^+\right)'} \left(t\right) + \rho_{0}^c\left(t\right) + G\left(t,\tilde{s}\right). \]
Let $\left(b^{\pm}\right)'=b^{\pm} + a \left(k/m\right)$. The boundary conditions imply that
\[ m \left( \left(b^{\pm}\right)'+ a\left(k/ m\right)\right) + k\, a \in \Z \]
so $m \left( b^{\pm}\right)' \in \Z$. Thus we may write $h$ in the desired form:
\[ h = h' -  \frac{d m}{2 r} \left(s- \left(k/m\right) t\right)^2= -\frac{d m}{2 r}\left(s- \left(k/m\right)t \right)^2 + a\left(s-\left(k/m\right)t \right)+ t \rho_{\left(b^-\right)'}^{\left(b^+\right)'} \left(t\right) + \rho_0^c + G\left(t,\tilde{s}\right).  \]

Finally, assume that $d \neq 0$, but $r \neq 1$.  Let $L_r$ denote a lift of $L$ to $Y_r$; note that $L_r$ is a section of the fibration $Y_r \to C_r.$ Then the above methods will apply, and we can conclude that $h$ is of the desired form.
\end{proof}

\subsection{Perturbing Hamiltonians} 
For a Hamiltonian function $H: Y \to \R$, we denote the time-$1$ flow of the Hamiltonian vector field $X_H$ by $\psi_H$. We first define the class of Hamiltonians that we will use to perturb our Lagrangians.

\begin{defn}\label{pertHdef}A \emph{perturbing Hamiltonian} is a function $H:Y \to \R$ depending only on $t$ and $s$ such that there exists $T >0$ such that
\begin{align*} H\left(t, s\right) &=\alpha_{+} t^2 + \epsilon_+| t |+ f_+\left(t,s\right) \textrm{ on } Y_T^+ \\
 H\left(t,s\right) &=\alpha_{-} t^2 + \epsilon_- |t| + f_-\left(t,s\right) \textrm{ on } Y_T^- . \end{align*}
where $\alpha_{\pm} \in \Q$, $\alpha_{\pm}, \epsilon_{\pm} \geq 0$,
\[ \alpha_- + \epsilon_+ = \alpha_- + \epsilon_-, \] and $f_{\pm} : \R \times S^1 \to \R$ are smooth functions that satisfy the following constraints:
\begin{enumerate}[(a)]
\item there exists $N_{\pm} \in \N$ such that $f_{\pm}\left(t +m N_{\pm}, s\right) = f_{\pm}\left(t, s\right)$ for all $t,s \in Y_T$, for all $m \in \Z$.
\item there exists $t_0 > T$ and $\kappa^{\pm} > 0$ such that $f_{\pm}\left(t,s\right) = 0$ if $|t-t_0|< \kappa^{\pm}$. (Note that then $f\left(t,s\right)=0$ for all $t>T$ such that there exists $m \in \N$ such that $|t-t_0 + m N| < \kappa$.) For each such $f$, fix a $\kappa\left(f_\pm\right) < 1$.
\item $\|f_{\pm}\|_{C^2} <\min\left\{ \frac{\kappa\left(f_\pm\right)^2}{4},  \frac{\alpha_{\pm}}{8} \right\}.$
\end{enumerate}

 We will use $f_{\pm}\left(H\right)$, $\alpha_{\pm}\left(H\right)$, and $\epsilon_{\pm}\left(H\right)$ to mean the corresponding parameter of $H$ at infinity. We define a partial order on perturbing Hamiltonians by setting $H' > H$ if $H' > H$ outside of a compact set. 
   \end{defn}

\begin{defn}  A perturbing Hamiltonian is
 \begin{enumerate}[(a)]
 \item \emph{a type-$\A$ Hamiltonian} if $H$ satisfies 
 \begin{align*} & \alpha_{\pm}\left(H\right) = f_{\pm}\left(H\right) = 0; \\
 & 0 \leq \epsilon_{-} < \frac{1}{m_0}; \\
 & 0 \leq \epsilon_{+} < \frac{1}{m_{\infty}}. \end{align*}
\item \emph{a type-$(m,k)$ Hamiltonian} if $H$ satisfies 
\begin{align*} & \alpha_{-}\left(H\right)=f_{-}\left(H\right)=0. \\
			& 0 \leq \epsilon_{
			-}\left(H\right) < \frac{1}{m_0}.
	 \end{align*}
			\end{enumerate}
\end{defn}

\begin{remark}[Universal notation] In the next section we will define Fukaya categories $\Facbd$, $\F\left(m_0, k_0\right)$, and $\F_{\W}$, the objects of which will (roughly) correspond to Lagrangians in $\Gabcd$, $\G\left(m_0, k_0\right)$, and $\G$, respectively. We will frequently have statements that will have versions for each category. We will shorten these statements by referring to the labelling by $I$; the absence of a label implies that we are working in the $\W$ setting. 

In each case, we refer to the ends of $Y$ where the boundary conditions are relaxed as the \emph{wrapped ends}, so there are no wrapped ends in the $\A$ setting, one wrapped end in the $\left(m,k\right)$ setting, and two wrapped ends in the $\W$ setting.

 \end{remark}

Let $\mathcal{G}_P\left(I\right)$ denote the set of Lagrangians which may be written as $\psi_{H}\left(L\right)$, where $H$ is a type-$I$ Hamiltonian. Note that  $\mathcal{G}_P\left(\A\right) \subset \mathcal{G}_P\left(m_0,k_0\right) \subset \mathcal{G}_P$. We now begin investigating the properties of such perturbed Lagrangians. 

\begin{prop}[Periodicity of Lagrangians at ends] Let $L \in \G_P$. Then there exists $T >0$ and $N \in \N$ such that 
\[ L_{T}^{\pm} \pm N \hat{e}_t = L_{T\pm{N}}^{\pm}. \]
\end{prop}
\begin{proof} This follows from the planarity of the Lagrangians, the rationality of $\alpha_{\pm}$, and the periodicity of the $f_{\pm}$. \end{proof}

\begin{prop} 
\begin{enumerate}[(a)]
\item Let $H_0$ and $H_1$ be type-$\A$ Hamiltonians such that $H_0 > H_1$ and let  $L_0$ and $L_1$  be type-$\A$ Lagrangians. Then $\psi_{H_0}(L_0)$ and $\psi_{H_1}(L_1)$ are disjoint outside of a compact set. 
\item  Let $H_0$ and $H_1$ be Hamiltonians which are $\left(m,k\right)$-suitable for type-$\left(m,k\right)$ Lagrangians $L_0$ and $L_1$ such that $H_0 > H_1$. Then $\psi_{H_0}(L_0)$ and $\psi_{H_1}(L_1)$ are disjoint on $Y_T^{-}$.
\end{enumerate} \end{prop}
\begin{proof} Note that points $\left(t,s, \phi_t,\phi_s\right)$ on $\psi_{H_i}(L_i)$ with $t << 0$ satisfy \[m_0 \phi_t + k_0 \phi_s = m_0 \,\epsilon_-\left(H_i\right).\] Since  $H_0 > H_1$, \[ 0< \epsilon_-\left(H_0\right)- \epsilon_-\left(H_1\right) < \frac{1}{m_0},\] and the two Lagrangians are disjoint for $t << 0$. The same reasoning shows disjointness for $t >>0$ in the $\A$ setting.
 \end{proof}
 
We cannot ensure disjointness of perturbed Lagrangians at an end with relaxed boundary conditions; instead we ensure transversality. We first note that any Lagrangian in $\G_P$ is close to a plane determined by the parameters $\alpha_{\pm}$ and $\epsilon_{\pm}$. Let $L \in \G$, and let $H$ be a perturbing Hamiltonian. We can write $H= H' + f$, where $H'$ is quadratic at the ends, and the decomposition is unique at the ends. Then $\psi_{H'}\left(L\right)$ is planar at infinity: there exists $T >0$ such that for any lift $\psi_{H'}\left(\widetilde{L}\right)$ to $\R^4$, there exist  
affine planes $P^+, P^- \subset \R^4$ such that
\[ \psi_{H'}\left(\tilde{L}_T^\pm\right) \subset P^{\pm}. \]
Concretely, we may write 
\begin{equation}\label{planes} P^\pm= \left \{ x,  \begin{pmatrix}-\frac{d}{r}  \frac{k_0^2}{m_0} +\alpha_{\pm}(H)  &  \frac{d}{r}  k_0 \\ \frac{d}{r}  k_{0} &- \frac{d}{r} m_{0}  \end{pmatrix} x + E_i  \Bigm\vert x \in \R^2 \right\}. \end{equation}

\begin{defn} For a Lagrangian $\tilde{L}$ in $\R^4$ which is the lift of a Lagrangian $L$ in $\G_P$, we call the planes $P^+$ and $P^-$ the \emph{corresponding planes} to $\tilde{L}$ and to $L$; we will also refer to the corresponding planes of the Lagrangian $L$ itself. In the latter case, the ambiguity in the lift will be immaterial. \end{defn}

Since $H'$ and $f$ Poisson-commute, we may write
\[ \psi_{H}\left(L_{T}^{\pm}\right)= \psi_{f}\left( \psi_{H'}\left(L_{T}^\pm\right)\right). \]
The $C^2$ bounds on $f$ bound the Hausdorff distance $d_H$ between a perturbed Lagrangian $\psi_{H}\left(L\right)$ and its corresponding planes at infinity:
 \begin{equation} \label{boundfromp} d_H(P_T^\pm, \psi_{H}\left(L_{T}^\pm\right))\leq
 \left(\psi_{H'}\left(L_T^\pm\right) - \psi_{H}\left(L_{T}^\pm\right)\right) \leq  \|\psi_{f_\pm}-\Id_{Y^\pm}\|_{C^{0}} \leq \|D f_{\pm}\|_{C^0} \leq \frac{\kappa_{\pm}^2}{4}. \end{equation}

\subsection{Correct positioning of Lagrangians}
We now investigate to what extent we can ensure transversality by imposing conditions on the perturbing Hamiltonians. The following lemma is key to many of the following arguments:
\begin{lemma}\label{transint}Let $L$ and $L'$ be two admissible Hamiltonians of type $\mathbf{A}=\left(\left(m_0,k_0\right),\left(m_\infty, k_\infty\right)\right)$ and $\mathbf{A'}=\left(\left(m_0',k_0'\right),\left(m_\infty', k_\infty'\right)\right)$ and slope $n$ and $n'$ respectively. 
 \begin{enumerate}[(a)] \item If $m_0 n \neq m_0' n'$ or $k_0 n \neq k_0' n'$, then there exists a constant $K\left(L, L'\right)$ such that for any two perturbing Hamiltonians $H, H'$ such that  $\alpha_\pm\left(H\right)-\alpha_\pm\left(H'\right) > K\left(L, L'\right)$, the corresponding planes $P^\pm$ and $(P^\pm)'$ of $\psi_{H}\left(L\right)$ and $\psi_{H'}\left(L'\right)$ intersect transversely.
\item If $m_0 n=m_0' n'$ and $k_0 n = k_0' n'$ then if $H> H'$, the corresponding planes  $P^\pm$ and $(P^\pm)'$ of $\psi_{H}\left(L\right)$ and $\psi_{H'}\left(L'\right)$ intersect in a line with a constant value of $t$ if $\alpha_{\pm}(H) -\alpha'_\pm(H') > 0$. In this case set $K\left(L,L'\right)=0$.
\end{enumerate}\end{lemma}
\begin{proof}Let $P$ and $P'$ denote the corresponding planes to the perturbed Lagrangians at the $\pm t<<0$ end. Let $\alpha=\alpha_{\pm}\left(H\right)$ and $\alpha'=\alpha_{\pm}\left(H'\right)$, and let $M\left(\alpha\right)$ and $M'\left(\alpha'\right)$ be the symmetric $2\times2$ matrices that appear in Equation \eqref{planes}.   At any $p \in \pi(P \cap P')$, $P$ and $P'$ intersect in a dimension $2-\operatorname{rk}\left(M\left(\alpha\right)-M'\left(\alpha\right)\right)$ set. We now decompose $M\left(\alpha\right)-M'\left(\alpha'\right)$ as the sum of two matrices, only one of which depends on the quantities $\alpha$ and $\alpha'$.
Let $\mu = n m_{0}-n' m_{0}'$, and $\kappa =  n k_{0}- n' k_{0}'$. When  $\mu \neq 0$, we can write 
\begin{align}\label{qdec} M\left(\alpha\right)-M'\left(\alpha'\right) &= \begin{pmatrix} - \left(\kappa\right)^2/ \mu &  \kappa \\
 \kappa & - \mu \end{pmatrix} + \begin{pmatrix}n \left(k_{0}\right)^2/\left(m_0\right)-  n' \left(k_{0}'\right)^2/\left(m_0'\right) + \alpha-\alpha'   & 0 \\
0 & 0 \end{pmatrix}\\
&\coloneqq M_1 + M_2 \left(\alpha-\alpha'\right); \end{align}
when $\mu =0$ we can write 
\begin{align}\label{qdec2}  M(\alpha)-M'(\alpha')&=\begin{pmatrix} n (k_{0})^2/(m_0)- n' (k_{0}')^2/(m_0')  & \kappa \\
\kappa & 0 \end{pmatrix} +\begin{pmatrix} \alpha-\alpha'   & 0 \\
0 & 0 \end{pmatrix}\\
& \coloneqq M_1 + M_2 (\alpha-\alpha').
\end{align}
Using this decomposition, one obtains that
\[ \det\left(M\left(\alpha\right)-M'\left(\alpha'\right)\right)= - \frac{\left(m_0' k_0 - m_0 k_0'\right)^2 n n'}{m_0 m_0'} + \left(\alpha - \alpha'\right)\left( n m_{0}-n' m_{0}' \right)= f\left(\alpha-\alpha'\right) \]
from which we may deduce $(a)$. In the case where $\mu \neq 0$, take $K\left(L, L'\right)$ to be the solution to the equation $f\left(K\left(L, L'\right)\right)$=0; otherwise take $K\left(L, L'\right)=0$.

 If $\mu=\kappa=0$, then $M_1$ vanishes, so $M(\alpha)-M'(\alpha')$ has rank $1$ if $\alpha-\alpha' \neq 0$. In this case one observes that $TP \cap TP'|_{p}$ is spanned by $\left(0,1,n k_{0}, - n m_{0}\right) \in \R^4$. \end{proof}


\begin{defn} Let $L$ and $L'$ be two admissible Hamiltonians of type $\mathbf{A}=\left(\left(m_0,k_0\right),\left(m_\infty, k_\infty\right)\right)$ and $\mathbf{A'}=\left(\left(m_0',k_0'\right),\left(m_\infty', k_\infty'\right)\right)$. Let $H$ and $H'$ be two Hamiltonians perturbing $L$ and $L'$. The ordered pair of Lagrangians $\left(\psi_H\left(L\right),\psi_{H'}\left(L'\right)\right)$ is in \emph{correct position} if $\psi_H\left(L\right)$ and $\psi_{H'}\left(L'\right)$ intersect transversely, if 
 for all but finitely many points $y \in L \cap L'$, $y$ lifts to a point $\tilde{y} \in \tilde{L} \cap \tilde{L'}$ with $ |\pi\left(\tilde{L} \cap \tilde{L}'\right)| \leq 2$, and:
\begin{enumerate}[(a)]
\item In the case where $\A=\A'$: if $H  -H' > 0$, $H-H'$ is an admissible Hamiltonian,  the corresponding planes of $\psi_H\left(L\right)$ and $\psi_{H'}\left(L'\right)$ do not coincide at either end, and any two lifts $\psi_{H}\left(\tilde{L}\right)$ and $\psi_{H'}\left(\tilde{L}'\right)$ to the universal cover are disjoint away from a compact set.
\item In the case where $\A \neq \A'$: if $H-H' > 0$, $H-H'$ is an admissible Hamiltonian,  $\alpha_{-}\left(H\right) - \alpha_{-}\left(H'\right) >  K\left(L, L'\right)$ if $\left(m_0,k_0\right) \neq \left(m_0',k_0'\right)$, $\alpha_{+}\left(H\right) - \alpha_{+}\left(H'\right) >  K\left(L, L'\right)$,and  any two lifts $\psi_{H}\left(\tilde{L}\right)$ and $\psi_{H'}\left(\tilde{L}'\right)$ to the universal cover are disjoint away from a compact set.
\end{enumerate}
An ordered collection $\{L_0, L_1, L_2, \ldots, L_n\}$ is in correct position if each pair $\left(L_i, L_{i+1}\right)$ is in correct position for $0 \leq i \leq n-1$.
\end{defn}

 \begin{defn} Let $L$ and $L'$ be in correct position. We say that $L$ and $L'$ are \emph{essentially transverse} at an end of $Y$ if their corresponding planes intersect in a point on that end, \emph{essentially non-transverse} at an end if their corresponding planes intersect in a line, and \emph{essentially disjoint} at an end if their corresponding planes do not intersect.  Note that in the fully wrapped case, $L$ and $L'$ are either essentially transverse  at both ends of $Y$, or  essentially non-transverse at both ends of $Y$. In this case being essentially transverse is equivalent to satisfying $[\gamma_L] \neq [\gamma_{L'}]$.

We also have a notion of two lifts $\tilde{L}$ and $\tilde{L'}$ being essentially transverse in any setting. Given two such lifts $\tilde{L}$ and $\tilde{L}'$, $\tilde{L}$ and $\tilde{L}'$ are \emph{$I$-essentially transverse}, \emph{$I$-essentially non-transverse}, and \emph{$I$-essentially disjoint} if there exists a type-$I$ Hamiltonian $H$ such that $\psi_{H}\left(\tilde{L}\right)$ and $\tilde{L'}$ intersect in a point, intersect in a line, or do not intersect, respectively.
  \end{defn}

\subsection{Associated quadratic forms and action} 
Let $L, L' \in \G_P$ be two Lagrangians in correct position. Fix lifts $\tilde{L}$ and $\tilde{L}'$ of $L$ and $L'$ in $\R^4$, and corresponding planes $\left(P^{\pm}\right),\left(P^{\pm}\right)'$ of these lifts. 

\begin{defn}\label{qdecomp} \begin{enumerate}[(a)] \item Let $\mu$ and $\kappa$ be defined as above for the Lagrangians $L$ and $L'$. Define the coordinate function $w_{L,L'}: \R^4 \to \R$ by $w_{L,L'}= s - (\kappa/\mu) t$ when $\mu \neq 0$; otherwise let $w_{L,L'} = s$. 
\item Consider the quadratic form defined by $\frac{1}{2}x^T\left(M-M'\right) x$; denote this form by $Q_{L,L', +}\left(x\right)$, and define $Q_{L,L', -}\left(x\right)$ similarly. We can write this sum of two  quadratic forms using the decomposition of Equations \eqref{qdec} and  \eqref{qdec2}: let $Q_{L,L',+}^\beta(x) = \frac{1}{2} x^T M_1 x$, and $Q_{L,L',+}^{\alpha}=\frac{1}{2}x^T M_2 x$. Then except in the case where $\mu=0$, $\kappa\neq0$, the coordinates $(w_{L,L'},t)$ diagonalize the quadratic forms. In this case we can write
\begin{align*} Q_{L,L', \pm} (x)&= Q_{L,L',\pm}^\beta(x) + Q_{L,L',\pm}^{\alpha}(x) \\
Q_{L,L',\pm}^\beta(x) &= \frac{1}{2} \beta w_{L,L'}(x)^2 \\
 			Q_{L,L',\pm}^{\alpha}(x) &=\frac{1}{2}\alpha_{\pm} t(x)^2. \end{align*}
\end{enumerate}
\end{defn} 

\begin{remark} The quadratic form $Q_{L,L}^\beta(x)$ is an invariant of the homology classes $[\gamma_{L}], [\gamma_{L'}]$. Thus to any two Lagrangians, we can associate such a quadratic form without the data of the perturbing Hamiltonians. \end{remark}

   \begin{remark} On any end of $Y$,  $Q_{L,L',{\pm}}$ is degenerate as a quadratic form on $\R^2$ if and only $L$ and $L'$ are essentially non-transverse at that end. On the wrapped end(s) of $Y$, $Q_{L,L',{\pm}}$ has signature $(2,0)$ when $m n - m' n' < 0$, signature $(1,1)$ when $m n - m' n<0$ or $m n - m' n'=0$ and $k d - k' n'\neq 0$, and is degenerate and positive semi-definite when  $m n - m n'=k n - k' n' = 0$. In the fully wrapped setting, the signatures of $Q_{L,L',{+}}$ and $Q_{L,L',{-}}$ agree. \end{remark}

\begin{lemma}\label{primitive} \begin{enumerate}[(a)] \item Let $h$ and $h'$ be two primitives for $\tilde{L}$ and $\tilde{L}'.$ Then for $\pm t >>0$, we can write
\begin{align*} \left(h-h'\right)\left(t,s\right)&= Q_{L,L',{\pm}}\left(t,s\right) + a t + b s + \rho_{0}^{c}\left(t\right) +f\left(t,s\right)+ G\left(t,s\right);\\
		&= \frac{1}{2} \alpha t^2 + Q_{L, L'}^\beta(t,s) + a t + b s + \rho_{0}^{c}\left(t\right) +f\left(t,s\right)+ G\left(t,s\right),  
		\end{align*}
where $a, b, c \in \R$, $\alpha \geq 0$, $\|f\|_{C^2} < 1$, and $G\left(t, s\right)$ is a compactly supported function on $C_R$ for some $R \in \N$. On any end with wrapped boundary conditions of the form $m \phi_t+ k \phi_s \equiv 0$, this takes the form 
\[ \frac{1}{2}\beta w^2+ a w + \rho_{0}^{c}\left(t\right) +f\left(t,s\right)+ G\left(t,s\right).\]
\item  For $\pm t >> 0$, $h-h'$ can be rewritten as
\[ Q_{L,L',\pm} \left(y_0\right) + Q_{L,L',\pm} \left(y-y_0\right)  +\rho_{0}^c(t)+ G\left(t,s\right)+f\left(t,s\right) \]
where $y_0$ is an intersection point of the corresponding planes $P^+$ of $\tilde{L}$ and $(P^+)'$ of $\tilde{L}'$. \end{enumerate} 
\end{lemma}
\begin{proof} (a) follows from Lemma \ref{formofh} and the decomposition of Definition \ref{qdecomp}. (b) follows from the fact that any function $F: \R^2 \to \R^2$ of the form $F(x)=x^T M x + (e_1, e_2) \cdot x$ can be expanded around a critical point $x_0$ as $x_0^TMx_0+(x-x_0)^T M(x-x_0)$. \end{proof}

\begin{defn}Let $y \in L \cap L$ lift to $\tilde{y} \in \tilde{L} \cap \tilde{L}'$.  Let $h$ and $h'$ be primitives for $\tilde{L}$ and $\tilde{L}'$ of the form prescribed by \ref{formofh}. Define $S_{\R^4}\left(\tilde{y}\right)$ by 
\[ S_{\R^4}\left(\tilde{y}\right)= (h-h') \left(\tilde{y}\right). \] 
\end{defn}

\begin{defn} Recall the quadratic form $Q_{L,L'+\pm}^\alpha$ given in the decomposition of Equation \eqref{qdec}. Define the \emph{action} of a point $y \in \left(L \cap L'\right)$ with $\pm t\left(y\right) > 0$ by 
\[ S(y) = Q_{L,L', \pm}^\alpha\left(\tilde{y}\right)=\frac{1}{2}\alpha_{\pm} t^2. \] Note that this action is defined regardless of the chosen lift of $\tilde{y}$, since all lifts occur at the same $t$-value. \end{defn}

\begin{prop}\label{quadbounds} \begin{enumerate}[(a)]
\item  There exists $C_1 \in \R$ such that for all $\hn \in \Z^3$, for all $\tilde{y} \in \tilde{L} \cap \left(\tilde{L}' - \hn\right)$ with $\pm t(y) \geq 0$, \[d\left(\tilde{y}, \left(P^\pm\right)\cap \left(\left(P^\pm\right)'-\hn\right)\right)<C_1.\]
\item  There exists $C_2 > 0$ such that for all $\tilde{y} \in \tilde{L} \cap \left(\tilde{L}'-\hn\right)$ with $ \pm t\left(\tilde{y}\right) >  0$, 
\[ \left| S_{\R^4}\left(\tilde{y}\right) - Q_{L,L',\pm} \left(\tilde{y}\right)\right| < C_2. \] 
\item If $L$ and $L'$ are essentially non-transverse at an end, then there exists $C_3 > 0$ such that for all $\tilde{y} \in \tilde{L} \cap \tilde{L'}$, 
\[ \left| S_{\R^4}\left(\tilde{y}\right) -S(y) \right| < C_3. \] 
\end{enumerate} \end{prop}
\begin{proof}\begin{enumerate}[(a)]
\item This follows from Equation \eqref{boundfromp}: there exists $D_0> 0$ such that for such $\tilde{y}$, there exist $x \in P^{\pm}, x' \in (P^{\pm})'-\hn$ such that $d\left(\tilde{y},x \right),d\left(\tilde{y},x'\right) < D_0$.  There is a constant $D_1>0$ independent of $\hn$ such that for any pairs of points $x \in  P^{\pm}$,  $ x' \in (P^{\pm})'-\hn$ with $d(x,x')<C$, $d\left(x,\left(P^\pm\cap \left((P^\pm)'-\hn\right)\right)\right)< C D_1$. Taking $C_1 = 2 D_1 D_0+ D_0$ gives the result.

\item  The quantity $ \left| S_{\R^4}\left(\tilde{y}\right) - Q_{L,L',\pm} \left(\tilde{y}\right)\right|$ depends only on the equivalence class $\overline{\hat{n}} \in \Z^3/ \Gamma(L, L')$. If $L$ and $L'$ are essentially disjoint on the $\pm t \to \infty$ end, then there are only finitely many classes $\overline{\hn} \in \Z^3/\Gamma(L,L')$ so that $\tilde{L}$ and $\tilde{L'}-\hn$ intersect and satisfy $\tilde{L} \cap \left(\tilde{'L}-\hn\right) \subset Y_{0}^\pm.$ Otherwise, by Lemma \ref{primitive}, we can write 
\[ D_+= \max_{|x|<C_1} \left(Q_{L,L',+} \left(x\right)\right)  + \|\rho_0^c + G\left(t,s\right)+f\left(t,s\right)\|_{C^2}. \]
We similarly obtain a bound $D_-$ on the $t>0$ end of $Y$. Take $C_2 = \max\{D_+, D_-\}$. 

\item This follows from (a) and (b), using the fact that in this case $ Q_{L,L',\pm}^\alpha= Q_{L,L',\pm}$. 
\end{enumerate} \end{proof}

\begin{lemma}\label{boundondiff} Fix a lift $\tilde{L}$ of $L$, and let $H$ be a Hamiltonian such that the set $\left(\psi_{H}\left(L\right), L, L'\right)$ is in correct position. Assume also that $Q^\beta$ is a degenerate quadratic form equal to $\frac{1}{2}\beta w_{L,L'}^2$ for some $\beta \in \R$, where $w_{L,L'}$ is the coordinate defined in Lemma \ref{qdecomp}. Then there exists $D > 0$ such that for all $\hn \in \Z^3$, if $y_0 \in \tilde{L} \cap \left(\tilde{L}'-\hn\right)$, and $y_1 \in \psi_{H}\left(\tilde{L}\right) \cap \left(\tilde{L'+ \hn}\right)$, then 
\[ \left|\left(S_{\R^4}\left(y_0\right) - S\left(y_0\right)\right) - \left(S_{\R^4}\left(y_1\right)-S\left(y_1\right)\right)\right|   < D. \] 

If we further assume that $L$ and $L'$ lift to planes in $\R^4$, then 
\[ \left|\left(S_{\R^4}\left(y_0\right) - S\left(y_0\right)\right) - \left(S_{\R^4}\left(y_1\right)-S\left(y_1\right)\right)\right|   = 0. \] \end{lemma}

\begin{proof} Let $P^\pm, P_H^\pm,$ and $(P')^\pm$ denote the $(\pm)-$corresponding planes of $\psi_{H}\left(\tilde{L}\right)$, $\tilde{L}$, and $\tilde{L'}$, respectively. From Proposition \ref{quadbounds}, it suffices to show that for all $\hn \in \Z^3$, for $x_0 \in P^\pm \cap\left( \left(P'\right)^\pm-\hn\right)$, $x_1 \in \left(P_H^\pm \left(P'\right)^\pm-\hn\right)$, 
\begin{align*} &\left| \left(Q_{L,L', +}\left(x_0\right)- Q_{L,L', +}^\alpha\left(x_0\right)\right) -\left(Q_{\psi_{H}\left(L\right),L', +}\left(x_1\right)- Q_{\psi_{H}\left(L\right),L', +}^\alpha\left(x_1\right)\right) \right| \\
		 &= \left|Q_{L,L'}^\beta (x_0)-Q_{L,L'}^\beta (x_1)\right|\\
		 & = \frac{1}{2} \beta \left( w_{L,L'}(x_0)^2-  w_{L,L'}(x_1)^2\right) =0
		 \end{align*}

First assume that $t\left(y\right) > 0$. Lift to the universal cover, and work in the $t,w_{L,L'}$ coordinates. In these coordinates, the intersection of the corresponding planes  $P$ and $P' + \hn$ is given by $x_0=\left(t_0,w_0\right)$ solving
\[ \begin{pmatrix} \alpha & 0 \\
				0 & \beta \\ \end{pmatrix} \begin{pmatrix} t_0 \\ w_0 \end{pmatrix} + \begin{pmatrix} e_0 \\ e_1 \end{pmatrix} = 0;  \]
the intersection of the corresponding planes $P_{H}$ and $P' + \hn$ is given by $x_1=\left(t_1, w_1\right)$ solving
\[ \begin{pmatrix} \alpha+\alpha_+\left(H\right) & 0 \\
				0 & \beta \\ \end{pmatrix} \begin{pmatrix} t_1 \\ w_1 \end{pmatrix} + \begin{pmatrix} e_0+\epsilon_+\left(H\right) \\ e_1 \end{pmatrix} = 0.  \]			
Thus  $\beta w_{L,L'}(x_0)^2=\beta  w_{L,L'}(x_1)^2$; when $\beta \neq 0$, $w_{L,L'}(x_0)^2= w_{L,L'}(x_1)^2$.
When $L$ and $L'$ are planar, one has $x_0=y_0$, $x_1= y_1$, and $S_{\R^4}(x_i)=Q_{L,L',+}(x_i)=Q_{L,L',-}(x_i)$.
\end{proof}

\section{$A_\infty$ structures}\label{ainfs}
\subsection{J-holomorphic curves}
The Fukaya categories we consider are $A_\infty$ categories where the higher products $\mu^n$ are defined by counts of pseudo-holomorphic curves. In our case, the definition of the $\mu^n$ products in our $A_\infty$ categories requires the use of specific almost complex structures, the definition of which reflects the controls we impose on the ends of $Y$.


\begin{defn} An \emph{eventually constant} almost complex structure is an $\omega$-compatible almost complex structure for which there exists $T$ and $J_{+},J_{-} \in \mathcal{J}\left(\R^4, \omega_{std}\right)$ for which for all points $p \in Y^+_T$, $J|_{p}= J_+$, and for all points $p \in Y^-_T$, $J|_p=J_-$, using the natural trivialization $T Y = Y \times \R^4$. Let $\mathcal{J}_{EC}$ denote the set of eventually constant $J$. \end{defn}

Recall that the choice of an almost complex structure $J$ on $Y$ will induce a metric $g_J$ defined by $g_J\left(v,w\right) = \omega\left(v, J w\right)$; an almost complex structure is eventually constant if and only if the metric tensor is eventually constant. 

\begin{prop} The space of eventually constant almost complex structures is contractible. \end{prop}
\begin{proof} The linear homotopy between any $J$ and a fixed $J_0$, constructed using the identification of $J$ with $g_J$, is through the space of eventually constant almost complex structures. \end{proof}

We will use the notation $g_E$ to denote the Euclidean metric on $Y$ pushed forward from the Euclidean metric on $\R^4$. We use $d_J\left(\cdot, \cdot\right)$ to denote the distance function induced from $J$, and $d_E\left(\cdot, \cdot\right)$ to denote the distance function induced by $g_E$. We use the notation $B_J\left(y, R\right) \subset Y$ to denote the $g_J$ ball of radius $R$ around $y'$, i.e. the set of points $y' \in Y$ such that $d_J\left(y,y'\right)<R$, and we use the notation $B_E\left(y,R\right)$ to denote the Euclidean ball centered at $y.$

Note that at each point $y \in Y$ there exists a constant $D_y < 1$ such that for all $v,w \in TY|_y$, 
\begin{equation}\label{normcomp} \frac{1}{D_y} g_E\left(v,w\right) > g_J\left(v,w\right) >D_y  g_E\left(v,w\right). \end{equation}  Let $\mathcal{J}$ be a family of eventually constant complex structures over a compact base; then there exists a universal constant $D_{\mathcal{J}}$ such that \eqref{normcomp} holds with $D_\J$ at all points, for all $J \in \mathcal{J}$. Then $B_{\J} \left(y, R D\right) \subseteq B_E\left(y, R\right)$ for all $y \in Y$. 

For a compact family $\J$, and collection of Lagrangians $L_0, \ldots, L_n \in \G_P$, $L_i = \psi_{H_i}\left(L_i'\right)$, fix  \[T_{\lin} \left(L_0, L_1 \ldots, L_n , \mathcal{J}\right) > 0\] such that the boundary conditions on each $L_i'$, $H_i$,and $J \in \J$ are satisfied on $Y_{T_{\lin}} \left(L_0, L_1 \ldots, L_n , \mathcal{J}\right)$.


\subsubsection{Bounds on J-holomorphic curves} 
Our boundary conditions can be used to derive bounds which will control the image of any $J$-holomorphic curve in a fixed homology class. These bounds will be useful throughout, and in particular will be used to show that the moduli spaces we use to define our structure maps are well-behaved.

Let $\left(L_0, \ldots, L_n\right)$ be a collection of Lagrangians $\G_P\left(I\right)$ in correct position, and let $y_i \in L_i \cap L_{i+1}, 0 \leq i \leq n-1$, $y_n \in L_0 \cap L_{n}$. Let $S$ be a disc with boundary marked points $\{p_0, \ldots, p_n\}$, and let $\p_i S$ denote the boundary component of $S$ so that $p_{i-1}, p_{i} \in \overline{\p_i S}$. Let $\J=\{ J_s \}_{s \in S}$ be a family of eventually constant almost complex structures parametrized by $S$ such that there exists $J$ such that $J_{s} = J$ outside of a compact subset of $Y$ for all $J_s$. Let $u: S \to Y$ be a map satisfying 
\begin{align} &\overline{\p_{\J}} u = 0; \label{Jeqn1}\\
			 &u\left(\p_i S\right) \in L_i  \label{Jeqn2}\\
			 & u\left(p_i\right) = y_i. \label{Jeqn3} \end{align}

We will use counts of such $J$-holomorphic $u$ to define the $A_\infty$ structures on our Fukaya categories. The features of our Lagrangians give us controls over the behavior of such $u$. In the $\A$ setting, we can ensure that $J$-holomorphic curves lie within a compact region in the same way that one does when defining a Fukaya-Seidel category:

\begin{lemma}\label{RIE1} Let $L_0, L_1 \ldots, L_n \in \G_P\left(\A\right)$ be a collection of transverse Lagrangians in $\G_P\left(\A\right)$ in correct position. Then the image of any $u$ satisfying Equations \eqref{Jeqn1}, \eqref{Jeqn2} and \eqref{Jeqn3} is contained within $Y \setminus Y_{T_{\lin} \left(L_0, L_1 \ldots, L_n , \mathcal{J}\right)}$.  \end{lemma}

\begin{proof} Let $J_+: \R^4 \to \R^4$ denote the complex structure on $\R^4$ so that $J_s|_{y} = J_+$ for any $s \in S$ and point $y$ such that $t\left(y\right) > T_{\lin} \left(L_0, L_1 \ldots, L_n , \mathcal{J}\right)$, and $g_{J_+}$ denote the induced metric on $\R^4$. 
The non-degeneracy of $\omega$ implies that there exists $v \in \R^4$ such that for any $w = \left(r_1, r_2, r_3, r_4\right) \in \R^4$, 
\[ \omega\left(v,w\right) = m_{\infty} r_3 - k_{\infty} r_4 \]
for all $w \in \R^4$. Consider the linear map $\Pi_v: \R^4 \to \C$ defined by $\Pi_v\left(w\right) =  \left(g_{J_+}\left(v, w\right), \omega\left(v,w\right)\right)$. Note that
\[ \Pi_v\left(J_+ w\right) =\left(\omega\left(v, J_+ J_+ w\right),\omega\left(v, J_+ w\right)\right) = \left(-g_{J_+}\left(v, w\right), \omega\left(v, w\right)\right)= i \Pi\left(w\right). \]
so $\Pi_v$ is $J_+$-holomorphic. Now, observe that the image of any lift $\left(\widetilde{L_j}\right)_{T_{\lin} \left(L_0, L_1 \ldots, L_n , \mathcal{J}\right)}^+$ under $\Pi_v$ is a ray (by the fact that $L_j$ is planar at infinity and Lagrangian), which is contained within the line \[\{ z \in \C \mid \mathrm{Im }\, z = m_{\infty} \epsilon_+\left(H_j\right) + n \}\] for some $n$ depending on the lift; when $i \neq j$ these lines are disjoint by the conditions on the perturbing Hamiltonians $H_j$. Since the composition \[\Pi_z \circ \tilde{u}: \tilde{u}\inv\left( \{ t \geq  T_{\lin} \left(L_0, L_1 \ldots, L_n , \mathcal{J}\right)\}\right) \to \C\] is holomorphic on the interior of $S$ and has compact image with boundary contained within a finite collection of parallel lines, we conclude that  $\tilde{u}\inv\left( \{ t \geq  T_{\lin} \left(L_0, L_1 \ldots, L_n , \mathcal{J}\right)\}\right)=\emptyset$. Thus the image of $u$ is contained within the region $\{y \in Y \mid t\left(y\right) < T_{\lin} \left(L_0, L_1 \ldots, L_n , \mathcal{J}\right)\}$. Applying the same argument at the negative end gives the result.

\end{proof}

After we relax the boundary conditions, we can no longer constrain all holomorphic curves to lie within a compact set via a holomorphic projection. We instead rely on the following bound, which we obtain by using the planar regions of each Lagrangian which are ensured by the vanishing conditions on the non-quadratic part of the perturbing Hamiltonians. 

\begin{lemma}\label{RIE} Let $L_0, L_1 \ldots, L_n \in \G_P$ be a collection of transverse Lagrangians in $\G_P$ in correct position; write $L_i = \psi_{H_i} \left(L_i'\right)$. Let $y_i \in L_{i} \cap L_{i+1}$, $0 \leq i \leq n$, and let $u:S \to Y$ be a $\mathcal{J}$-holomorphic map satisfying Equations \eqref{Jeqn1}, \eqref{Jeqn2}, and \eqref{Jeqn3} for $\J=\{J_s\}_{s \in S}$. For $T > 0$, let $\ell_T^+\left(u\right)$ denote the length of the interval $\left(t \circ u\right)\inv\left(Y_T^+\right) \subset \R$, and $\ell_T^-\left(u\right)$ denote the length of the interval $\left(t \circ u\right)\inv\left(Y_T^-\right) \subset \R$. Let $\ell_T\left(u\right)=\ell_T^+\left(u\right) + \ell_T^-\left(u\right)$. Then there exists  $R > 0$, $M > 0$, and $T_B > T_{\lin}\left(L_0, L_1 \ldots, L_n,\mathcal{J}\right)$, depending only on the collection $\left(L_0, L_1 \ldots, L_n , \mathcal{J}\right)$,  such that for all such $u$ with $\ell_{T_B} \left(u\right) > M$
\[ \ell_{T_B} \left(u\right) \leq R \int_{S} u^* \omega. \]
\end{lemma}
\begin{proof} We first fix several pieces of geometric data about the collection $\left(L_0, L_1 \ldots, L_n , \mathcal{J}\right)$. First, let $\epsilon>0$ be such that for each pair $0 \leq i < j \leq n$, two lifts $\widetilde{L_i}$ and $\widetilde{L_j}$ are a distance at least $\epsilon$ away from each other outside of a compact set. For two corresponding planes which have translates that intersect, observe that there exists $C_{ij} > 0$ such that for any two translates $P_i^{\pm}\left(a_i, b_i\right)$, $P_j^{\pm}\left(a_j, b_j\right)$ which intersect at a fixed $t$-value $t_{i,j,a_i, b_j,a_j, b_j}$, for all $y \in P_i^\pm\left(a_i, b_i\right)$,
\[ C_{i j} |t\left(y\right)-t_{i,j,a_i, b_j,a_j, b_j}| < d_E\left(y, P_j^\pm\left(a_j, b_j\right)\right). \]
Let $C=\min \{C_{ij}\}$.

Consider now the collection of small perturbing functions $f_+^{i}= f_+\left(H^+_i\right)$. Let $N_i$ denote the period of $f_+^i$, and let $\kappa_i$ denote $\kappa\left(f_+^i\right)$.
Note that for every point $y \in L_i$ with $t\left(y\right)> T_{\lin}\left(L_0, L_1 \ldots, L_n,\mathcal{J}\right)$, $d_E\left(y, P_i^+\right) <  \kappa_i$. Let \[N=N_0 \cdot N_1 \cdots N_n >1.\]  Finally let \[\gamma = \min \{ \kappa_i/2, \epsilon, \delta \}.\]

Let $J_+ \in \End\left(\R^4\right)$ denote the limiting complex structure of the family $\J$ at the $t >> 0$ end of $Y$. For each affine corresponding plane $P_i^+ \subset \R^4$, choose $v_i \in \R^4$ with $g_{J_+}\left(v_i, v_i\right)=1$ such that $v_i$ is $g_{J_+}$-orthogonal to the translate of $P_i^+$ passing though the origin. Let $\Pi_{v_i}: \R^4 \to \C$ denote the map defined above:
\[ \Pi_{v_i}\left(v\right) = \left(g_{J_+}\left(v_i, v\right), \omega\left(v_i, v\right) \right). \] Then there exists $\iota \in \R$ such that 
\[ \Pi_{v_i}\left(P_i^+\right) = \{ z \in \C \mid \operatorname{Re}\left(z\right) = \iota\}, \]
i.e., $P_i$ maps to a vertical line under $\Pi_{v_i}$. Let $w_i$ and $\iota'$ be such that
\begin{align*} &\omega\left(v_i, w_i\right)=g_{J_+} \left(v_i, J w_i\right) = 0\\
& g_{J_+}\left(w_i,  w_i\right)=1;\\
&g_{J_+}\left(P_i^+, w_i\right)=\iota'.\end{align*} Then the holomorphic lines spanned by $v_i$ and $w_i$ are orthogonal. One can then verify that the map 
\[ \Pi_{v_i} \oplus \Pi_{w_i}: \left(\R^4, \omega, J_+\right) \to \left(\C \oplus \C, \omega_{\C} \oplus \omega_{\C}, i \oplus i\right)\] is an isomorphism of K\"{a}hler vector spaces, and the image of $P_i^+$ under such a map is the product of two vertical lines. 
 
Now, let $u: S \to Y$ be such a $\J$-holomorphic map so that $u\inv\left(Y_{T_{\lin}\left(L_0, L_1 \ldots, L_n,J\right)}^+\right) \neq \emptyset$. Lift $u$ to a map $\tilde{u}: S \to \R^4$ with boundary on lifts $\widetilde{L_i}$, and let $P_i^+$ be the  corresponding planes to these lifts. When $P_i^+$ and $P_j^+$ intersect, let $t_{ij} = t\left(P_i^+ \cap P_j^+\right)$.

 Note that $u$ satisfies a maximum principle for the $t$ -coordinate by the argument above, so any $t$-value achieved by $u\left(S\right)$ is achieved by a point $z \in \p S$. For each $i$, $0 \leq i \leq n$, consider the interval $I_i = t\left( \p_i S\right) \subset \R$; let $\ell_i$ denote the length of the interval $I_i$. Note that 
 \[ \sum_{i=0}^n \ell_i \geq \ell_{T_{\lin}\left(L_0, L_1 \ldots, L_n,\mathcal{J}\right)}^+\left(u\right), \]
so there exists some $i$ such that 
\[ \ell\left(I_i\right) \geq \frac{\ell_{T_{\lin}\left(L_0, L_1 \ldots, L_n,\mathcal{J}\right)}^+\left(u\right)}{n}. \]
Then there exists some subinterval $I^0 \subset I_{i}$ of the form $I^0=[t_{i j}, t_{i j'}]$ with $j, j'\neq i$ so that $t_{ij''} \notin \left(t_{i j}, t_{i j'}\right)$ for all $j'' \neq i$, and 
\[ \ell\left(I^0\right) \geq \frac{\ell\left(I_i\right)}{n} \geq \frac{\ell_{T_{\lin}\left(L_0, L_1 \ldots, L_n,\mathcal{J}\right)}^+\left(u\right)}{n^2}. \]

We now choose a sequence of points $y_0, \ldots, y_K \in \tilde{u}\left(\p_i S\right)$ such that \begin{itemize}
\item $t\left(y_k\right) \in I^0$ for all $k$;
\item $t\left(y_k\right) = \left(t_0\right)_i + m_k N$, $m_k \in \N$, and $m_k \neq m_{k'}$ for $k \neq k'$;
\item $|t\left(y_k\right)-t_{ij} |> \frac{\gamma+\delta}{C}$ for $i \neq j$
\item $K$ is large: \[ K \geq \frac{\ell\left(I^0\right)}{N} - 2 N - \frac{2 \gamma+\delta}{C} \geq \frac{\ell_{T_{\lin}\left(L_0, L_1 \ldots, L_n,\mathcal{J}\right)}^+\left(u\right)}{N n^2}  - 2 N - \frac{2 \gamma+\delta}{C}.\]  \end{itemize}

Since $|f_i|<\frac{\kappa_i^2}{4}$, we observe that if $y \in \widetilde{L_i}$ satisfies $d_{E}\left(y, y_k\right) < \frac{\kappa}{2}$, then $y \in P_i^+$. Note also that for all $j \neq i$,
\[ d_E\left(y_k, P_j\right) \geq C |t\left(y_k\right)-t_{ij} | > \gamma+\delta \]
so 
\[ d_E\left(y_k, \widetilde{L_i}\right) > d\left(y_k, P_i\right) - \delta > \gamma. \]

Thus each $y_k$ is a distance at least $\gamma$ away from $\widetilde{L_j}$ for all $j \neq i$, and is a distance at least $\frac{\kappa}{2} > \gamma$ away from the non-planar region of $\widetilde{L_i}$.

Now, for each $y_k$, consider the set 
\[ Q_k = B_E\left(y_k, D_{J^+} \frac{\gamma}{3}\right). \] 
By construction, $Q_k \cap Q_{k'}=\emptyset$ if $k \neq k'$, and $Q_k \cap \widetilde{L_i} \subset P_i^+$.

We also conclude that $Q_k \cap \widetilde{L_j} = \emptyset$ for $j \neq i$. Finally, note that $\widetilde{L_i}$ is planar within $Q_k$, i.e. 
\[ \widetilde{L_i} \cap Q_k =  P_i^+ \cap Q_k. \]

Let $U_k = \tilde{u}\inv\left(Q_k\right) \subset S$. The above shows that $\p S \cap U_k \subset \p_i S$, and $\tilde{u}\left(\p S \cap U_k\right)\subset P_i^+$. By Gromov's monotonicity lemma (\hspace{1sp}\cite{gromov}) and the reflection principle, we can bound 
\[ \int_{U_k} \tilde{u}^* \omega \geq \frac{\pi}{2}\left(\frac{ D_J^+ \gamma }{3} \right)^2.\]

Note that the quantity $D_{J^+} \gamma^2$ depends only on $J_+$ and the Lagrangians $L_0, \ldots, L_n$, and the choice of $L_i$. Letting $R$ denote the minimum over this quantity $D \gamma^2$ over such a choice of $I$ and over both ends, and taking $T_B=T$ gives the result. \end{proof}

\begin{cor}\label{boundedregion} Let $A > 0$, let $\mathcal{S}$ be a compact family of discs with $n$ boundary marked points, let $\J: \mathcal{S} \to \J_{EC}$ be a family of eventually constant almost complex structures, and let $y \in Y$.  Then there exists a compact region $V\left(A, y, \mathcal{S}, \J\right) \subset \left(Y\right)$ such that all $u: S \to Y$ which satisfy Equations \eqref{Jeqn1}, \eqref{Jeqn2}, \eqref{Jeqn3}  with $y \in u\left(S\right)$, $S \in \mathcal{S}$, and which have $\int_{S} u^* \omega < A$, satisfy
\[ u\left(S\right) \subset V\left(A, y, \mathcal{S}, \J\right). \] \end{cor}
\begin{proof} We can find a universal $R > 0$ and $T_B$ so that Lemma \ref{RIE} is satisfied for this family. Then if $y' \in Y$ has 
\[ |t\left(y'\right)| >   \left(T_B + |t\left(y\right)|+1+A\right)\left(R +1\right) \] 
and $y' \in u\left(S\right)$, then 
\[ \int_{S} u^* \omega \geq \frac{\ell_{T_B}\left(u\right)}{R} > A. \]
\end{proof}

\begin{cor}\label{boundedmodspaces} The spaces $\overline{\mathcal{M}_A\left(y_0, \ldots, y_n, [\beta], \J\right)}$ defined by 
\[ \bigcup_{\omega\left([\beta]\right) < A}  \overline{\mathcal{M}_A\left(y_0, \ldots, y_n, [\beta], \J\right)}. \]  are compact. \end{cor}
\begin{proof} We have excluded sphere and disc bubbling by our conditions on the Lagrangians. It suffices to show that all $\left(r,u\right) \in \overline{\mathcal{M}_A\left(y_0, \ldots, y_n, [\beta], \J\right)}$ are contained within some compact region $V_A \subset Y$. But this follows from the fact that the image of each such $u$ passes through the point $y_0$ and must have 
\[   \ell_{T_B}\left(u\right) \leq R \omega\left([\beta]\right) < R A \]
where $T_B$ and $R$ are above.  The result then follows by Gromov compactness. 
\end{proof}

\subsubsection{Moduli spaces of $J$ holomorphic curves}\label{modspaces}
The definition of the moduli spaces that we use to define the $\mu^n$ will be exactly the same as in the literature (c.f. \cite{seidelbook}), but our boundary conditions force us to slightly adapt some standard arguments that are used to show that counts are well-defined. Let $\{ L_\sigma \}_{\sigma \in \Sigma}$ be a set Lagrangians in $\G_{P}$ indexed by a partially ordered set $\Sigma$ with the property that any collection $\left(L_{\sigma_0}, L_{\sigma_1}, \ldots L_{\sigma_n}\right)$ with $\sigma_0 >\sigma_1 > \ldots > \sigma_n$ is in correct position. We inductively define compactifications of moduli spaces of $J$-holomorphic curves using this set of Lagrangians. 

We first define the moduli spaces of bigons. Let $L_0, L_1 \in \{ L_\sigma \}_{\sigma \in \Sigma}$ be two Lagrangians with $L_0 > L_1$. Let $y_0, y_1 \in L_0 \cap L_1$. Let $[\beta] \in \pi_2\left(M,L_0,L_1\right)$. Let $\J=\{J_{x}\}_{x\in[0,1]} \in \mathcal{J}_{EC}$ be a family of eventually constant almost complex structures with the property that there exists $J \in \J_{EC}$ such that for all $x \in [0,1]$, $J_x = J$ outside of a compact set. 

Let $\widehat{\mathcal{M}}\left(y_0,y_1,[\beta], \J\right)$ denote the moduli space of solutions $u: \R \times [0,1] \to Y$ in the class $[\beta] \in \pi_2\left(M,L_0,L_1\right)$ to the equation
\begin{align}\label{Jholo}& \overline{\p}_\J u=0 \end{align}
with the boundary conditions \begin{align}
		    & u\left(\R \times \{0\}\right) \subset L_0,  u\left(\R \times \{1\}\right) \subset L_1; \\
		    & \lim_{t \to \infty}u \left(t,x\right) = y_0, \lim_{t \to  -\infty} u\left(t,x\right) = y_1. \end{align}
Let $\mathcal{M}\left(y_0, y_1, \J\right)$ denote the quotient of  $\widehat{\mathcal{M}}\left(y_0,y_1,\J\right)$ by the action of $\R$ by re-parametrization. For generic $\J\left(L_0,L_1\right)$, the moduli space $\mathcal{M}\left(y_0, y_1, [\beta], \J\left(L_0,L_1\right)\right)$ is smooth, and admits a Gromov compactification by broken strips (since we have excluded sphere or disc bubbling by the condition that $\langle \omega,  \pi_2\left(Y,L_i\right) \rangle=0$, and since all such curves $u$ lie within a fixed compact region by Corollary \ref{boundedregion}).

We then define moduli spaces which will we use to define higher products in the Fukaya categories. Let $\overline{\mathcal{R}_{n+1}}$ be the Deligne-Mumford-Stasheff compactification of the space of discs with $n+1$ boundary marked points in counterclockwise order and let $\overline{\mathcal{S}_{n+1}}$ denote the universal family over the disc. Denote the boundary marked points on such a disc by $p_0, \ldots, p_{n}$. We choose strip-like ends 
\[ \epsilon^{n+1}_j: \R_+ \times [0,1] \times \overline{\mathcal{R}_{n+1}} \to \overline{\mathcal{S}_{n+1}} \]
and
\[ \epsilon^{n+1}_j: \R_- \times [0,1] \times \overline{\mathcal{R}_{n+1}} \to \overline{\mathcal{S}_{n+1}} \]
which parametrize the disc near the marked points. We demand that these ends are compatible near the boundary strata: near each boundary face we have gluing maps
\[ \left(R, \infty\right) \times \overline{\mathcal{R}_{n+1}} \times \overline{\mathcal{R}_{k+1}} \times  \to \overline{\mathcal{R}}_{n+k}, \] and we require that on these boundary collars $\epsilon^{n+k}_j$ agrees with the induced strip-like end for sufficiently large $R$.

Assume that we have chosen domain dependent almost complex structures $\J_{L_0,L_1}$ for each pair of Lagrangians in the set of Lagrangians we are considering such that the moduli spaces of maps $\overline{\mathcal{M}\left(y_0,y_1,[\beta],\J\left(L_0,L_1\right)\right)}$ above are  compact manifolds of the expected dimension. 
For each set of Lagrangians $L_0, \ldots, L_n \in \{ L_\sigma \}_{\sigma \in \Sigma}$ with $L_0 > \ldots > L_n$, we now choose domain-dependent eventually constant almost complex structures $\J\left(L_0, \ldots, L_n\right): \overline{\mathcal{S_n}} \to \mathcal{J}_{EC}$ inductively (see \cite{seidelbook}), demanding compatibility:  
we require that $\J\left(L_0, \ldots, L_n\right)$ agree near each boundary puncture with the complex structure given by the strip-like coordinates, and that above $\overline{\mathcal{R}_{n+1}} \times \overline{\mathcal{R}_{k+1}}$, $\J\left(L_0, L_1,\cdots, L_{n+k}\right)$ be compatible with the decomposition of $\overline{\mathcal{S}_{n+k}}$. We also demand that in each fiber of $\overline{\mathcal{S}_{n+1}}$ over $\overline{\mathcal{R}_{n+1}}$ the almost complex structures be equal to a fixed almost complex structure away from a compact set. (This is an implicit assumption on all the families of eventually constant almost complex structures we will consider.) Note that for each such set $\left(L_0, \ldots, L_n\right)$, there exists a uniform $T$ such that for each $s \in \overline{\mathcal{S_n}}$,  $\J\left(L_0, \ldots L_{n}\right)_s$ is linear on $Y_{T}$.

Now, let $L_0, L_1, \ldots, L_n$ be such a collection of Lagrangians. Let $y_n \in L_{i-1} \cap L_{i}$ for $0 \leq i \leq n-1$, and let $y_{n} \in L_0 \cap L_n$. Let $[\beta] \in \pi_2\left(Y, L_0, \ldots, L_n\right)$.  Define $\overline{\mathcal{M}\left(y_0, \ldots, y_{n-1},y_{n}, [\beta], \J\left({L_0, \ldots, L_d}\right)\right)}$ to be the set of pairs $\left(r,u\right)$ where  $r \in \overline{\mathcal{R}}_{n+1}$ and $u$ is a map from the fiber in $\overline{\mathcal{S}_{n+1}}$ over $r$ satisfies which Equations \eqref{Jeqn1}, \eqref{Jeqn2}, \eqref{Jeqn3} for $\J\left(L_0, \ldots, L_n\right)$. For generic $\J\left(L_0, \ldots, L_n\right)$, this space is smooth and compact (where again we rely on Lemma \ref{RIE} to prove compactness).

\subsection{$A_{\infty}$ products} Throughout this subsection, assume that any collection $\left(L_0, \ldots, L_n, \J\right)$ is a collection of Lagrangians in $I$ with $L_0 > L_1 > \ldots > L_n$, and $\J=\J\left(L_0, \ldots, L_n\right)$ is a domain-dependent family of eventually constant complex structures so that the moduli spaces $\overline{\M}\left(L_0, \ldots, L_n, [\beta], \J\right)$ are well-defined.

Recall the definition of the Novikov field $\Lambda$:
\[ \Lambda = \left\{ \sum_{n \in \N} \xi_n t^{\lambda_n} \mid \xi_n \in \C, \lambda_n \in \R, \lambda_n < \lambda_{n+1}, \lambda_n \to \infty. \right \} \]
The Novikov field has a valuation $\mathfrak{v}: \Lambda \to \R$
\[ \val \left( \sum_{n \in \N} \xi_n t^{\lambda_n} \right) = \min_{n} \{ \lambda_n \mid \xi_n \neq 0 \} \]
which has the property that for all $a,b \in \Lambda$, $\val\left(a+b\right) \geq \min\left( \val(a), \val(b)\right)$ and $\val\left(ab\right) = \val(a) + \val(b)$.

For $L_0, L_1 \in \mathcal{G}_P\left(I\right)$ such that the pair $\left(L_0,L_1\right)$ is in correct position, define the following $\Lambda$ module: 

\[ CF^*\left(L_0, L_1\right) = \bigoplus_{y_n \in L_0 \cap L_1} \Lambda \cdot  y_n \]

This module has a $\Z$-grading which originates from the grading on the Lagrangians $L_0$ and $L_1$ (see \cite{AuB}).

\begin{remark} Since our Lagrangians lift to sections of the fibration $T^*\R^2 \to \R^2$, they have a canonical grading from the Morse indices of the critical points of their primitives; see \cite{SeGr}. \end{remark} 

\begin{defn}[Topological grading] \label{gradingrem}$CF^*\left(L_0, L_1\right)$ has an additional grading.
Fix lifts $\tilde{L}_0$ and $\tilde{L}_1$ of $L_0$ and $L_1$; recall that we define $\Gamma\left(L_0, L_1\right) \subset \Z^3$ by 
\[ \Gamma\left(L_0,L_1\right)=\left\{\hat{n} \in \Z^3 \mid \pi\left(\tilde{L}_0 \cap \tilde{L}_1\right) = \pi\left(\tilde{L}_0 \cap\left( \tilde{L}_1-\hat{n}\right)\right)\right\}.  \] 
Define a $\Z^3/\Gamma\left(L_0,L_1\right)$ grading on $CF^*\left(L_0,L_1\right)$  by setting
 \[ CF^*\left(L_0,L_1\right)_{\overline{\hn}} =\bigoplus_{y \in \pi\left(\tilde{L}_0 \cap \left(\tilde{L}_1-\hat{n}\right)\right)}  \Lambda y. \]
 For $y \in CF^*\left(L_0,L_1\right)$, we will use $y_{\overline{\hn}}$ to denote the projection of $y$ to the $\overline{\hn}$ graded piece. We will also write 
\[Y \supset \left(L_0 \cap L_1\right)_{\overline{\hn}} = \pi\left(\tilde{L}_0 \cap\left( \tilde{L}_1 -\hat{n}\right)\right). \]
\end{defn}

\subsection{Completions of the Floer complexes}\label{comps}  Throughout this subsection, let $L_0, L_1 \in \G_P$ be two Lagrangians in correct position. We define two completions of the Floer complex $CF^*(L_0,L_1)$; along the way we will use the following definitions:
\begin{defn}For $y \in L_0 \cap L_1$, we used the boldfaced $\mathbf{y}$ to denote the action-corrected generator $T^{-S(y)} y$. \end{defn} 

\begin{defn} $r \in \R$, $r>0$, and define the following $\C$-subalgebra of the Novikov field: 
\[ \Lambda_{r}= \left \{ \sum c_n T^{\lambda_n} | \lambda_n \in \R, \sum |c_n| e^{-2 \pi r \lambda_n} < \infty \right\}. \]
Define an evaluation map $\ev_\tau: \Lambda_r \to \C$ for every $\tau \in \C$ with $\im \tau \geq r$ by
\[ \ev_\tau \left( \sum c_n T^{\lambda_n} \right)=\sum c_n e^{2 \pi i \tau \lambda_n}. \] 
For $r >0$, define $\ev_{\underline{r}}=\ev_{i r}$. \end{defn}

 We first define a completion over $\Lambda$:
\begin{defn}
\begin{equation} \label{completion}
\widehat{CF}^*\left(L_0, L_1\right) =\left\{ \sum_{ y_i \in L_0 \cap L_1}  c_i y_i \, \mid \, c_i \in \Lambda, \lim_{t\left(y_i\right) \to \pm \infty}  \val\left(c_i\right) \to \infty \right\}.\end{equation} \end{defn}




\begin{defn} Define the following $\Lambda_r$-submodule of $\widehat{CF}^*\left(L_0,L_1\right)$: 
\[ \widehat{CF}_{\underline{r}}^*\left(L_0, L_1\right)=\left\{ \sum_{y_{i} \in L_0 \cap L_1} c_i  \mathbf{y}_i \mid c_i \in \Lambda_r, \lim_{|t\left(y_i\right)| \to \infty} |\ev_{\underline{r}}\left(c_i\right)|^{1/|t\left(y_i\right)|}  =0\right\}. \] 
\end{defn}

The periodicity of the Lagrangians $L_0$ and $L_1$ allows us to define an indexing by a subset of $\Z$ on the set of generators $y \in L_0 \cap L_1$ so that there exists constants $M, C>0$ such that $Mi-C<t(y_i) < Mi + C$. Then the following follows from the fact that a Laurent series $\sum_{i \in \Z} a_i z^i$ converges everywhere on $\C^*$ if and only if $\lim_{|i| \to \infty} |a_i|^{-1/i} = 0$:
\begin{prop}\label{convergence} \[ \sum_{y_{i} \in L_0 \cap L_1} c_i \textbf{y}_i \] defines an element of $\widehat{CF}_{\underline{r}}^*\left(L_0, L_1\right)$ if and only $\sum  \ev_{\underline{r}} \left(c_j\right) e^{-2 \pi c r t\left(y_j\right)}$ converges for all $c \in \R$. \end{prop}

\begin{defn} For $\tau \in \C$ with $\im \tau > 0$, define a graded $\C$-complex $\whCF^*_{\tau}(L_0, L_1)$ via the map $\ev_{\tau}$ on $\whCF^*_{\underline{\im \tau}}(L_0, L_1)$. \end{defn}


\subsection{Convergence of Floer products in the Novikov ring completion}
  \begin{prop}\label{finitenessofsum} \begin{enumerate}[(a)]
  \item Let $y_{i} \in L_{i} \cap L_{i+1}$, $0 \leq i \leq n-1$, $y_n \in L_0 \cap L_n$. Then the sum  
 \[ \sum_{[\beta] \in \pi_2\left(Y, L_0, L_1, \ldots, L_{n}\right)} T^{\omega\left([\beta]\right)}\# \overline{\mathcal{M}\left(y_0, \ldots y_{n-1}, y_n, [\beta],\J\right)}, \]
converges in $\Lambda$.

\item For a collection of intersection points $y_i \in L_i \cap L_{i+1}$, $0 \leq i < n$, the sum 
\begin{equation}\sum_{y \in L_0 \cap L_n, \left([\beta]\right) \in \pi_2\left(Y, L_0, \ldots, L_{n}\right)} T^{\omega\left([\beta]\right)}\# \overline{\mathcal{M}\left(y_{0}, \ldots, y, [\beta],\J\right)} y \end{equation}
converges in $\widehat{CF}^*\left(L_0, L_n\right).$
 \end{enumerate} \end{prop}

 \begin{proof}\begin{enumerate}[(a)] \item The statement that the sum converges can be rephrased as the following: for each $A >0$, there are only finitely many homotopy classes with $\omega\left([\beta]\right) < A$. This follows from Corollary \ref{boundedregion} and Gromov compactness.
 \item   Again let $A > 0$. Let $T_B=T_B\left(L_0, \ldots, L_n, \J\right)>0$ and $R = R\left(L_0, \ldots, L_n, \J\right) > 0$ be as in the statement of Lemma \ref{RIE}. Then all for all $y \in L_0 \cap L_n$ with \[ |t\left(y\right)|> \left(T_B + |t\left(y_0\right)|+A+1\right)\left(R +1\right),\] for all $u$ contributing to the sum 
\[ \sum_{[\beta] \in \pi_2\left(Y, L_0, L_1, \ldots, L_{n}\right)} T^{\omega\left([\beta]\right)}\# \overline{\mathcal{M}\left(y_0, \ldots y_{n-1}, y, [\beta],\J\right)}, \]
\[ \omega\left([u]\right) \geq \frac{\ell_{T_B}\left(u\right)}{R} \geq A. \]
Then
\[ \val\left(\sum_{[\beta]} T^{\omega\left([\beta]\right)}\# \overline{\mathcal{M}\left(y_0, \ldots y_{n-1}, p, [\beta],\J\right)}\right)  \geq \min_{[\beta]} \omega\left([\beta]\right) \geq A.\]
Thus as $|t\left(y\right)| \to \infty$, the valuation of the coefficient on $y$ becomes large, so the sum converges. \end{enumerate}
\end{proof}

For such a collection of intersection points $y_0, \ldots, y_{n-1}$, define
\[ \mu^n\left(y_{n-1}, \ldots, y_0\right) =  \sum_{y_n \in L_0 \cap L_n, [\beta] \in \pi_2\left(Y, L_0, L_1, \ldots, L_{n}\right)} T^{\omega\left([\beta]\right)}\# \overline{\mathcal{M}\left(y_0, \ldots y_{n-1}, y_n, [\beta],\J\right) y_n}. \]

\begin{prop} The $\mu^n$ can be extended linearly to  maps
\[ \mu^n : \widehat{CF}^*\left(L_{n-1}, L_n\right) \otimes \cdots \otimes \widehat{CF}^*\left(L_{0}, L_1\right)  \to \widehat{CF}^*\left(L_0, L_n\right) \]
for any collection $\left(L_0, \ldots, L_n, \J\right)$; that is, given a collection $w_i \in \widehat{CF}^*\left(L_{i}, L_{i+1}\right)$, $0 \leq i \leq n-1$, with
\[ w_i =  \sum_{y_{\left(m,i\right)} \in L_{i} \cap L_{i+1}}  c_{\left(m,i\right)} y_{\left(m,i\right)}, \]
the sum defined by 
\[ \sum_{m_0, \ldots, m_{n-1}} \left(c_{\left(m_0, 0\right)} \cdots c_{\left(m_{n-1}, {n-1}\right)}\right) \mu^n\left(y_{\left(m_{n-1},n-1\right)}, \ldots, y_{\left(m_{0},0\right)}\right) \]
converges in $\widehat{CF}^*\left(L_0, L_n\right) $.  \end{prop}
\begin{proof} 
Let $C \in \R$ be such that $C<0$ and
\[ C < \min_{\left(m,i\right)} \val \left(c_{\left(m, i\right)}\right); \]
then
\[  n C < \val \left( c_{\left(m_0, 0\right)}\cdots c_{\left(m_{n-i}, {n-1}\right)}\right) \]
for all collections $\left(m_0, 0\right), \ldots, \left(m_{n-i}, n-1\right)$. 
Let $y \in L_0 \cap L_n$. The coefficient of $y$ in the above sum is given by 
\[  \sum_{m_0, \ldots, m_{n-1}} \left(c_{\left(m_0, 0\right)}\cdots c_{\left(m_{n-i}, {n-1}\right)}\right) T^{\omega\left([\beta]\right)} \#\overline{\mathcal{M}\left(y_{\left(m_0,0\right)}, \ldots, y_{\left(m_{n-1},n-1\right)}, y, \beta, \J\right)}. \]
We first show that this converges in $\Lambda$ for all $y$; the proof is similar to the proof of part (b) in Proposition \ref{finitenessofsum}. Let $A>0$. Then for any collection of $\left(y_{\left(m_0,0\right)}, \ldots, y_{\left(m_{n-1},n\right)}\right)$ for which one of the points satisfies 
\[ |t\left(y_{\left(m_i, i\right)}\right)|> \left(T_B + |t\left(y\right)|+1+A\right)\left(R +1\right),\]
any $u: S \to Y$ contributing to the count
 \[ \sum_{[\beta]} T^{\omega\left([\beta]\right)}\# \overline{\mathcal{M}\left(y_0, \ldots y_{n-1}, y, [\beta],\J\right)}, \]
 satisfies
\[ \omega\left([u]\right) \geq A. \]
Then for such a collection $y_{\left(m_0,0\right)}, \ldots, y_{\left(m_{n-1},n-1\right)}$ of intersection points, 
\[ \val \left(\sum_{[\beta]} \left(c_{\left(m_0, 0\right)}\cdots c_{\left(m_{n-i}, {n-1}\right)}\right) T^{\omega\left([\beta]\right)} \#\overline{\M\left(y_{\left(m_0,0\right)}, \ldots, y_{\left(m_{n-1},n-1\right)}, y, \beta, \J \right)}\right) \geq A + n C. \]
This condition is satisfied by all but finitely many such collections of endpoints, and we conclude that the coefficient on each $y \in L_0 \cap L_n$ converges.

Now again let $A > 0$. Let $t_A$ be such that if $|t\left(y_{\left(m,i\right)}\right)|> t_A$, $c_{\left(m,i\right)} > A$. There are finitely many $y_{\left(m,i\right)}$ with $|t \left(y_{\left(m,i\right)}\right) |< t_A$, so there exists $t_A'$ such that if a $\J$-holomorphic curve $u$ contributing to the sum has $u\inv\left(Y_{t_A'}\right)\neq \emptyset$ and $u\left(p_i\right) =y_{\left(m,i\right)}$ for some such $y_{\left(m,i\right)}$, then $\omega[u]> A$. Now let $y \in L_0 \cap L_n$ have $|t\left(y\right) |> t_A'$. Then we can separate the  sum calculating the coefficient of $y$ into two parts, where one part is the sum over curves where all boundary marked points go to points close to $y$, and the other is the sum over the other curves, which are forced to have large image when projected to $\R$ under $t: Y \to \R$ and thus large area.
\begin{align*}  &\sum_{m_0, \ldots, m_{n-1}} \left(c_{\left(m_0, 0\right)}\cdots c_{\left(m_{n-i}, {n-1}\right)}\right) T^{\omega\left([\beta]\right)} \#\overline{\M\left(y_{\left(m_0,0\right)}, \ldots, y_{\left(m_{n-1},n\right)}, p, \beta, \J\right)} \\
& = \sum_{|t\left(y_{m_i}\right) |< t_A \text{ for all } i} \left(c_{\left(m_0, 0\right)}\cdots c_{\left(m_{n-i}, {n-1}\right)}\right) T^{\omega\left([\beta]\right)} \# \overline{\M\left(y_{\left(m_0,0\right)}, \ldots, y_{\left(m_{n-1},n\right)}, p, \beta, \J\right)} + \\
& \quad   \sum_{|t\left(y_{m_i}\right)| \geq t_A \text{ for some } i } \left(c_{\left(m_0, 0\right)}\cdots c_{\left(m_{n-i}, {n-1}\right)}\right) T^{\omega\left([\beta]\right)} \#\overline{\M\left(y_{\left(m_0,0\right)}, \ldots, y_{\left(m_{n-1},n\right)}, p, \beta, \J\right)}. \end{align*}
Label the first sum $\Sigma_1$, and the second sum $\Sigma_2$. Each term contributing to $\Sigma_1$ has valuation at least $A + n C$ by the condition that $t\left(y\right) > t_A'$, and every term contributing to $\Sigma_2$ has valuation at least $A+nC$ by the condition on $t_A$. We conclude that 
\[\val\left(c\left(y\right)\right) \geq \min \left\{\val\left(\Sigma_1\right), \val\left(\Sigma_2\right)\right\} \geq A+n C.\] Since there are only finitely many $y \in L_0 \cap L_n$ with $|t\left(y\right)| < t_A$, we conclude that there are only finitely many terms with valuation less than $A+nC$. The convergence of the sum follows. 
\end{proof}

\subsubsection{ $A_\infty$ relations} Since we have excluded disc and sphere bubbling by the assumptions on our Lagrangians, the $\mu^n$ products will satisfy the $A_\infty$ relations (see \cite{seidelbook}) for any collection of perturbed Hamiltonians in correct position with compatible choices of data in the definition of the moduli spaces.

\begin{remark}\label{overZ} In fact, the $A_{\infty}$ relations hold for complexes defined over $\Lambda_{\Z}$, the Novikov ring with coefficients in $\Z$, since the counts of holomorphic discs of any given class are valued in $\Z$. \end{remark}

\begin{remark} Let $\whHF^*\left(L_0, L_1\right)$ denote the cohomology of the complex $\whCF^*\left(L_0, L_1\right)$.  Note that $\mu^2$ induces an associative product $\whHF^*\left(L_1, L_2\right) \otimes \whHF^*\left(L_0,L_1\right) \to \whHF^*\left(L_0, L_2\right)$; we write this product as $[a]\cdot [b] = [\mu^2\left(a,b\right)]$. \end{remark}

\subsection{Continuation elements}

We relate our Floer complexes via continuation elements, with the aim of formally inverting the morphisms given by multiplication with these elements when we define our Fukaya categories. Given $L$ an admissible planar Lagrangian, and $H_0,H_1: Y \to \R$ Hamiltonians suitable for $L$ so that the pair $\left(\psi_{H_1}\left(L\right), \psi_{H_0}\left(L\right)\right)$ is in correct position, one would like to define continuation elements
 \[[c_{H_0 \to H_1}] \in \whHF^*\left(\psi_{H_1}\left(L\right), \psi_{H_0}\left(L\right)\right) \]  by counting maps with moving boundary conditions that interpolate between $\psi_{H_0}\left(L\right)$ and $\psi_{H_1}\left(L\right)$. To implement this, we use moduli spaces of such curves in two situations where those moduli spaces are well behaved.

\subsubsection*{Compactly supported Hamiltonians} 
Let $L,L' \in \G_P$, and let $H$ be a compactly supported Hamiltonian on $Y$ suitable for $L$ such that the pairs $\left(L, L'\right)$ and $CF^*\left(\psi_{H}\left(L\right), L'\right)$ are in correct position. We define a continuation map 
\[ \Phi_{H}: \whCF^*\left(L, L'\right) \to \whCF^*\left(\psi_H\left(L\right), L'\right). \]
Let $S$ denote the strip $\R \times [0,1]$. Let $\rho: \R \to [0,1]$ be a surjective increasing function such that $\rho\left(s\right)=0$ for $s << 0$, and $\rho\left(s\right) =1$ for $s >> 0$. Let $\beta \in \Omega^1\left(S\right)$ be a one-form satisfying 
\begin{align*} \beta |_{\R \times \{1\} } &= d \rho \\
\beta |_{ \R \times \{0\}  } &= 0. \end{align*}
Fix a domain-dependent eventually constant complex structure $\J$ on $Y$.

For $y_0 \in \psi_{H_0}\left(L\right)$, $y_1 \in \psi_{H_1}\left(L\right)$, consider the set of maps $u: S \to Y$ such that 
\begin{align} \label{Jholpert} \left(d u - X_{H} \otimes \beta\right)_\J^{0,1} &= 0 \\ 
			\label{bdry} u\left(s,1\right) &\in \psi_{\rho\left(s\right)H}\left(L\right)\\
			\label{bdry1} u\left(s,0\right) &\in L \end{align}
with
	\begin{align} \label{limity} \lim_{s \to -\infty} u\left(s, \cdot\right) &= y_{0}\\
	 \lim_{s \to \infty} u\left(s, \cdot\right)&= y_1.\end{align}
For such $u$, define the geometric and topological energy of such maps:
\[ E^{\text{geom}}\left(u\right) = \int_{S}\frac{1}{2} \| d u - X_H \otimes \beta \|^{2} = \int_{S} u^* \omega - d \left(u^*H \right) \wedge \beta; \]
\[ E^{\text{top}}\left(u\right) = \int_{S} u^* \omega - d \left(u^* H \beta\right) = E^{\text{geom}}\left(u\right) - \int_S u^*H d \beta.    \]	

We now use the assumption that our Hamiltonians are exact when lifted to the universal cover to  bound the topological energy of $u$: lifting $u$ to the universal cover and expressing $\omega = d \theta$, $\theta |_{L'} = d h$, $\theta |_{L'} = d h'$ for some functions $h,h': \R^2 \to \R$, we use Stokes' theorem to write
\begin{align*} &\int_S u^* \omega - d \left(u^* H \beta\right)  \\ 
 & = - \int_{\p S} u^* \theta - u^*H \beta \\
 &=  -\int_{\R \times \{0\}} u^*d h' - \int_{\R \times \{1\}} \left(u^*\left(d h\right) + \rho u^*\left(d H\right) + u^*H d \rho\right) \\
 & =- \int_{\R \times \{0\}} u^*d h ' + \int_{\R \times \{1\}}\left(u^*\left(d h\right) +\rho u^*\left(d H\right)\right) \\
 &= h'\left(y_0\right)-h'\left(y_1\right)-h\left(y_0\right)+h\left(y_1\right)+H\left(y_1\right)\\
 &= S_{\R^4} \left(\tilde{y}_0\right) -S_{\R^4} \left(\tilde{y}_1\right). \end{align*}
 Note that the quantity $S_{\R^4}\left(\tilde{y}_0\right)-S_{\R^4}\left(\tilde{y}_1\right)$ is independent of the choice of lift of $\tilde{y}_0$. Denote this quantity $S\left(y_0,y_1\right)$.
 Then we obtain a bound on the geometric energy of $u$:
\begin{equation}\label{geomenbd} |E^{\text{geom}}\left(u\right)| - |S\left(y_0, y_1\right)| \leq  \left|\int_S u^*H d \beta\right| \leq \max_{y \in Y} |H| \int_{\p S} \beta = \max |H|. \end{equation}
Away from the support of $H$, Equation \eqref{Jholpert} reduces to the unperturbed Cauchy-Riemann equation, so in this region, a solution $u$ will satisfy the maximum principle. Moreover, in this region, the Hamiltonian isotopy is constant. Thus the methods of Lemma \ref{RIE} can be used in combination with the geometric energy bound above to produce a compact region containing the image of any such $u$. Then for generic $\J$, the moduli space $\M\left(y_0,y_1, \beta, \J\right)$ of such maps is a smooth manifold with dimension depending on the degrees of $y_0$ and $y_1$ as generators of the Floer complex, and this manifold admits a compactification $\overline{\M\left(y_0,y_1,\beta, \J\right)}$.

Define $\Phi_H$ on the generators of $CF^*\left(\psi_{H_0}\left(L\right), L\right)$ by
\[ \Phi_H\left(y\right) = \sum_{y_1 \in \psi_{H_1}\left(L\right) \cap L, u \in \overline{\M\left(y,y_1,\beta, \J\right)} } T^{E^{\mathrm{geom}}\left(u\right)}y_1 \]
where as usual the sum is over all $y$ such that $\overline{\M\left(y,y_1,\beta, \J\right)}$ is zero-dimensional, that is, the generators $y_1 \in CF^*\left(\psi_{H_1}\left(L\right), L\right)$ with the same degree as $y$.

\begin{lemma}\label{comph} For any compactly supported $H$,  $L, L' \in \G_P$ with $\left(L,L'\right)$ and $\left(\psi_{H}(L), L'\right)$ in correct position, and corresponding map $\Phi_{H}$ defined on the generators of $CF^*\left(L,L'\right)$ as above:
\begin{enumerate} [(a)]
\item There exists $T > 0$ with $\supp H \subset Y_T$ such that $\Phi_{H} \left(y\right)= y$ for all $y \in \left(L \cap L'\right)_T$.
\item $ \Phi_H$ extends to a map $\whCF^*\left(L, L'\right) \to \whCF^*\left(\psi_{H}\left( L\right), L'\right)$.
\item $\Phi_{H}$ extends to a map  $\whCF_{\underline{r}}^*\left(L, L'\right) \to \whCF_{\underline{r}}^*\left(\psi_{H}\left( L\right), L'\right)$.
\item $\Phi_{H}$ is a chain map.
\item $\Phi_{H}$ is a quasi-isomorphism, with quasi-inverse $\Phi_{-H}$.
\item For two compactly supported Hamiltonians $H_0$ and $H_1$, $\Phi_{H_1} \circ \Phi_{H_0}$ agrees with $\Phi_{H_1+H_0}$ on cohomology.  \end{enumerate} \end{lemma}
\begin{proof}  By choosing lifts appropriately, we can define the topological $\Z^3/\Gamma(L,L')$ grading of Definition \ref{gradingrem} on the two Floer complexes so that the continuation map must respect the grading. Then for each $\hn \in \Z^3/\Gamma$, we get an induced map 
\[ \Phi_{H}^\hn: CF^*\left(L, L'\right)_\hn \to CF^*\left(\psi_{H}\left(L\right), L' \right)_\hn \]
For all $\hn$, the complexes $CF^*\left(L, L'\right)_\hn$ and  $CF^*\left(\psi_{H}\left(L\right),  L'\right)_\hn$ have finitely many generators. Thus on each graded piece, (d), (e) and (f) hold by standard arguments. Thus it suffices to prove (a),(b), and (c).

From Equation \eqref{geomenbd}, we know that the energy of any strip connecting generators $y_0 \in CF^*\left(L, L'\right)_\hn$ and $y_1 \in CF^*\left(\psi_{H}\left(L\right), L'\right)_\hn$ is bounded by \[ |S\left(y_0,y_1\right)|+\max |H|.\] By Lemma \ref{boundondiff}, we know that there exists $D> 0$ and $\alpha \in \R$ such that 
\[ |S\left(y_0,y_1\right)| < | S\left(y_0\right)-S\left(y_1\right)| + D =\left| Q^\alpha_{L,L', \pm}\left(y_0\right) - Q^\alpha_{L,L', \pm}\left(y_1\right)\right| + D< \left| \frac{1}{2}\alpha\left(t\left(y_0\right)^2 - t\left(y_1\right)^2\right)\right| + D \]
But $\alpha\left(t\left(y_0\right)^2 - t\left(y_1\right)^2\right)$ is uniformly bounded for all such pairs $\left(y_0, y_1\right)$.  So we obtain a uniform bound on the geometric energy $E^{\text{geom}}\left(u\right)$ for all $u$ counted by $\Phi_{H}$.

Away from the support of $H$, Equation \eqref{Jholpert} reduces to the non-perturbed $J$-holomorphic equation. In this region methods of Lemma \ref{RIE} will apply, so we obtain a similar bound of the form
\[ \ell_{T_B} \left(u\right) \leq R E^{\text{geom}}\left(u\right). \]
for some $T_B $ so that $\supp H \cap Y_{T_B} = \emptyset$. Then using the uniform bound on $E^{\text{geom}}\left(u\right)$, we can find $T > T_B$ such that for all $y_0 \in \left(L \cap L'\right)_T$ and maps $u \in \overline{\M\left(y_0, y_1, \beta\right)}$, the image of $u$ must be disjoint from the support of $H$. But in this region, any $u: \R \times [0,1]$ counted by the continuation map must be constant; any non-constant $u$ can re-parametrized by translation by $s$, and thus will come in a higher-dimensional family. Thus the continuation map is the identity on generators outside of a compact set, so (b) follows.

Now, note that for any $y \in L \cap L'$, $\Phi_n\left(y\right)$ counts only finitely many $u$ since the moduli spaces $\overline{\M\left(y,y_0,\beta\right)}$ are empty for all but finitely many $n$, and each zero dimensional moduli space is compact. Thus for any $T$, for a collection $c_i \in \Lambda_r$, 
\[ \Phi_n\left(\sum_{y_i \in Y \setminus Y^T} c_i y_i \right) \in CF_R^*\left(\psi_{H}\left(L\right), L'\right)\subset  \whCF_{\underline{r}}^*\left(\psi_{H}\left(L\right), L'\right).\]
Combining this with (a) shows that that $\Phi_{H}$ extends to a map  $\whCF^*\left(L, L'\right) \to \whCF^*\left(\psi_{H}\left( L\right), L'\right)$. Finally, note that for any generator which is stationary under the flow of $H$, the action of that generator remains unchanged. Thus the extension restricts to a map  $\whCF_{\underline{r}}^*\left( L, L'\right) \to \whCF_{\underline{r}}^*\left(\psi_{H}\left( L\right), L'\right)$. 
\end{proof}


\subsubsection*{Hamiltonians with a unique minimum} 
Let $L \in \G\left(\Wr\right)$, and let $H_0$ and $H_1$ be two Hamiltonians suitable for $L$ such that the function $H_1 -H_0 = H$ has a unique minimum at $t=s=0$ as a function on $\R \times S^1$ with $H\left(0\right)=0$, and such that the pair $\left(\psi_{H_1}\left(L\right), \psi_{H_0}\left(L\right)\right)$ is in correct position.

 Define the continuation element $c_{H_0 \to H_1} \in CF^*\left(\psi_{H_1}\left(L\right), \psi_{H_0}\left(L\right)\right)$ by
\[c_{H_0 \to H_1} =\ \sum_{y_n \in L \cap \{t=s=0\}} y_n, \]
i.e., by the sum of the points fixed by the Hamiltonian isotopy $\psi_{H}^t\left(L\right)$ at the minimum of $H$. We use the notation $c_{H}$ to denote the continuation element $c_{0 \to H}$ here and throughout.


\subsubsection*{Perturbing Hamiltonians}
Let $L$ be an admissible planar Lagrangian, and let $H_0,H_1: Y \to \R$ be perturbing Hamiltonians suitable for $L$ so that the pair $\left(\psi_{H_1}\left(L\right), \psi_{H_0}\left(L\right)\right)$ is in correct position. Let $H= H_1 - H_0$, so $\psi_{H}\left(\psi_{H_0}\left(L\right)\right)=\psi_{H_1}\left(L\right)$. We now fix a way to decompose $H$ into a compactly supported Hamiltonian and a Hamiltonian which has a unique minimum as a function $H: \R \times S^1 \to \R$ at $\left(t,s\right) = \left(0,0\right)$.

Recall that we can write $H$ at the two ends of $Y$ as 
\[ H_{\pm}\left(t, s\right) =\alpha_{\pm} t^2 + \epsilon_\pm | t |+ f_\pm\left(t,s\right). \]
Let $a= \alpha_+ + \epsilon_+ = \alpha_- + \epsilon_-.$ Let $ \sigma_{\alpha_-, \epsilon_-}^{\alpha_+, \epsilon_+}, \chi, \rho: R \to \R$ be extrapolating functions as described in Definition \ref{intfun}. Define
\begin{equation} \label{hamform} H_Q\left(t,s\right) = \sigma_{\alpha_-, \epsilon_-}^{\alpha_+, \epsilon_+}\left(t\right) + \frac{a \chi\left(t\right)\left(\cos\left(2 \pi s\right) +1\right)}{8 \pi} +\rho_0^1\left(-t+1\right)f_-\left(t,s\right) + \rho_0^1\left(t-1\right)f_+\left(t,s\right).
\end{equation}

Note that the map $H \mapsto H_Q$ is linear in $H$.  We have
\begin{align*} &\frac{\p H_Q \left(t,s\right)}{\p t}\\
& = \begin{dcases} 2 \left(\alpha_+ + \epsilon_+\right) t & t \in \left[0,1/2\right); \\
			      \frac{d  \sigma_{\alpha_-, \epsilon_-}^{\alpha_+, \epsilon_+}\left(t\right)}{dt} + \frac{a\left(\sin\left(2 \pi s\right)+1\right)}{8 \pi}\frac{d \chi\left(t\right)}{d t} + \frac{\p}{\p t} \left( \rho_0^1\left(t-1\right) f_+\left(t,s\right)\right) & t \in \left[1/2,2\right);\\
			      2 \alpha_+ t + \epsilon_+ + \frac{ \p f \left(t,s\right)}{\p t} & t \in \left[2, \infty\right). \end{dcases}\end{align*}
Recall that $\|f_+\|_{C^2} < \alpha/8$, and $\alpha_+, \epsilon_+ > 0$; from this we see that the partial derivative is non-vanishing for $t \in [0,1/2)$ and $t \in [2, \infty) $. We can also bound the partial derivative away from $0$ for $t \in [1/2,2)$:
\begin{align*} 
\left |\frac{\p H\left(t,s\right)}{\p t} \right| =&\left|    \frac{d  \sigma_{\alpha_-, \epsilon_-}^{\alpha_+, \epsilon_+}\left(t\right)}{dt} + \frac{a\left(\sin\left(2 \pi s\right)+1\right)}{8 \pi}\frac{d \chi\left(t\right)}{d t} + \frac{d \rho_0^1\left(t-1\right)}{d t} f_+\left(t,s\right) +\left( \frac{\p f_+\left(t,s\right)}{\p t} \right) \rho_0^1\left(t-1\right)   \right| 
\\
& \geq  \frac{ 2 \alpha_+t  + \epsilon_+}{2}   -  \frac{\alpha_+  + \epsilon_+}{4 \pi}  -  2 \|f_+\|_{C^2} > 0.
\end{align*}
From this we conclude that the partial derivative is non-vanishing for $1 \leq t \leq 2$. The same analysis for the $t< 0$ region will show that the partial derivative is non-vanishing for $t < 0$. We conclude that the unique minimum of $H$ occurs at $\left(t,s\right)=\left(0,0\right)$. 

 Let $H_K$ be the compactly supported Hamiltonian such that
\[ H = H_Q + H_K. \] We define a continuation element $c_{H} \in {CF}^*\left(\psi_{H}\left(L\right), L\right)$ by
\[ c_{H} = \Phi_{H_K}\left(c_{H_Q}\right). \]

 \subsubsection{Properties of the continuation elements}
\begin{lemma}\label{decompofh}
\begin{enumerate}[(a)]
\item For all Hamiltonians $H$ perturbing $L \in \G_P$,  $\mu^1\left(c_H\right)=0$. 
\item Let $H$ and $H'$ be two Hamiltonians perturbing $L$ in $\G_P$. Then $\left(H+H\right)_Q=H_Q+H'_Q$ and 
\[ \mu^2\left(c_{H_Q}, c_{H'_Q}\right)=c_{\left(H+H'\right)_Q}.\]
\item Let $H_0, H_1,$ and $H_2$ be Hamiltonians such that the set $\left(\psi_{H_2}\left(L\right), \psi_{H_1}\left(L\right), \psi_{H_0}\left(L\right)\right)$ is in correct position. Then in $HF^*\left(\psi_{H_2}\left(L\right), \psi_{H_0}\left(L\right)\right),$
\[ [c_{H_0 \to H_1}] \cdot  [c_{H_1 \to H_2}]= [c_{H_0 \to H_2}]. \]
\end{enumerate} \end{lemma}
\begin{proof} \begin{enumerate}[(a)] 
\item Since $\Phi_{H_K}$ is a chain map, it suffices to show that $\mu^1\left(c_{H_Q}\right)=0$. Note that we can compute the differential $\mu^1\left(c_{H_Q}\right)$ in any cover of $Y$. Lift $\psi_{H}\left(L\right)$ and $\psi_{H_K}\left(L\right)$ to $T^*C_R$, where $r$ is defined using the class $[\gamma_L] = [\left(r, d k, dm\right)]$; the lifts of $\psi_{H}$ and $\psi_{H_I}\left(L\right)$ are then sections of the fibration $T^* C_R \to C_R$. Let $G: C_R\to \R$ be the compactly supported function that appears in the primitive of $L \in T^* \R^2$ in Lemma \ref{formofh}.   The symplectomorphism 
\[ \psi_{-G-H_K}: T^*C_R \to T^* C_R. \]
induces an isomorphism of chain complexes
\[  \whCF^*\left(\psi_{H}\left(L\right), L\right) \to \whCF^*\left(\psi_{H}\left(\psi_{-G-H_K}\left(L\right)\right), \psi_{-G-H_K}\left(L\right)\right). \]

Note that $\psi_{H}\left(\psi_{-G-H_K}\left(L\right)\right)$ and $\psi_{-G-H_K}\left(L\right)$ are both product Lagrangians when restricted to the product submanifold \[\left([-1/2,1/2] \times \R, d t \wedge d \phi_t\right) \oplus  \left(S^1_R \times \R, d s \wedge d \phi_s\right) \subset T^*\left(\R \times S^1_r\right).\] Choose an almost complex structure $\J$ which respects this splitting; recall that the differentials on the Floer complexes can be computed up to homotopy using any almost complex structure that ensures regularity. Any holomorphic disc that contributes to  $\mu^1\left(c_H\right)$ can be constrained to lie in \[ \left([-1/2,1/2] \times S^1, d t \wedge d \phi_t\right) \oplus  \left(S^1_R \times S^1, d s \wedge d \phi_s\right) \subset Y_R\] by using the similar methods of the proof of Lemma \ref{RIE} to obtain \emph{a priori} constraints on the image of $u$ based on $\omega\left(u\right)=H_Q\left(y'\right)$, and then writing $H_Q =  \left(\frac{H_Q}{n}+\ldots+ \frac{H_Q}{n}\right)$, choosing $n$ large and applying (a). But in this region all intersection points lie at $t=\phi_t=0$, so applying the maximum principle to the holomorphic map $\pi_{\left(t, \phi_t\right)} \circ u$ allows us to deduce that $\pi_{\left(t, \phi_t\right)} \circ u$ is constant. Thus any $u$ must split as \[u = \{0\} \times\left( \pi_{s, \phi_s} \circ u\right).\] Such maps are automatically regular since all holomorphic maps to $\C$ are regular. Then the problem reduces to showing that the sum of the index 0 intersection points is closed in $T^*S^1_R$. Labelling the intersection points by their $s$ value, we have \[\psi \left(y_{n}\right) = \left(T^{A} \right)y_{n+1/2} - \left(T^A\right) y_{n-1/2},\] for $n \in \Z/ R \Z$, where $A = H\left(1/2,0\right)-H\left(0,0\right)$. Summing over all the intersection points gives the result.

\item The map $H \mapsto H_Q$ is linear in $H$, so the first equality follows.

Note that $H_Q$ and $H_Q+H_Q'$ each have minima only at $t=s=0$. We can use the exactness of $L$ when lifted to $\R^4$ to express the energy of any $u: \overline{\D} \to Y$ contributing to $\mu^2\left(c_{H_Q}, c_{H_Q \to H_Q+ H'_Q}\right)$ with \[u\left(p_0\right)=y_0, u\left(p_1\right)=y_1, u\left(p_2\right)=y_2\] as \[\int_{\D} u^* \omega = -(H_Q+H_Q')\left(y_2\right).\] Since $H_{Q+Q'}\left(y_2\right) \geq 0$, we must have that $\int_{\D} u^* \omega =0$, so such $u$ must be constant; thus $y_0=y_1=y_2$ must be a fixed point of the isotopy $\phi_{H_2}^t\left(L\right)$. We conclude that 
\begin{align*} \mu^2\left(c_{H_Q}, c_{H_Q \to H_Q+ H'_Q}\right) = \sum_{y \in L \cap \{t = s = 0\}} y \quad = c_{H_0 \to H_2}.\end{align*}

\item By gluing the domain of the maps counted by \[\Phi_{\left(H_1-H_0\right)_K}\left(c_{H_0 \to H_0+\left(H_1-H_0\right)_Q}\right);\]  \[\Phi_{\left(H_2-H_1\right)_K}\left(c_{H_1 \to H_1 + \left(H_2-H_1\right)_Q}\right),\] the space of discs $u: S \to Y$ counted by $\mu^2\left(c_{H_0 \to H_1}, c_{H_1 \to H_2} \right)$ can be deformed to the space of maps counted by $\Psi_{(H_2)_K}\left(c_{\left(H_2\right)_Q}\right)$.

\end{enumerate}
\end{proof}

\begin{lemma}\label{qisomcont}
Let $H_1$ and $H_0$ be type-$I$ Hamiltonians perturbing a type-$I$ Lagrangian $L$, and let $H'$ be a type-I Hamiltonian perturbing a type-$I$ Lagrangian $L'$. Assume that the set \[\left(\psi_{H_1}\left(L\right), \psi_{H_0}\left(L\right), \psi_{H'}\left( L'\right)\right)\] is in correct position. Also assume that $\alpha_{\pm}\left(H_0\right) - \alpha_{\pm} \left(H'\right) > K_{\pm}\left( L, L'\right)$ on the wrapped end(s) of $Y$. Then the multiplication map defines a map 
\[ \mu^2\left( \cdot, c_{H_0 \to H_1}\right): CF^* \left(\psi_{H_0} \left( L\right), \psi_{H'} \left(L'\right) \right)\to  CF^* \left(\psi_{H_1} \left(L\right),\psi_{H'} \left(L'\right)\right) \]
which is a quasi-isomorphism. Similarly, if the set $\left(\psi_{H'}\left( L'\right), \psi_{H_1}\left(L\right), \psi_{H_1}\left(L\right)\right)$ is in correct position and  $\alpha_{\pm} \left(H'\right) - \alpha_{\pm}\left(H_0\right)  > K_{\pm}\left(L', L\right)$ on the wrapped end(s) of $Y$ then the map 
\[ \mu^2\left(c_{H_1 \to H_0}, \cdot\right): CF^*\left(\psi_{H'}\left( L'\right), \psi_{H_0}\left(L\right)\right) \to  CF^* \left(\psi_{H'} \left(L'\right), \psi_{H_1} \left(L\right)\right) \]
is a quasi-isomorphism. 

Moreover, the quasi-isomorphisms defined above extend to maps 
\begin{align*} & \mu^2\left( \cdot, c_{H_0 \to H_1}\right): \whCF^* \left(\psi_{H_0} \left( L\right), \psi_{H'} \left(L'\right) \right)\to  \whCF^* \left(\psi_{H_1} \left(L\right),\psi_{H'} \left(L'\right)\right), \\
 & \mu^2\left(c_{H_1 \to H_0}, \cdot\right): \whCF^*\left(\psi_{H'}\left(L'\right), \psi_{H_0}\left(L\right)\right) \to  \whCF^* \left(\psi_{H'} \left(L'\right), \psi_{H_1} \left(L\right)\right) \end{align*}
which are injective on cohomology, and restrict to quasi-isomorphisms
\begin{align*} & \mu^2\left( \cdot, c_{H_0 \to H_1}\right): \whCF_{\underline{r}}^* \left(\psi_{H_0} \left( L\right), \psi_{H'} \left(L'\right) \right)\to  \whCF_{\underline{r}}^* \left(\psi_{H_1} \left(L\right),\psi_{H'} \left(L'\right)\right), \\
 & \mu^2\left(c_{H_1 \to H_0}, \cdot\right): \whCF_{\underline{r}}^*\left(\psi_{H'}\left(L'\right), \psi_{H_0}\left(L\right)\right) \to  \whCF_{\underline{r}}^* \left(\psi_{H'} \left(L'\right), \psi_{H_1} \left(L\right)\right). \end{align*}
\end{lemma}

\begin{proof} We prove the first case; the second can be proved with the same reasoning. Change notation slightly: let $H_1 - H_0 = H$, and let $L$ now denote the perturbed Lagrangian $\psi_{H_0}\left(L\right)$ and $L'$ now denote  the perturbed Lagrangian $\psi_{H'} \left(L'\right)$. 

We have seen that if $G$ is a compactly supported Hamiltonian, the map
\begin{align*} \Phi_{G}: \whCF^*\left(L, L'\right) \to \whCF^*\left(\psi_{G}\left(L\right), L'\right)
       \end{align*}
is a quasi-isomorphism. Using the decomposition $[c_H]=\Psi_{H_K}\left(H_Q\right)$, we can reduce to the case when $H=H_Q$, so $H$ is a Hamiltonian of the form given by Equation \eqref{hamform}.
 
Note that the multiplication with the continuation element respects the $\Z^3/\Gamma$ grading on the Floer complexes. We first show that 
 \[ \cdot [c_H]_{\hat{n}} : HF^*\left(L,L'\right)_{\hn} \to HF^*\left(\psi_{H}\left(L\right), L'\right)_{\hat{n}} \]
 is an isomorphism for each $\hat{n} \in \Z^3$. 

Recall that we associate to $L$ and $L'$ elements $[\gamma_L], [\gamma_{L'}] \in H_1\left(T^3\right)$ of the form $\left(r, d k, d m\right)$ and $\left(r', d' k', d' m'\right)$. Let $R=r r'$, and consider the cyclic cover $Y_{R} \to Y$ of Remark \ref{cycliccovers}. Let $L_{R}$ and $L'_R$  be lifts of $L$ and $L'$ to $Y_{R}$, respectively, which lift further to $\tilde{L}$ and $\tilde{L}'$; note that both $L_R$ and $L_R'$ are sections of the fibration $Y_R \to C_R$. Let \[c_{H_R} \in \whCF^*\left(L_R, L_R'\right) = \pi_R \inv\left(c_H\right).\] The natural identifications 
\begin{align*} CF^*\left(L,L'\right)_{\hat{n}} &\cong CF^*\left(L_R, L_R'\right)_{\hat{n}} \\
CF^*\left(\psi_{H}\left(L\right),L'\right)_{\hat{n}}& \cong CF^*\left(\psi_{H}\left(L_R\right), L_R'\right)_{\hat{n}}\end{align*}
 commute with the induced map  $\cdot [c_{H_R}]_\hn : HF^*\left(L_R,L'_R\right)_{\hn} \to HF^*\left(\psi_{H}\left(L_R\right), L'_R\right)_\hn$. Thus it suffices to show that this second map is an isomorphism.

Let $\hat{n} \in \Z^3$. Let $h$ be such that $\tilde{L}$ is the graph of $d h$, and let $h'$ be such that $\tilde{L}'-\hn$ is the graph of $d h'_{\hn}$. Note that $\whCF^*\left(L,L'\right)_\hn$ is generated by points of $L_R \cap L_R'$ that lie above critical points of $h - h'$. Assume now that that $h-h'$ and $h-h'+H$ have a critical point at $t=s=0$, and at that point both functions vanish. Let $y_0 \in \whCF^*\left(L_R,L'_R\right)_{\hn}$ and $y_1 \in \whCF^*\left(L,L_R'\right)_{\hn}$, denote the  corresponding generators above this critical point and assume that $y_0$ and $y_1$ have the same degree. Then the constant solution is the only disk which contributes to the product $\mu^2\left(y_0, c_{H_R}\right)$, since all of the disks contributing to this product must have symplectic area zero. We conclude that $\mu^2\left(y_0, c_{H_R}\right)=y_1$.

Note that for $i \neq 0,1,2$, the Floer complexes vanish in degree $i$. We will show that for $i=0,1,2$, there is a compactly supported Hamiltonian $G: Y_R \to \R$ so that the functions $h-h'-G$ and $h-h'+H-G$ have this property, and such that the corresponding generators $y_0 \in CF^i\left(L_R,\psi_{G}\left(L_R'\right)\right)_{\hn}$ and $y_1 \in CF^i\left(\psi_{H}\left(L_R\right),\psi_{G}\left(L_R'\right)\right)_{\hn}$ are closed and  span  rank-1 modules $ HF^i\left(L_R,\psi_{G}\left(L_R'\right)\right)_{\hn}$ and $HF^i\left(\psi_{H}\left(L_R\right),\psi_{G}\left(L_R'\right)\right)_{\hn}$ respectively. Then the map \[\cdot [c_{H_R}]_\hn : HF^*\left(L_R,\psi_{G}\left(L'_R\right)\right)_\hn \to HF^*\left(L_R,\psi_G\left(L'_R\right)\right)_\hn\] is an isomorphism. We then use the commutativity of the below diagram to conclude that $\cdot [c_H]_ \hn$ is an isomorphism: 
     \[\begin{tikzcd}
     HF^i\left(\tilde{L}, \tilde{L}'\right)_{\hn} \arrow{r}{\cdot [c_H]_\hn} \arrow{d}{\Phi_{G}}[swap]{\cong}& HF^i\left(\psi_{H}\left(\tilde{L}\right), \tilde{L}'\right)_{\hn} \arrow{d}{\Phi_{G}}[swap]{\cong}\\
     HF^i\left(L_R, \psi_{G}\left(L'_R\right)\right)_{\hn} \arrow{r}{\cdot [c_{H_R}]_\hn} & HF^i\left(\psi_{H}\left(L_R\right), \psi_{G}\left(L_R\right)\right)_{\hn}. \end{tikzcd}\]
We distinguish the cases $[\gamma_L] = [\gamma_{L'}]$ and $[\gamma_L]\neq [\gamma_{L'}]$, and work in the $\mathbf{A}$, $\left(m,k\right)$, and $\W$ settings separately.

\vspace{10 pt}
\noindent\textbf{${[\gamma_L] \neq [\gamma_{L'}]}$,  $\A$ setting}. Let $\lambda=\frac{k}{2}\left( \frac{d}{r}-\frac{d'}{r'}\right);$ by assumption $\lambda \neq 0$. Let $\epsilon = \epsilon\left(H\right)$, $\epsilon_0=\epsilon\left(H_0\right)-\epsilon\left(H'\right)$, $\epsilon_{1}=\epsilon\left(H_1\right)-\epsilon{H'}$.
 By writing $H=n \left(\frac{1}{n}H\right)$ for some large $n$, we can assume that 
 \[ \epsilon= \epsilon_1- \epsilon_0 < \frac{\epsilon_0}{4 |k/m|+4}. \]
 and
  \[ |\lambda| > 8 \pi \epsilon \left(|k/m|^2+1\right). \]
  Note that $\epsilon\left(\frac{1}{n} H\right)= \frac{\epsilon\left(H\right)}{n}$, so this step involves only one choice of $n$. Using Lemmas \ref{primitive} and \ref{formofh}, we can write
 \[ h - h' = \lambda w^2+ a w+\phi\left(t\right) + G_0\left(t,s\right) \]
where $w=s+\left(k/m\right)t$, $G_0$ is a compactly supported function on $C_R$, and $\phi: \R \to \R$ is a function with the following properties:
\begin{itemize}
\item There exists $c \in \R$ and $T > 0$ such that 
for $t <- T$,
\begin{align*} \phi\left(t\right)&=\left(b_- - \epsilon_0\right)t  \end{align*}
and for $t > T$, 
\begin{align*} \phi\left(t\right)&=\left(b_+ + \epsilon_0\right)t+c. \end{align*}
where $b_+, b_- \in \frac{1}{m}\Z$.
\item If $\left(b_- - \epsilon_0\right)\left(b_+ + \epsilon_0\right) > 0$, then \begin{align*} \left|\frac{d\phi\left(t\right)}{d t} \right| &> \epsilon \left(|k/m|+1\right) \end{align*} for $t \in [-1,1]$, and \begin{align*}\left|\frac{d \phi}{d t} \right| &> \epsilon \end{align*} everywhere.
\item If $\left(b_- - \epsilon_0\right)\left(b_+ + \epsilon_0\right) < 0$, then $\phi\left(t\right)=\left(M/2\right) t^2$ for $t \in [-2,2]$, where $M$ has the same sign as $\left(b_+ + \epsilon_0\right)$, and satisfies
\begin{equation} \label{mbounds} \left|1+\frac{4 k^2\lambda}{m^2\left(|M|+\epsilon\right)}\right| < 2 \end{equation}
and   \begin{align*}\left|\frac{d \phi\left(t\right)}{d t}\right| > \epsilon \end{align*} for $t \notin [-1,1]$.
\end{itemize}

Let $g=h-h'-G_0$. By construction $\p g/\p t$ is non-vanishing if  $\left(b_- - \epsilon_0\right)\left(b_+ + \epsilon_0\right) > 0$, and vanishes only at zero if  $\left(b_- - \epsilon_0\right)\left(b_+ + \epsilon_0\right) < 0$. In the former case let $G=G_0$. In the latter case, consider the compactly supported diffeomorphism $\Theta: C_R \to C_R$ defined by 
\[ \Theta\left(t,w\right) = \left(t, w - \frac{a}{2 \lambda} \chi_{T/\epsilon}\left( t\right)\right). \] Note that the function $g \circ \Theta: C_R \to \R$ has a unique critical point at $w=t=0$. Moreover,  $G_1=g \circ \Theta -g$ is a compactly supported function on $Y_R$. Let $G=G_0+G_1$. We conclude that $\tilde{L}$ and $\psi_{G} \left(\tilde{L}'\right)$ intersect in $\R^4$ only at $t=w=0$. In both cases $\mu^1$ vanishes on $\whCF^*\left(L_R, \psi_{G} \left(L_R'\right)\right)_\hn$; the complex is either zero or concentrated in degree $0,1,$ or $2$.

We wish to show that the same holds for $\whCF^*\left(\psi_{H}\left(L_R\right), \psi_{G} \left(L_R'\right)\right)_\hn$. First assume that  \[\left(b_- -\epsilon_0\right)\left(b_- + \epsilon_0\right) > 0.\] In this case we can write 
\begin{align*} \frac{\p }{\p t} \left(g+H\right) &= \frac{d}{d t} \left(\phi+\sigma_{\epsilon}^{\epsilon}\right)+ \frac{\epsilon k\sin\left(2\pi\left(w-\left(k/m\right) t\right)\right)\chi\left(t\right)}{4m}+ \frac{\epsilon\left(\cos\left(2\pi\left(w-\left(k/m\right) t\right)\right)+1\right)}{8 \pi} \frac{d \chi\left(t\right)}{d t} \\
 &> \left|\frac{d\phi}{d t}\left(t\right)\right| - \frac{\epsilon \left(|k/m|+1\right)}{2}\\
 & > 0. \end{align*}
Thus $\whCF^*\left(L_R, \psi_{-G} (L_R'\right))_\hn =0.$ In this case the continuation map is trivially an isomorphism.

Now assume that  $\left(b_- -\epsilon_0\right)\left(b_- + \epsilon_0\right) < 0$. Then for $t \in [-2T/\epsilon, 2T/\epsilon],$ $w -\frac{a}{2 \lambda} \chi_{T/\epsilon}\left( t\right)=w-\frac{a}{2 \lambda}$, so in this region we can write:
\begin{align*} \frac{\p }{\p t} \left(g+H\right) = \frac{d}{d t} \left(\phi+\sigma_{\epsilon}^{\epsilon}\right)+ \frac{\epsilon k\sin\left(2\pi\left(w-\left(k/m\right) t\right)\right)\chi\left(t\right)}{4m}- \frac{\epsilon\left(\cos\left(2\pi\left(w-\left(k/m\right) t\right)\right)+1\right)}{8 \pi} \frac{d \chi\left(t\right)}{d t} \\
\end{align*}
For $t \in [-1/2,1/2]$, we have the expressions
\begin{align*}& \frac{\p }{\p t} \left(g+H\right) =M t +\epsilon t+\frac{\epsilon k\sin\left(2\pi\left(w-\left(k/m\right) t\right)\right)}{4m} \\
& \frac{\p }{\p w} \left(g+H\right) = 2 \lambda w -\frac{\epsilon \left(\sin\left(2\pi\left(w-\left(k/m\right) t\right)\right)\right)}{4}. \end{align*}
We can conclude that any critical points of $g+H$ in this region satisfy \[ 2\left(k/m\right) \lambda w = \left(M+\epsilon\right)t.\] But we then can bound $ \frac{\p }{\p w} \left(g+H\right)$ away from zero in this for such $t,w$: any solutions to $\frac{\p }{\p w} \left(g+H\right) $ must satisfy
\begin{equation} \label{badeqn} 2 \lambda w -\frac{\epsilon}{4} \left(\sin\left(2\pi \left(1+\frac{2 k^2 \lambda}{m^2\left(M+\epsilon\right)}\right)w\right) \right) = 0. \end{equation}
Using the inequality $|\sin\left(x\right)| < |x|,$ and the bound from Equation \eqref{mbounds},
 we write 
\[ \frac{\epsilon}{4} \left(\sin\left(2\pi \left(1+\frac{2 k^2 \lambda}{m^2\left(M+\epsilon\right)}\right)w\right) \right)< 4 \pi \epsilon |w| < 2 \lambda w, \]
so there are no solutions to Equation \eqref{badeqn} in this region other than at the origin.

For $t$ with $1/2 < |t| < 2$, we have  
\begin{align*} &\left|\frac{\p }{\p t} \left(g+H\right) \left(t,w\right) \right|\\
& > M t-  \left|\frac{d \sigma_{\epsilon}^{\epsilon}\left(t\right) }{dt}+\frac{\epsilon k\sin\left(2\pi\left(w-\left(k/m\right) t\right)\right)\chi\left(t\right)}{4m} - \frac{\epsilon\left(\cos\left(2\pi\left(w-\left(k/m\right) t\right)\right)+1\right)}{8 \pi} \frac{d \chi\left(t\right)}{d t}\right|\\
  &> \frac{M}{2} - \epsilon \left(2+|k/m|\right) > 0. \end{align*}
For $t$ with $|t| > 2$, we have 
\begin{align*}  \frac{\p }{\p t} \left(g+H\right) \left(t\right)  = \frac{d\phi\left(t,w\right)}{d t} + \epsilon -\frac{a}{2 \lambda}\frac{T}\epsilon \frac{\p \left(g+H\right) }{\p w} \\
						           \end{align*}
						           so it suffices to show that $ \frac{d\phi\left(t,w\right)}{d t}  + \epsilon$ is non-vanishing; this is true since $\left|\frac{d \phi\left(t\right)}{d t}\right| > \epsilon$ for $t \notin [-1,1]$ 
The same methods show that $\frac{\p \left(g+H \right)\left(t,w\right)}{\p t} \neq 0$ when $t<-1/2$. We conclude that $g+H$ has a unique critical point at $\left(t=w=0\right)$. Hence $\phi_{H}\left(L_R\right)$ and $\psi_{-G} \left(L_R'\right)$ intersect only at $t=w=0$, and $CF^i\left(\psi_{H}\left(L_R\right), \psi_{-G}\left( L_R\right)\right)_\hn$ is concentrated in degree $i$ for $i=0,1,$ or $2$.

We now consider the map $\mu^2\left(\cdot, c_{\tilde{H}}\right): CF^i\left(L_R, \psi_{-G}\left(L_R\right)\right)_\hn \to CF^i\left(\psi_{H}\left(L_R\right), \psi_{-G}\left(L_R\right)\right)_\hn$. By the exactness of the lifts of the Hamiltonians and the vanishing of $H$ at its minima, the area of any $J$-holomorphic curve contributing to this product is zero; thus the product counts only the constant solution over $t=w=0$. This map is evidently an isomorphism. 

\vspace{10 pt}

\noindent \textbf{$[\gamma_L] \neq [\gamma_{L'}], \W$ setting}. 
The proof in this setting is similar. The fact that $L$ and $L'$ are in correct position implies that when $L$ and $L'$ do not have the same total slope, $h-h'$ is of the form
\begin{equation} h - h' = \lambda w^2  +a t^2 +b w+ \phi\left(t\right) + G_0\left(t,s\right) + f\left(t,s\right),\end{equation}
where: $\lambda> 0$; $w=w_{L, L'}=s - c t$ for some $c \in \R$; $a > 0$; $G$ is compactly supported on $C_R$; $f$ satisfies \[|f\left(t\right)|< \alpha=\min\{1, \alpha_{0,-}, \alpha_{0,+}\}\] and $f\left(t\right)$ is zero for $t \in [-2,2]$, and 
\begin{align}& \left|\frac{\p f}{\p w} \right |\leq \frac{\alpha_{0,+}}{8}\\
& \left| \frac{\p f}{\p t }\right| \leq \frac{ \left(1+|c|\right)\alpha_0}{8}, \\
 & \left|\frac{\p f}{\p t }\right|  \leq |a t|; \end{align}
and $\phi: \R \to \R$ is a function with the following properties: $\phi\left(t\right)$ is identically zero for $t \in [-2,2]$ and satisfies 
\begin{align*} \frac{-at}{2}-1&>\frac{d \phi}{dt} \left(t\right) \quad \text{for all} \,\, t > 0; \\
 		        \frac{at}{2}+1&<\frac{d \phi}{dt} \left(t\right)\quad \text{for all} \,\,t <0. \end{align*}
We also assume 
\begin{equation} \label{mbounds2} 2\pi \left|\left( \alpha_+ +\epsilon_+\right)\left(a+  \alpha_ + + \epsilon_+\right) \left(1+ |c| \right)\right| < 2\lambda \end{equation}
and \begin{equation} a-4\left( \alpha_+ +\epsilon_+\right) > 0.\end{equation}

Let  $g=h-h'-G_0$. Then by construction $\p g/\p t > 0$ away from $0$, and the function has a unique critical point at a point $\left(t,w\right)=\left(0,w_0\right)$. Again let $\Theta: C_R \to C_R$ be defined by $\left(t,w\right) \mapsto \left(t, w-w_0\chi_{\epsilon/T}\left(t\right)\right)$ where $T>2$, and let $-G_1=g- g \circ \Theta $. Then $g-G_1$ has a unique critical point at $t=s=0$. Set $G=G_0+G_1$; then $CF^i\left(\psi_{H}\left(L_R\right), \psi_{G}\left(L_R\right)\right)_\hn$.

We wish to show that the same holds for $CF^i\left(\psi_{H}\left(L_R\right), \psi_{G}\left(L_R\right)\right)_\hn$. Write
\begin{align*} &g \circ \Theta+H \\
&= \lambda \left(w'\right)^2  +a t^2 +b \left(w'\right)+ \phi\left(t\right)  + f\left(t,w'\right) + \sigma_{\alpha_-,\epsilon_-}^{\alpha_+, \epsilon_+}\left(t\right)- \frac{\chi\left(t\right) \left(\alpha_+ + \epsilon_+\right)\cos\left(2 \pi \left(w'-ct\right)\right)+1}{8 \pi}\\
&\quad\quad\quad+\rho_0^1\left(-t-1\right)f_-\left(t,w'\right) + \rho_0^1\left(t+1\right)f_+\left(t,w'\right), \end{align*}
where $w' = w -\left(\chi_{T}\left(t\right) b/2 \lambda\right)$. Note that for $t \in [-T,T]$, $\left(t,w'\right)=\left(t,w+C\right)$ for some constant $C$. 

 For $t \in [-1/2, 1/2]$, \[\phi\left(t\right)=\rho_0^1\left(t+1\right)=\rho_0^1\left(t\pm1\right)=0,\] and$\chi\left(t\right)=1$. Thus in this region
\[ \frac{\p\left(g \circ \Theta+H\right)\left(t,w'\right)}{\p t} = \left(2 a + 2\epsilon_++2\alpha_+\right)t + \frac{\left(\alpha_++\epsilon_+\right) c\sin\left(2\pi\left(w-c t\right)\right)}{4};\]
 \[ \frac{\p }{\p w} \left(g \circ \Theta+H\right) = 2 \lambda w - \frac{\left(\alpha_+ +\epsilon_+\right) \sin\left(2\pi\left(w-c t\right)\right)}{4}.\]
 The same methods as in the above, now using the bound from Equation \eqref{mbounds2}, show that $g-G_0-H$ has no critical points away from $t = 0$ in this region. 
 
  For $t$ with $1/2 < t< 2$, we can bound $ \frac{\p }{\p t} \left(g-G_1+H\right)$ below:
  \begin{align*}  &\frac{\p\left(g \circ \Theta+H\right)\left(t,w\right) }{\p t} \\
  & = 2 a t + \frac{\p f\left(t,w\right)}{\p t}+  \frac{d \sigma_{\alpha_-,\epsilon_-}^{\alpha_+, \epsilon_+}\left(t\right)}{dt} + \frac{d \chi\left(t\right)}{dt}\frac{ \left(\alpha_+ + \epsilon_+\right)\cos\left(2 \pi \left(w-ct\right)\right)+1}{8 \pi}\\
  & \quad \quad \quad+ \frac{\chi\left(t\right) \left(\alpha_+ + \epsilon_+\right)\sin\left(2 \pi \left(w-ct\right)\right)}{4 } + \frac{d \rho_0^1\left(t+1\right)}{d t} f_+\left(t,w\right) +\rho_0^1\left(t+1\right)\frac{ \p f_+\left(t,w\right)}{\p t}\\
  & > 2 a t +\frac{ 2 \alpha_{+} t + \epsilon_{+}}{2} - \frac{\left(\alpha_+ + \epsilon_+\right)}{2} -\frac{ \alpha_+}{4}\\
  & \geq a + \frac{\left(\alpha_+ + \epsilon_+\right)}{2}   -\frac{ \alpha_+}{4}  > 0. \end{align*}
For $t$ with $t > 2$, we have 
\begin{align*} \frac{\p\left(g \circ \Theta+H\right)\left(t,w\right) }{\p t}& =  \frac{\p\left(g +H\right)\left(t,w\right) }{\p t}- \frac{b}{2 \lambda T} \frac{d \chi_t\left(t\right)}{dt} \frac{\p\left(g +H\right)\left(t,w'\right)}{\p w}\\
 \frac{\p\left(g \circ \Theta+H\right)\left(t,w\right) }{\p w}& = \frac{\p\left(g +H\right)\left(t,w'\right)}{\p w}.\end{align*}
Thus it suffices to show that  $  \frac{\p\left(g +H\right)\left(t,w\right) }{\p t}$ is non-vanishing in this region. We have that 
\begin{align*} \frac{\p\left(g +H\right)\left(t,w'\right) }{\p t}= 2 a t + \frac{\p f\left(t,w'\right)}{\p t}+2 \alpha' t+ \frac{\p f_+\left(t,w'\right) }{\p t} 
					& \geq a t -\frac{ \alpha_+}{8} > 0. \end{align*}

Similar methods will show that there are no critical points in the  $t < 0$ region. Thus the unique critical point of $g \circ \Theta+H$ occurs at $t=w=0$. Again this allows us to conclude that the map \[\mu^2\left(\cdot, c_{H_R}\right): \whCF^i\left(L_R, \psi_{G}\left(L_R'\right)\right)_\hn \to \whCF^i\left(\psi_{H}\left(L_R\right), \psi_{G}\left(L_R'\right)\right)_\hn\] is an isomorphism, and thus $\cdot [c_{H}]_{\hn}$ is a quasi-isomorphism.

When $L$ and $L'$ have the same total slope, we can write 
\vspace{10pt} \begin{equation} h - h' = a t^2 + b(t s)+ c s + \phi\left(t\right) + G_0\left(t,s\right) + f\left(t,s\right),\end{equation}
for some $a,b \neq 0$. In this case we again choose a compactly supported diffeomorphism $ \Theta: C_R \to C_R$  that takes the unique critical point of the function $a t^2 + b(t s)+ c s$ to $(0,0)$, and show that for carefully chosen  $\Theta$, the minima of $(h-h'-G_0) \circ \Theta$ and $(h-h'-G_0) \circ \Theta +H$ both lie above $t=s=0$. Let $G=G_0+(h-h'-G_0) \circ \Theta - h-h'$ shows that 
$\mu^2\left(\cdot, c_{\tilde{H}}\right): CF^i\left(L_R, \psi_{-G}\left(L_R\right)\right)_\hn \to CF^i\left(\psi_{H}\left(L_R\right), \psi_{-G}\left(L_R\right)\right)_\hn$. By the exactness of the lifts of the Hamiltonians and the vanishing of $H$ at its minima, the area of any $J$-holomorphic curve contributing to this product is zero; thus the product counts only the constant solution over $t=w=0$. This map is evidently an isomorphism. 

\vspace{10pt}

\noindent \textbf{$[\gamma_L]=[\gamma_{L'}]$, $ \A$ setting.} Again we  lift to $Y_R$, and perturb $L_R$ and $L_R'$ by a compactly supported Hamiltonian.

For $|t| >> 0$, we can write 
\begin{align*} h - h' &= a s + b_{\pm} t+ \epsilon_0 |t|. \end{align*}
Assume first that $a \neq 0$. Then we assume that $|\epsilon| < a$; we further assume that $2 R \epsilon < \epsilon_1$.  We can write 
\begin{align*}\left( h - h' \right)\left(t,s\right)&= a s +\phi\left(t\right)+ \sigma_{\epsilon_0}^{\epsilon_0}\left(t\right) \chi\left(tt\right) + \frac{\epsilon\left( \cos\left(2 \pi s\right) + 1\right)}{8 \pi} +G\left(t,s\right); \\
 \left(h - h' +H \right)\left(t,s\right)&= a s +\phi\left(t\right) + \sigma_{\epsilon_1}^{\epsilon_1}\left(t\right) + G\left(t,s\right). 
 \end{align*}
where $G\left(t,s\right)$ is compactly supported on $C_R$. Let $g=h-h'-G$. Then $\frac{\p g}{\p s}$ is nonvanishing everywhere, as is $\frac{\p g+H}{\p s}$. We conclude that \[\whCF^*\left(L_R, \psi_G\left(L_R'\right)\right)_{\hn} = \whCF^*\left(\psi_{H}\left(L_R\right), \psi_G\left(L_R'\right)\right)_{\hn} = 0,\] so the multiplication map $\whCF^*\left(L_R, \psi_G\left(L_R'\right)\right)_{\hn} \to \whCF^*\left(\psi_{H}\left(L_R\right), \psi_G\left(L_R'\right)\right)_{\hn}$ is trivially an isomorphism.

Now assume $a=0$. Then, similarly as in the first case, we can write
\begin{align*}\left( h - h'+H \right)\left(t,s\right)&=\phi\left(t\right)+ \sigma_{\epsilon_0}^{\epsilon_0}\left(t\right) +c \chi_T\left(t\right) \frac{\epsilon\left( \cos\left(2 \pi s/R\right) + 1\right)}{8 \pi}  +\chi\left(t\right) \frac{\epsilon\left( \cos\left(2 \pi s\right) + 1\right)}{8 \pi} + G\left(t,s\right); 
\end{align*}
and $g = h-h'-G$ 
where $\phi$, $T$ and $c$ are such that $\frac{\p g}{\p t}$ is either non-vanishing (in which case the Floer complexes are both zero and the map is trivially an isomorphism) or vanishes only at $0$, and $\frac{\p g}{\p s}$ vanishes only at $s=0,R/2$, and $\frac{\p g+H}{\p t}$ has the same behavior, and $G$ is a compactly supported function on $C_R$. Then the Floer complexes $\whCF^*\left(L_R, \psi_{G}\left(L_R'\right)\right)_{\hn}$ and $\whCF^*\left(\psi_{H}\left(L_R\right), \psi_{G'}\left(L_R'\right)\right)_{\hn}$ are concentrated in degrees $i$ and $i+1$ for $i=0$ or $i=1$, and are each rank one in each non-zero degree. We can show that the differentials on these complexes vanish using the same methods as in the proof of Lemma \ref{formofh} (b). Thus as graded $\Lambda$-modules,
\[  \wh{HF}^*\left(L_R, \psi_{G}\left(L_R'\right)\right)_{\hn} =   \wh{C F}^*\left(L_R, \psi_{G}\left(L_R'\right)\right)_{\hn} \]
 \[ \wh{H F}^*\left(\psi_{H}\left(L_R\right), \psi_{G}\left(L_R'\right)\right)_{\hn} = \wh{C F}^*\left(\psi_{H}\left(L_R\right), \psi_{G}\left(L_R'\right)\right)_{\hn}.  \]
Thus the multiplication by the continuation map furnishes an isomorphism \[ \wh{H F}^i\left(L_R, \psi_{G}\left(L_R'\right)\right)_{\hn} \cong \wh{HF}^i\left(\psi_{H}\left(L_R\right), \psi_{G}\left(L_R'\right)\right)_{\hn}.\] We conclude that $\cdot [c_H]_\hn^i$ is an isomorphism for each $\hn \in \Z^3$. We can similarly deform $L_R$ and $L_R'$ via a compactly supported Hamiltonian to ensure that the generators in degree $i+1$ lie over $t=s=0$ to see that $\cdot [c_H]_\hn^{i+1}$ is an isomorphism.

\noindent \textbf{Other settings.} The proof in the $[\gamma_L]=[\gamma_{L'}]$ case in the $\W$ setting is similar; again at the ends of $C_R$ the function $h-h'$ will be (a small perturbation of) a quadratic function in $t$ and a linear function in $s$, and we can deform $h$ and $h'$ by compactly supported functions so that the generators of the Floer complexes lie above $t=s=0$ and such that the Lagrangians split as product Lagrangians in for small $t$. 

Finally, in all cases, the proof in the partially wrapped $\left(m,k\right)$ setting can be obtained by combining the methods of the proofs above, treating each end of $Y_R$ separately. 

\vspace{10pt}

\noindent\textbf{Extension to completions.} We now show that these isomorphisms extend to the completions $\whCF^* \left(L, L'\right)$ and $\whCF_{\underline{r}}^* \left(L, L'\right)$. Note that for each $y \in CF^*\left(L, L'\right)$, the output $\mu^2\left(y, c_{H_Q}\right)$ counts only finitely many $u: \D \to Y$; thus the statement is evidently true in the finite $\A$ setting. We now show controls on the outputs for $y \in \left(L \cap L'\right)_T$ for some $T > 0$. 

We first claim that away from a compact set, $\mu^1\left(y\right) = 0$ for all $y \in CF^*\left(L, L'\right)$. This follows from the fact that $L$ and $L'$ are in correct position, so for all but finitely many $\hn \in \Z^3/\Gamma$, $|\tilde{L} \cap \tilde{L}'_\hn| =0, 1$ or $2$, and as shown above, $\chi\left(CF^*\left(L,L'\right)_{\hn}\right)$ is equal to $1$ or $0$ for all $\hn$. Thus away from a finite set $A$ of $\hn \in \Z^3/\Gamma$, the map
\[ \mu^2\left(\cdot, c_{H_Q}\right)_{\hn}: CF^*\left(L, L'\right)_{\hn}\to CF^*\left(\psi_{H}\left(L\right), L'\right)_{\hn} \]
computes the induced map on homology. 

Let $\hn \notin A$, let $y \in CF^*\left(L, L'\right)_{\hn}$ and $y_{H_Q}$ be the generator of the same degree in $CF^*\left(\psi_{H_Q}\left(L\right), L'\right)_{\hn}$; then the map $y \mapsto y_{H_Q}$ defines a bijection
\[ \left\{ \bigcup_{\hn \notin A} \left(\psi_{H}\left(L\right) \cap L' \right)_{\hn} \right\} \leftrightarrow \left\{\bigcup_{\hn \notin A} \left(L \cap L'\right)_{\hn}\right\}. \]
Since $\mu^2\left(\cdot, c_{H_Q}\right)_{\hn}$ is an isomorphism for $\hn \notin A$ and respects the degrees of the generators, so $ \mu^2\left(y, c_{H_Q}\right)_{\hn} =  c y_{H_Q}$
for some $c \in \Lambda^*$. 
For any $u: \D \to Y$ contributing to $\mu^2\left(y,c_{H_Q}\right)$, 
\[ \int_{\D} u^* \omega = S_{\R^4}\left(y\right)- S_{\R^4}\left(y_{H_Q}\right) \]
by Stokes' theorem. Thus \[ \mu^2\left(y, c_{H_Q}\right)_{\hn} =  n T^{ S_{\R^4}\left(y\right)- S_{\R^4}\left(y_{H_Q}\right)} y_{H_Q} \]
for some $n \in \N$. But since the quasi-isomorphisms described above exist over $\Lambda_{\Z}$ (see Remark \ref{overZ}), we must have that $n \in GL_1\left(\Z\right) = \{1,-1\}$. Since $S_{\R^4}\left(y\right)- S_{\R^4}\left(y_{H_Q}\right) > 0$ for all $y \in L \cap L' \setminus X_0$, the map $\mu^2$ extends to a map $\whCF^*\left(L,L'\right)\to \whCF^*\left(\psi_H\left(L\right),L'\right)$.

Now let $y = \sum_{i} c_i \mathbf{y}_i \in \whCF^*\left(L,L'; \Lambda_r\right)$ and let $\mu^2\left(y , c_{H_Q}\right)$ be the formal sum $\sum_{i} \mu^2\left(c_i \mathbf{y_i}, c_{H_Q}\right)$. We now show that this sum defines an element of $\whCF_{\underline{r}}^*\left(\psi_{H}\left(L\right),L'\right)$ if and only if $y \in \whCF_{\underline{r}}^*\left(L,L'\right)$. Note that 
\[ \mu^2\left(y,c_{H_Q}\right) =\sum_{\hn \in A} c_i \mu^2\left(\mathbf{y}_i, c_{H_Q}\right)_{\hn} +\sum_{\hn \notin A} c_i T^{S_{\R^4}\left(y_i\right)-S_{\R^4}\left(\left(y_i\right)_{H_Q}\right)-S\left(y_i\right)-S\left(\left(y_i\right)_{H_Q}\right)} \mathbf{\left(y_i\right)_{H_Q}}.\]
The first sum is finite; write the second sum as $\sum_{i} d_{i} \mathbf{\left(y_i\right)_{H_Q}}$.

By Lemma \ref{boundondiff},  there exists $D> 1$ such that $\left|S_{\R^4}\left(y_i\right)-S_{\R^4}\left(\left(y_i\right)_{H_Q}\right)-S\left(y_i\right)+S\left(\left(y_i\right)_{H_Q} \right)\right|< D$ for all $i$; by the linearity of the equations governing the intersection of the corresponding planes, there exists $C_1, C_2 > 0$ such that for all $y_i$ with $|t\left(y_i\right)|$ sufficiently large, 
\[C_1\left( \left|t\left(\left(y_i\right)_{H_Q}\right)\right| -1 \right)< |t\left(y_i\right)| < C_2\left( \left|t\left(\left(y_i\right)_{H_Q}\right)\right|+1\right). \]
Let $R=e^{-2\pi r}$. Then
 \[  \lim_{|t\left(y_i\right)| \to \infty}\left|R^{D-C_1}  \ev_{\underline{r}}\left( c_i\right)^{\left(C_1|t\left(y_i\right)|\right)\inv}\right|\leq  \lim_{|t\left(y_i\right)| \to \infty} \left|\ev_{\underline{r}}\left(d_i\right)^{|t((y_i)_{H_Q})|\inv}\right| \leq \ \lim_{|t\left(y_i\right)| \to \infty}\left| R^{-D -C_2} \ev_{\underline{r}}\left(c_i\right)^{\left(C_2|t\left(y_i\right)|\right)\inv} \right|  \]
so the middle term vanishes if and only if $\lim_{|t\left(y_i\right)| \to \infty} \ev_{\underline{r}}\left(c_i\right)^{|t(y_i)|\inv}$ vanishes; thus $\mu^2\left(y, c_{H_Q}\right)$ defines an element of $\whCF_{\underline{r}}^*\left(\psi_{H}\left(L\right),L'\right)$ if and only if $y$ converges in $\whCF_{\underline{r}}^*\left(L,L'\right)$. We conclude that the map $\mu^2\left(\cdot, c_{H_Q}\right)$ is a quasi-isomorphism.
\end{proof}

\subsection{Convergence of products over $\C$.}

We now show that the $\mu^2$ products on cohomology converge over $\C$. We first prove a collection of propositions which have proofs that use similar methods to those in the proof in Lemma \ref{qisomcont}. From the proof of the lemma, we can conclude the following: 
\begin{prop}\label{compwelldef} \begin{enumerate}[(a)]
\item For $x = \sum_{i} a_i x_i \in  \whCF\left(L_0, L_1; \Lambda_r\right)$,  $x$ defines an element of the completion $\whCF_{\underline{r}}^*\left(L_0, L_1\right)$ if $[x-x]'=0$ for some $x' \in \whCF_{\underline{r}}^*\left(L_0, L_1\right)$. 
\item Let $x = \sum_{i} a_i x_i \in  \whCF^*\left(L_0, L_1; \Lambda_r\right)$, and assume that there exists a perturbing type-$\Wr$ Hamiltonian $H_0$ such that $\mu^2\left(x, c_{H_0}\right) \in \whCF_{\underline{r}}^*\left(\psi_{H_0}(L_0), L_1\right)$. Then $x \in \whCF_{\underline{r}}^*\left(L_0, L_1; \Lambda_r\right)$. \end{enumerate} \end{prop}

\begin{lemma}\label{onlyone}Let $\left(L_0,L_1,L_2\right) \in \G_P\left(I\right)$ be Lagrangians in correct position in $\F\left(I\right)$ such that the differential $\mu^1$ vanishes on $CF^*\left(L_{0},L_{1}\right)$, $CF^*\left(L_{1},L_{2}\right)$,  and $CF^*\left(L_{0},L_{2}\right)$. Let $\hn_1, \hn_2 \in \Z^3$. Let $x_0 \in \left(L_0 \cap L_1\right)_{\hn_1}$ and $x_1 \in \left(L_1 \cap L_2\right)_{\hn_2}$, and let $x_0$ lift to a point $\tilde{x}_0$. Then if  $\left(L_0 \cap L_2\right)_{\hn_1+\hn_2}$ is non-empty, and  if the Floer complex $CF^*\left(L_0, L_2\right)_{\hn_1+\hn_2}$ is generated in degree ${\deg x_0 + \deg x_1}$ by a point $x_2$, the signed number of maps counted by $\mu^2\left(x_1, x_0\right)$ that lift to maps $u: S \to \R^4$ with $u\left(p_0\right)=\tilde{x}_0$ and $u\left(\p_0 S\right)\subset \tilde{L}_0$, $u\left(\p_1 S\right) \subset \tilde{L}_1-\hn_1$, $u\left(\p_2 S\right) \subset \left(\tilde{L}_2-{\hn_1 - \hn_2}\right)$ is $\pm 1$.

 \end{lemma}
  \begin{proof} We have seen that there exists $R \in \N$ and Hamiltonians $G_1,G_2 : Y_R \to \R$, which are compactly supported on $Y_R$ such that $\tilde{L}_0, \wt{\psi_{G_1}\left(L_1\right)},  \wt{\psi_{G_2}\left(L_2\right)}$ each intersect at over $t=s=0$, and 
 \begin{align*} \Phi_{-G_1}\left(x_0\right) &= T^{A_0} y_0 \in CF^*\left(L_0,\psi_{G_1}\left(L_1\right)\right);\\ 
 \Phi_{G_1-G_2}\left(x_1\right)&= T^{A_1} y_1 \in CF^*\left(\psi_{G_1}\left(L_1\right),\psi_{G_2}\left(L_2\right)\right);\\
  \Phi_{-G_2}\left(x_2\right)&= T^{A_2} y_2 \in CF^*\left(L_0,\psi_{G_2}\left(L_2\right)\right);
\end{align*}
and
 \[ \Phi_{-G_1}\left(\tilde{x}_0\right) = T^{A_0} \tilde{y_0};\]
 where $y_i$ denotes the corresponding generator of the Floer complex over $t=s=0$. Then any disc counted by $\mu^2\left(y_1, y_0\right)$ with the prescribed boundary conditions must have energy $0$ and hence be constant, so
  \[\mu^2\left(y_0, y_1\right)_{\overline{\hn_1+\hn_2}} =y_2.\] 
  Thus we must have that
  \[ \mu^2\left(x_0,x_1\right)_{\overline{\hn_1+\hn_2}} = n T^{A}x_2 = T^{-A_0- A_1} \psi_{G_2}\left(y_2\right)\] for some $n \in \Z$, where $A = - S_{\R^4} \left(x_2\right)+S_{\R^4}\left(x_0\right)+S_{\R^4}\left(x_1\right)$, and $n$ is the signed number of discs counted by $\mu^2$. But then since $\Psi_{G_2}$ is an isomorphism over $\Lambda_{\Z}$, we conclude that $n=\pm1$. 
 \end{proof}
 
 \begin{cor} If $L_0, L_1, L_2 \in \G_P$ are such that the set $\left(L_0,L_1, L_2\right)$ is in correct position,  $L_0$ and $L_1$ are essentially non-transverse or $L_1$ and $L_2$ are essentially non-transverse, and the differentials of the complexes $CF^*\left(L_0, L_1\right), CF^*\left(L_1, L_2\right)$, and $CF^*\left(L_0, L_2\right)$ vanish, then 
 for each $x \in L_0 \cap L_1, z \in L_1 \cap L_2$, $\mu^2\left(z,x\right)$ counts at most one $J$-holomorphic map. Moreover, there exists a uniform $C > 0$ such that for all such $x, z$ which bound a $J$-holomorphic triangle with end point $y \in L_0 \cap L_2$,
\[ \mu^2\left(\mathbf{z},\mathbf{x}\right)=T^{A} \mathbf{y}\]
where $|A|\leq C$. There also exists $D > 0$ such that if $\mu^2\left(\mathbf{z},\mathbf{x}\right)=T^A \mathbf{y}$, $|t\left(y\right)|<D (|t(z)+t(x)|+1)$ .
  \end{cor}
\begin{proof} We apply the previous lemma to prove the first statement when $L_0$ and $L_1$ are essentially non-transverse. Let $x \in L_0 \cap L_1, z \in L_1 \cap L_2$ be such that $\mu^2\left(z,x\right)$ counts at least one map. 
Let $\hn_2, \hn_2' \in \Z^3$ be such that $z$ lifts to points $\tilde{z}, \tilde{z}'$ in $\left(\tilde{L}_1+\hn_1\right) \cap \left( \tilde{L}_2-\hn_2 \right)$ and $\left(\tilde{L}_1+\hn_1\right) \cap \left( \tilde{L}_2-\hn_2'\right)$ respectively. Then $\hn_2-\hn_2'$ defines an element of the group $\Gamma\left(L_1, L_2\right) \subset \Z^3$ of Definition \ref{gradingrem}. Since $L_0$ and $L_1$ are essentially non-transverse, $[\gamma_{L_0}] =[ \gamma_{L_1}]$, so $\Gamma\left(L_1, L_2\right)=\Gamma\left(L_0, L_2\right)$ since the stabilizer subgroups are topological invariants. But then $\overline{\hn_1 + \hn_2-\hn_1 + \hn_2'} = \overline{\hn_2-\hn_2'}\in \Gamma\left(L_0,L_2\right)$. Thus $\mu^2\left(z,x\right)$ is concentrated in $CF^*\left(L_0, L_2\right)_{\overline{\hn_1 +\hn_2}}$; by the previous lemma the signed number of maps counted by $\mu^2\left(z,x\right)$ is 1.

Now, fix $x$ and $z$, and let $u: S \to \R^4$ be the unique triangle which contributes to $\mu^2\left(\mathbf{z},\mathbf{x}\right)=T^A \mathbf{y}$; lift $u$ to a map $S \to \R^4$ which maps boundary marked points to $\tilde{x}$,  $\tilde{z}$,and $\tilde{y}$ .  Apply Stokes' theorem to write
\[ A=  \left(S_{\R^4}\left(\tilde{x}\right)-S(x)\right)+\left(S_{\R^4}\left(\tilde{z}\right) -S(z)\right)- \left(S_{\R^4}\left(\tilde{y}\right)-S(y)\right). \]
Since $L_0$ and $L_1$ are essentially non-transverse, there exists a Hamiltonian $H: \R^4 \to \R$ which is $C^2$-close to a quadratic function of $t$ such that $\psi_{H}\left(L_1\right)=\psi_{H}\left(L_0\right)$. We can thus apply Lemma \ref{boundondiff}; there exists a uniform $C_2 >0$ such that $|\left(S_{\R^4}\left(z\right) -S(z)\right)- \left(S_{\R^4}\left(y\right)-S(y)\right)|< C_2$ for all $z$ and $y$. We apply Lemma \ref{quadbounds} (c) to bound  $|\left(S_{\R^4}\left(x\right)-S(x)\right)| < C_3$. Then taking $C=C_2 + C_3$ gives the second statement.

The third statement follows from linear algebra. All points are close to the intersection points of the corresponding planes of the Lagrangians, and the $t$-value at which two planes $P_0$ and $P_2$ intersect is determined by a linear expression in the intersection points of the planes $P_0$ and $P_1$ and the planes $P_1$ and $P_2$.

\end{proof}

\begin{prop} \label{conv} Let $L_0, L_1, L_2 \in \G_P$ be such that the set $\left(L_0,L_1, L_2\right)$ is in correct position. Let $x \in \whCF_{\underline{r}}^*\left(L_{0}, L_{1}\right)$, $z \in \whCF_{\underline{r}}^*\left(L_{1}, L_{2}\right)$ satisfy $\mu^1\left(x\right)=0$, $\mu^1\left(z\right)=0$. Then $\mu^2\left(z,x\right)$ converges in $\whCF^*_{\underline{r}}\left(L_0, L_2\right)$. \end{prop}
\begin{proof}
First assume that $L_0$ and $L_1$ are essentially non-transverse, and that the differential vanishes on the complexes $CF^*\left(L_0, L_1\right)$, $CF^*\left(L_1, L_2\right)$, and $CF^*\left(L_0,L_2\right)$.  Let $C$ and $D$ be as in the above proposition. Let $R=e^{-2 \pi r}.$
Write $x=\sum_{i} a_i \mathbf{x}_i$, $z = \sum_{j} b_j \mathbf{z}_j$.  If  $\mu^2\left(\mathbf{x}_i, \mathbf{z}_j\right)$ is non-vanishing, then we can write $\mu^2\left(\mathbf{x}_i, \mathbf{z}_j\right)=C_{ij} T^{A_ij} \mathbf{y}_{ij}$ for some $y_{ij} \in L_0 \cap L_2$ where $C_{ij}=0$ or $1$; if $C_{ij}=1$ then $y_{ij}$ is unique. Write then
\[ \mu^2\left(x,z\right) = \sum_{i,j} a_i b_j C_{ij} T^{A_{ij}} \mathbf{y}_{ij}, \]
We now show this sum satisfies the criterion of Proposition \ref{convergence}, i.e., for all $c \in \R$,
\[ \sum_{i,j}\left|R^{A_{ij}} \ev_{\underline{r}}\left(a_{i} b_{j}\right) R^{c t\left(y_{ij}\right)} \right|< \infty. \]
Note that 
\begin{align*} &\sum_{i,j} |R^{A_{ij}}   \ev_{\underline{r}} \left(a_{i} b_{j}\right) R^{-|c t\left(y_{ij}\right)|}| \\
&\leq R^{-C-1} \sum_{i,j} R^{ -| c D t\left(x_i\right)|} R^{-|c Dt\left(z_j\right)|} |\ev_{\underline{r}}\left(a_i b_j\right)|\\
& = R^{-C-1} \left( \sum_{i}R^{-| c D t\left(x_i\right)|} |\ev_{\underline{r}}\left(a_i\right)| \right) \left( \sum_{j}R^{-|c D t\left(z_j\right)|} \left|\ev_{\underline{r}} \left(b_j \right)\right| \right)  < \infty.
 \end{align*}
The same method will show convergence for all $c < 0$.

Now assume that $L_0, L_1,$ and $L_2$ are planar Lagrangians which each intersect transversely with $[\gamma_{L_i}]=[r_i, d_i m_i, d_i k_i]$. Now fix several pieces of data. Let  $M_0,M_1,M_2 \in \operatorname{Sym}_{2 \times 2} \left(\R\right)$, $E_0, E_1, E_2 \in \R^2$ be such that for $i=0,1,2$,
\[ L_i =\{ (X, M_i  X+ E_i ) \mid X \in \R^2. \}\]
 Recall from Definition \ref{qdecomp} that to each pair of Lagrangians we associate quadratic forms and coordinates $w_{L_i, L_j}$. Denote these quadratic forms $Q_{L_{0},L_{1}}, Q_{L_{1},L_{2}},Q_{L_{0},L_{2}} $ by $Q_0,Q_1$, and $Q_2$ respectively. Let $Q_{i}=Q^\alpha_i+Q^{\beta}_i$, and let $w_i$ be the coordinate associated to the pair of Lagrangians in question. 

Let $d_i' = d_i/r_i$, $i=0,1,2$. We now show that $\mu^2(x,z) \neq 0$ implies \[(d_1'm_1 -d_0' m_0)(d_2' m_2 - d_1' m_1)(d_2' m_2 - d_0 'm_0)>0.\] Note that for all pairs of transverse planar Lagrangians $L$ and $L'$, $CF^*(L, L')=0$ is concentrated in a single degree $i$, where $i=0$ or $1$ depending on the signature of $Q_{L,L'}$. Assume that $(d_1' m_1 -d_0 'm_0)=0$. Then for all $x_i \in L_1 \cap L_0$, $\det (M_1 -M_0) < 0$, so $\deg(x_i)=1$. Then the non-vanishing of the product implies that $z$ is non-zero in degree $0$, hence concentrated in degree zero. But this implies that $d_2 ' m_2 - d_1' m_1>0$, hence $d_2' m_2 - d_0' m_0 >0$, so the output is concentrated in degree zero; this contradicts the fact that $\mu^2(x,z)$ is concentrated in $\deg (z)+\deg(x)=1$.

 Define a linear map $B: \R^2 \times \R^2 \to \R^2$ as follows: for a pair $\left(z,x\right)$, let $B\left(z,x\right)$ denote the unique point in $\R^2$ such that there is a planar triangle $\Delta\left(x,z\right)$ in $\R^4$ with sides lying in translates of the planar Lagrangians $L_i$ and vertices lying over $z,x,$ and $B\left(z,x\right)$.
Explicitly, $B\left(z,x\right)$ is given by \[\left(M_0-M_2\right)\inv\left(\left(M_1-M_2\right) z +\left(M_0-M_1\right)x\right).\]  Define $A\left(z,x\right)$ to be the symplectic area of the triangle  $\Delta\left(z,x\right)$. 
Extend $B$ and $A$ to maps on $\R^4 \times \R^4$ by composing with the projection to the first two coordinates.

Let $x_i \in L_0 \cap L_1$ and $z_j \in L_1 \cap L_2$; fix $\hn_1, \hn_2 \in \Z^3$ such that $\tilde{x}_i \in \tilde{L_1}-{\hn_1}$, $\tilde{z}_j \in \left(L_1 -\hn_1\right) \cap \left( L_2 - \hn_1 -\hn_2\right)$ satisfy $0<w_0(\tilde{x}_i), w_1(\tilde{z}_j)<1$. Using Proposition \ref{onlyone}, we can write the product $\mu^2\left(z_j, x_i\right)$ as the sum of products corresponding to different choices of the lift of $z_j$: 
\[ \mu^2\left(z_j, x_i\right)= \sum_{n \in \N} \mu^2\left(z_j, x_i\right)_{\hn_1+\hn_2+n \hat{e}_s},\]
where
\[ \mu^2\left(z_j, x_i\right)_{\hn_1+\hn_2+n \hat{e}_s}=\pm T^{A_{ijn}} \mathbf{y}_{ijn}. \]
Here  $\mathbf{y}_{ijn}=\pi(\mathbf{B\left(\tilde{x}_i, \tilde{z}_j-n\hat{e}_s\right)})$ denotes the action-corrected generator of $CF^*\left(L_0,L_2\right)_{\hn_1+\hn_2+n \hat{e}_s}$. For $c \in \R$, define the following function: 
\begin{equation} \label{functionsum} f_{ij}(c)= \sum_{n} R^{A_{ijn}+c t\left(y_{ijn}\right)}. \end{equation}
Then we can write $f_{i j}(c)$ as
\[ f_{ij}(c) =R^{-\left(S(x_i) +S(z_j)\right)} \sum_{n \in \N}R^{S\left(B\left(\tilde{z}_j+n \hat{e}_s, \tilde{x}_i\right)\right)} R^{A\left(\tilde{z}_j+n \hat{e}_s, x\right)} R^{c t\left(B\left(\tilde{z}_j+n \hat{e}_s, \tilde{z}\right)\right)}. \]
Expanding the exponent on $R$, applying Stokes theorem to find $A\left(\tilde{z}_j+n \hat{e}_s, x\right)$, and expressing the differences $S-S_{\R_4}$ in terms of the degenerate quadratic forms $Q^\beta_i$ allows us to write the coefficient on $R$ on the $n$th term as:
\begin{align*} &-S(x_i)-S(z_j)+S(B(\tilde{z}_j+n \hat{e}_s, \tilde{x}_i)) + A(\tilde{z}_j+n \hat{e}_s, \tilde{x}_i)+c t(B(\tilde{z}_j+n \hat{e}_s, \tilde{x}_i))\\ 
&=  -S(x_i)-S(z_j)+S(B(\tilde{z}_j+n \hat{e}_s, \tilde{x}_i)) + S_{\R^4}(\tilde{x}_i)+S_{\R^4}(\tilde{z}_i+n \hat{e}_s) - S_{\R^4}(B(\tilde{z}_j+n \hat{e}_s, \tilde{x}_i))\\
&\quad\quad \quad \quad +c\, t(B(\tilde{z}_j+n \hat{e}_s, \tilde{x}_i))\\
&= -Q^\beta_0(\tilde{x}_i)-Q^\beta_1(\tilde{z}_j+n\hat{e}_{s}) + Q^{w}_2(B(\tilde{z}_j+n \hat{e}_s, \tilde{x}_i)) + c \,t B(\tilde{z}_j+n \hat{e}_s, \tilde{x}_i). \end{align*} 
A calculation then shows that this expression is equal to
\[ \left(\frac{\left(d_0' m_0 - d_1' m_1\right)\left(d_1' m_1-d_2 'm_2\right)}{2 \left(d_0' m_0-d_2 'm_2\right)}\right) \left(n - w_0\left(\tilde{x}_i\right)+w_1\left(\tilde{z}_j\right)\right)^2 + c t\left(B\left(\tilde{z}_j+n \hat{e}_s, \tilde{x}_i\right)\right). \]

Since $m=\left(d_0 'm_0 - d_1' m_1\right)\left(d_1 'm_1-d_2 'm_2\right)\left(d_0' m_0-d_2 m_2\right)\inv$ is positive by assumption, $f_{i j}(c)< \infty$ for all $c$. This proves that $\mu^2\left(z_j, x_i\right)$ defines an element of the completion. 

Now, for $x \in \whCF_{\underline{r}}^*\left(L_{0}, L_{1}\right)$, $z \in \whCF_{\underline{r}}^*\left(L_{1}, L_{2}\right)$, write 
\begin{align*} & x= \sum_{i} a_i x_i \\
		      &z =  \sum_{j} b_j z_j.  \end{align*}
		    
Define
    \[ f_{x,z}(c)= \sum_{i,j,n \in \Z} \ev_{\underline{r}}\left(a_i\right) \ev_{\underline{r}}\left(b_j\right) R^{A_{ijn}+c y_{ijn}}; \]
we wish to show that $f_{x,z}(c)< \infty$ for all $c \in \R$. By the linearity of $B$, there exists $D > 0$ such that $|t\left(B\left(\tilde{z}_j+n \hat{e}_s, \tilde{x}\right)\right)|< D( |t(z)|+D |t(x)| + D|n|+1)$ for all $i,j,n \in \Z$. Then we can write
     \begin{align*} f_{x,z}(c) &=\sum_{i, j,n \in \Z} \left| R^{m \left(n- w_0\left(\tilde{x}_i\right)-w_1\left(\tilde{z}_j\right)\right)^2/2} \ev_{\underline{r}}\left(a_i b_i\right) R^{c  t\left(B\left(\tilde{z}_j+n \hat{e}_s, \tilde{x}_i\right)\right)} \right|\\
     &< R^{-2 m} \sum_{n}\left( R^{m n^2/2} R^{-3 -| D c n|} \left(\sum_{i,j} \left| \ev_{\underline{r}}(a _i b_j) R^{D c |t\left(x_i\right)|}R^{D c |t\left(z_j\right)|}\right) \right| \right) \\
   &\leq R^{-2 m}  \left(\sum_{n} R^{m n^2/2} q^{-3-|D c n|}\right)
 \left(\sum_{i} \left |\ev_{\underline{r}}( a _i) R^{- | D c t\left(x_i\right)|} \right| \right)
 \left(\sum_{j} \left |\ev_{\underline{r}}(b_j) R^{- |D c t\left(b_j\right)|} \right| \right)   \\
     & < \infty. \end{align*}

To prove the general case, note that there exist type-$\Wr$ Hamiltonians $H_0$, $H_1$, and $H_2$ such that the set $\left(\psi_{H_0}(L_0), \psi_{H_1}(L_1), \psi_{H_2}(L_2)\right)$ is in correct position in $\F\left(\Wr\right)$ and such that $\psi_{H_0}(L_0), \psi_{H_1}(L_1)$ and  $\psi_{H_2}(L_2)$ satisfy the conditions prescribed above (so each Lagrangian is planar if the Lagrangians are essentially tranverse, and are small perturbations of planar Lagrangians otherwise).  Let \begin{align*} &\Phi_2: \whCF_{\underline{r}}^*\left(L_0, L_2\right) \to \whCF_{\underline{r}}^*\left(\psi_{H_0}(L_0), \psi_{H_2} \left(L_2\right)\right) ;\\ 
&\Phi_1: \whCF_{\underline{r}}^*\left(L_1, L_2\right) \to \whCF_{\underline{r}}^*\left(\psi_{H_1}\left( L_1\right), \psi_{H_2} \left(L_2\right)\right) ;\\
&\Phi_0: \whCF_{\underline{r}}^*\left(L_0, L_1\right) \to \whCF_{\underline{r}}^*\left(\psi_{H_0}(L_0), \psi_{H_1} \left(L_1\right)\right),\end{align*}
 denote the quasi-isomorphisms from multiplication furnished by the continuation elements. Then
\[ \left[\Phi_{H_2} \left(\mu^2\left(z,x\right)\right)\right]=\left[\mu^2\left(\Phi_1\left(z\right) ,\Phi_0\left(x\right)\right)\right] \]
as classes in $HF^*(\psi_{H_0}(L_0), \psi_{H_1}(L_1))$.
We have shown that\[\mu^2\left(\Phi_1\left(z\right), \Phi_0\left(x\right)\right) \in \whCF^*_{\underline{r}}\left(\psi_{H_0}(L_0),\psi_{H_1}\left(L_2\right)\right).\] Appealing to Proposition \ref{compwelldef}, we deduce that $\Phi_{H_2} \left(\mu^2\left(z,x\right)\right) \in \whCF^*_{\underline{r}}\left(\psi_{H_0}(L_0),\psi_{H_2}(L_2)\right)$; then the same proposition allows us to conclude that $\mu^2\left(z,x\right)$ is an element of the completion  $\whCF^*_{\underline{r}}\left(L_0,L_2\right)$.

\end{proof}

\subsection{Unitary flat connections}We will enlarge our Fukaya categories by equipping our Lagrangians with local systems (see \cite{AuB}). Let $\nabla$ be a unitary flat connection on the trivial rank-one vector bundle $V_0 \to L$, and denote this data by $\left(L, \nabla\right)$. Define
\[ \whCF^*\left(\left(L, \nabla \right), \left(L', \nabla' \right)\right)= \bigoplus_{y \in L \cap L'} \Hom\left(V_0|_y, V_0'|_y\right) \]
and weight each disc $u$ contributing to a product $\mu^n$ by the monodromy $e^{2 \pi i \nu}: (V_{0})_0|_{y_{\out}} \to (V_0)_n|_{ y_{\out}}$  obtained by parallel transport around its boundary: 
\[ \mu^n\left( y_{n-1}, \ldots, y_{0}\right) = \sum_{[u]} \#\mathcal{M}\left( y_0, \ldots, y_{n-1}, y_{\out}\right) T^{\omega[u]} e^{2 \pi i \nu} y_{\out} \]
Since $\nabla$ is unitary, adding such coefficients does not affect the analysis of the convergence of the $\mu^n$ over $\Lambda$ or over $\C$; we can thus define a completion $\widehat{CF}^*\left(\left(L, \nabla \right), \left(L, \nabla' \right)\right)$ as in Equation \eqref{completion}. (In general, in all proofs of convergence results, we will take Lagrangians without local systems.)

We say that a collection of Lagrangians with local systems $\{\left(L_0, \nabla_0\right), \ldots, \left(L_n, \nabla_n\right)\}$ is in correct position if the set $\{\left(L_0, \ldots, L_n\right)\}$ is in correct position. We can relate Floer complexes with continuation elements defined in the same way above, but now weighting by the parallel transport of the connections. 

\begin{prop}\label{kingprop} Let $L$ and $L'$ be in $\G_P\left(I\right)$. Let $\left(\psi_{H}\left(L\right), \nabla\right)$ and $\left(\psi_{H'}\left(L'\right), \nabla' \right)$ be in correct position, and assume that $\alpha_{\pm}\left(H_0\right) - \alpha_{\pm} \left(H'\right) > K_{\pm}\left( L, L'\right)$ on the wrapped end(s) of $Y$. Let $\hat{n} \in \Z^3$. Then if  $\tilde{L}$ and $\tilde{L'}$ are $I$-essentially transverse,  
\[ HF^*\left(\left(\psi_{H}\left(L\right), \nabla\right), \left(\psi_{H'}\left(L'\right), \nabla'\right)\right)_{\hn} = \Lambda. \]
If   $\tilde{L}$ and ${\tilde{L'}-\hn}$ are $I$-essentially non-transverse and the monodromies of $\nabla$ and $\nabla$ around the loop $[\left(1,0,0\right)] \in \pi_1\left(Y\right)=\pi_1\left(T^3\right)$ agree, then 
\[ HF^*\left(\left(\psi_{H}\left(L\right), \nabla\right), \left(\psi_{H'}\left(L'\right), \nabla'\right)\right)_{\hn} = \Lambda \oplus \Lambda \]
where there is a generator in degrees $i$ and $i+1$ for $i=0$ or $1$. Otherwise 
\[  HF^*\left(\left(\psi_{H}\left(L\right), \nabla\right), \left(\psi_{H'}\left(L'\right), \nabla'\right)\right)_{\hn} = 0. \]
\end{prop}										
\begin{proof} We have seen that if  $\tilde{L}$ and ${\tilde{L'}-\hn}$ are $I$-essentially transverse,  there exists some type-$I$ Hamiltonian $H''$ such that $\left|\psi_{H''} \circ \psi_{H} (\tilde{L}) \cap \psi_{H'} \left(\tilde{L'}\right)-\hn \right|=1$; then  
\[  HF^*\left(\left(\psi_{H}\left(L\right), \nabla\right), \left(\psi_{H'}\left(L'\right), \nabla'\right)\right)_{\hn} \cong  HF^*\left(\left(\psi_{H''} \circ \psi_{H}\left(L\right), \nabla\right), \left(\psi_{H'}\left(L'\right), \nabla'\right)\right)_{\hn} = \Lambda. \]

If $\tilde{L}$ and $\tilde{L'}$ are $I$-essentially non-transverse, we have seen there exists some type-$I$ Hamiltonian $H''$ such that $CF^*\left(\psi_{H''} \circ \psi_{H} (\tilde{L}) , \psi_{H'} \left(\tilde{L'}\right)\right)_{\hn}= \Lambda y_i \oplus \Lambda y_{i+1}$, with generators of degree $i$ and $i+1$ for $i=0$ or $1$. Denote the monodromies of  $\nabla_0$ and $\nabla_1$ around the loop $[\left(1,0,0\right)] $ by $e^{2 \pi i \nu}$ and $e^{2 \pi i \nu'}$ respectively. Then \[\mu^1(y_{i})= T^{A} (e^{2 \pi i \nu}-e^{2 \pi i \nu'});\] note that this differential vanishes if and only $e^{2 \pi i \nu}=e^{2 \pi i \nu'}$. 
\end{proof}

\begin{defn}[Complexified action corrections]  Let $(L_0, \nabla_0), (L_1, \nabla_1) \in \G_{P}$. Let $w_{L_0, L_1}$ denote the coordinate function associated to the pair of Lagrangians.  Then we can always decompose   $\nabla_0 - \nabla_1$ as
\begin{equation}\label{nablade} \nabla_0 - \nabla_1 = 2 \pi i( B \mathrm{d} w_{L_0,L_1} + A \mathrm{d }t). \end{equation}
For $y$ a generator of $CF^*(L_0, L_1)$, define the \emph{complexified action-corrected generator} 
\[ \mathbf{y}=e^{-2 \pi i A t(y)}T^{-S(y)}y. \]
\end{defn}

\section{Fukaya categories}\label{fukcatsec} 
\subsection{Definition of Fukaya categories} We define three variants of the Fukaya categories via the localization approach found in unpublished work of Abouzaid and Seidel (\hspace{-.2pt}\cite{whatisthis}), which has recently been used in other works, e.g. \cite{GPS1}. This approach is well-suited to our purposes as it allows us to relate our categories easily through localization maps following the work of Ganatra, Pardon, and Shende in \cite{GPS2}. 

For $I= \mathbf{A}$, $\left(m,k\right)$, or $\W$, consider a set of Lagrangians $\{L_{\sigma}\}_{\sigma \in  \Sigma_{I}}$ with $L_{\sigma} \in \G\left(I\right)$ which contains at least one representative of each class of type-I Lagrangians up to compactly supported Hamiltonian isotopies. Then choose for each $L_\sigma$, $\sigma \in \Sigma_I$, an $I$-cofinal sequence $H_{\sigma, 0}, H_{\sigma, 1}, \ldots$ of perturbing type-$I$ Hamiltonians with the property that any set
\[ \{ \psi_{H_{\sigma_0, i_0}} \left(L_{\sigma_0}\right), \psi_{H_{\sigma_1, i_1}} \left(L_{\sigma_1}\right), \ldots, \psi_{H_{\sigma_k, i_k}}\left(L_{\sigma_k}\right)\}\]
with $i_0 > i_1 > \ldots > i_k$ is in correct position. By $I$-cofinal we mean a sequence of perturbing type-I Hamiltonians with the property that for any generic $L' \in \G_P\left(I\right)$, there exists $n \in \N$ such that the set $\left(\psi_{H_{\sigma,n}}\left(L_{\sigma}\right), L'\right)$ is in correct position.

\subsubsection{Definition of $\mathcal{O}_I$} Let $\mathcal{O}_I$ be the $A_{\infty}$ category with objects $\left(L_{\sigma}, \nabla, i\right)$ where $\sigma \in \Sigma_I$, $\nabla$ is a flat unitary connection on $V \to L$ the rank-one trivial $\C$-bundle, and $i \in \N$.
In the following, we often simplify notation by writing $\left(L_0, \nabla, i\right)$ as an object $L$; when unambiguous, we refer to the entire set as well as its grading data and $\mathrm{Spin}$ structures as the ``Lagrangian." The reader should keep in mind that when referring to the Lagrangian, we mean this entire set of data.

Let $\left(L_{\sigma_0}, \nabla_0, i_0\right) > \left(L_{\sigma_1}, \nabla_1, i_1\right)$ when $i_0 > i_1$. The morphisms in the directed category $\mathcal{O}_I$ are defined by:
\begin{align*} &\mathcal{O}_I\left(\left(L_{\sigma_0}, \nabla_0, i_0\right), \left(L_{\sigma_1}, \nabla_1, i_1\right)\right)\\
&\quad = \begin{cases}  \wh{CF}^*\left(\left(\psi_{H_{\sigma_0,i}}\left( L_{\sigma_0}\right), \nabla\right), \left(\psi_{H_{\sigma_1},j}\left( L_{\sigma_1}\right), \nabla_1\right)\right) &  i_0 > i_1  \\
							\Lambda \{e\} & \left(L_{\sigma_0}, \nabla_0, i_0\right) = \left(L_{\sigma_1}, \nabla_1, i_1\right) \\
							0 & \textrm{otherwise}. \end{cases}\end{align*}
Here $\{e\} \in \whCF^*\left(\left(\psi_{H_0}\left(L\right), \nabla_0\right), \left(\psi_{H_0}(L_0), \nabla_0\right)\right)$ is defined to be a formal element of degree zero. 
The higher products 
\[ \mu^n:  \mathcal{O}_I\left(L_{n-1},L_n\right) \otimes \cdots \otimes \mathcal{O}_I\left(L_0, L_1\right)\to \mathcal{O}_I \left(L_0, L_n\right) \]
are defined when $L_0 \geq L_1 \geq \ldots \geq L_n$, using compatible choices of strip-like ends and almost-complex structures as described in \ref{modspaces}. As previously noted, the $\mu^n$ will satisfy the $A_\infty$ relations for any finite collection of Lagrangians where each pair $\left(L_i, L_j\right)$ with $i > j$ is in correct position.

Let $\sigma \in \Sigma$ and let $\nabla$ be a flat unitary connection on $L_{\sigma}$. Let $L_{\sigma}^i$ be shorthand for the data $\left(\psi_{H_{\left(\sigma, i\right)}}\left(L_{\sigma}\right),\nabla\right)$, when used in Floer complexes.

\subsubsection{Localization} For $j>i$, let 
\[ c_{\left(\left(L_\sigma, \nabla\right),j,i\right)} \in \whCF^*\left(L_{\sigma}^{j},L_{\sigma}^{i}\right) \]
denote the continuation element. Denote the set of all such quasi-units by $Q_I$:
\[ Q_I=\{c_{\left(\left(L, \nabla\right), j,i \right)}\}_{\left(L, \nabla\right) \in O\left(I\right), j > i}. \]

We define $\F\left(I\right)=\O_I[Q\inv]$ by taking the $A_{\infty}$ quotient category in the sense of Lyubashenko-Ovsienko (\hspace{1sp}\cite{LyOv}). This defines a category with the same objects as $\mathcal{O}_I$. The morphisms between two objects are calculated by considering the composition of functors.
\[ \O_I \to \mathrm{Tw} \, \O_I \to  \mathrm{Tw}\, \O_I/ \textrm{Cones}\left(Q_I\right). \]

For any pair $\left(L_{\sigma}, L_{\sigma'}\right) \in \Ob\, \F\left(I\right)$, the $\Lambda$-modules $\wh{HF}^*\left(L_{\sigma}^{i+j}, L_{\sigma'}^j\right)$, $i \geq 0$, form a directed system where the maps between objects are furnished by the continuation elements. We use the following lemma from the work of Ganatra-Pardon-Shende (Lemma 3.57 in \cite{GPS1});
\begin{prop}
\[ H^*\Hom_{\F\left(I\right)}\left(L_{\sigma}, L_{\sigma'}\right)\cong \lim_{i \to \infty} \wh{HF}^*\left(L_{\sigma}^{j+i}, L_{\sigma'}^{j}\right). \]
for any $j$. \end{prop}

\begin{remark} While we use completed Floer complexes in every variant of our Fukaya categories, the morphism complexes that appear for the unwrapped category are finitely generated free $\Lambda$-modules, so the use of the completion is immaterial in that instance. In the partially wrapped category we have seen that the completion is not necessary in the sense: the subspace of $\widehat{CF}^*$ consisting of finite sums of generators is closed under the $\mu^n$. In contrast, it is essential to work with the completion for the fully wrapped category. \end{remark}

\subsection{Invariance of choices}{\label{inv}}
\begin{prop} The  choices involved in the construction of $\F\left(I\right)$ are immaterial: different choices of strip-like ends, complex structures, and 1-forms $\beta$ used to define continuation elements will produce quasi-isomorphic $A_\infty$ categories. \end{prop}
\begin{proof} This is a standard fact that follows from the contractibility of the space of choices of such data; see \cite{GPS1}. \end{proof}

The following proposition and its proof is identical to Proposition 3.39 in \cite{GPS1}; we record the proof here since it will also be used below.
\begin{prop} The $A_\infty$ quasi-equivalence class of $\F\left(I\right)$ is independent of choices of cofinal sequences or isotopy classes. 
\end{prop}
\begin{proof}The idea is that given two choices $\Sigma_I$, $\Sigma_I'$ and choices of cofinal sequences, we can embed each Fukaya category into a larger category containing both categories, and each embedding will be a quasi-equivalence. 

The observation in \cite{GPS1} is that one can define a wrapped category with only a partial ordering on its elements. Let $N$ be a partially ordered set, and let $\mathcal{O}_I$ be a partially ordered set consisting of Lagrangians $\left(L_{\sigma},\nabla, a\right)$ with $a \in N$, with associated perturbing Hamiltonians $H_{\sigma, a}.$ Assume that $\mathcal{O}_I$ has the property that every totally ordered collection in $\mathcal{O}_I$ is in correct position. Assume also that for each Lagrangian $L$, there is a sequence $L < L^{\left(1\right)} < L^{\left(2\right)} < \ldots $ which is $I$-cofinal. Let $Q_I$ be a set of quasi-units so each continuation element corresponding to the positive perturbations $L^{\left(i\right)} \to L^{\left(j\right)}$ is an element of $Q_I$. Then we may localize $\mathcal{O}_I$ at $Q_I$ to construct a category $\F\left(I\right)$. For each $L,L'\in \O_I$, the maps below are isomorphisms (Lemma 3.37, \cite{GPS1}):
\begin{equation}\label{isos} \lim_{i} \whHF^*\left(L^{\left(i\right)}, L\right) \to \lim_{i} H^*\Hom_{\F\left(I\right)} \left(L^{\left(i\right)}, L\right) \leftarrow H^*\left(\Hom_{\F\left(I\right)}\left(L,L'\right)\right). \end{equation}

Now, let $\mathcal{O}_I$, and $\mathcal{O}_I'$ denote two such categories constructed with two choices of $\Sigma_I$. Let the objects of $\mathcal{O''}_I$ consist of the set $\Z_{+} \cup \mathcal{O}_I \sqcup \mathcal{O}_I$ with the lexicographical partial order. Find a sequence of Hamiltonians for each object in $\mathcal{O}_I \cup \mathcal{O}'_I$ which is cofinal with respect to this partial ordering; let $Q_I''$ be the corresponding continuation elements. Let $\F''\left(I\right)$ be the localization of $\mathcal{O}''$ at $Q_I''\cup Q_I \cup Q_{I}''$. Then there are natural functors $\F\left(I\right) \to \F''$ and $\F'\left(I\right) \to \F''\left(I\right)$, each of which is a quasi-equivalence by Equation \eqref{isos}.
\end{proof}

\subsection{Localization functors}Let $\A=\left(\left(m_0,k_0\right),\left(m_{\infty}, m_{k_\infty}\right)\right)$. We now define \emph{localization functors}
\[ \F \left(\mathbf{A}\right) \to \F \left(m_0, k_0\right) \to \F_{\W}. \]

The functor $\F\left(\mathbf{A}\right) \to \F \left(m_0, k_0\right)$ is constructed as follows: given $\mathcal{O}_\A$ and a set of quasi-units $Q_\A$, we construct a category $\mathcal{O}_{\left(m_0, k_0\right)}$ with a partial ordering and a set of continuation morphism $Q_{\left(m_0,k_0\right)}$ such that each element $\left(L,\nabla, i\right)\in \mathcal{O}_\A$ is an element of $\mathcal{O}\left({m_0,k_0}\right)$, and $Q_\A \subset Q_{\left(m_0,k_0\right)}$, and the choice of strip-like ends and complex structures are compatible. We can then use this set of data to define $\F\left(m_0, k_0\right)$. Similarly, we choose $\mathcal{O}_\W$ to include $\mathcal{O}_{\left({m_0,k_0}\right)}$ and $Q_{\left(m_0,k_0\right)}  \subset  Q_\W$. Then there exist localization morphisms 
 \begin{align} \label{loca} \iota_{A}: \F \left(\mathbf{A}\right) & \to \F\left(m_0, k_0\right); \\  
					 \iota_{\left(m,k\right)}: \F\left(m,k\right) &\to \F_{\W} \end{align}
which are the inclusion on objects. 
(Compare \cite{GPS2}, Section 5.)

\begin{remark} \label{inversion}Let $\Psi$ be the symplectomorphism $Y \to Y$ given by  $\left(t,s, \phi_t, \phi_s\right) \mapsto \left(-t, s, -\phi_t, \phi_s\right)$. Then $\Psi$ induces equivalences of categories  \[\F_{\W} \to \F_{\W}\] and \[\F\left(\left(m_0,k_0\right), \left(m_{\infty},k_{\infty}\right)\right) \to \F\left(\left(m_\infty,k_\infty\right), \left(m_{0},k_{0}\right)\right).\]

When all categories are defined via compatible choices, the following diagram commutes:

\[
\begin{tikzcd}
&\F\left(m_0, k_0\right)\arrow{rd}{i_{\left(m_0,k_0\right)}}&\\
\F\left(\mathbf{A}\right)  \arrow{ru}{i_A} \arrow{rd}{i_A \circ \Psi^*} && \F_{\W} \\
&\F\left(m_\infty, k_{\infty}\right) \arrow{ru}[swap]{\Psi^* \circ i_{\left(m_\infty,k_\infty\right)}} & \\
\end{tikzcd}\label{pullbackdiag}
\]
In particular, there is a diagram of categories
\[
\begin{tikzcd}
&H^0 \F\left(m_0, k_0\right)\arrow{rd}&\\
H^0\F\left(\mathbf{A}\right)  \arrow{ru} \arrow{rd}&& H^0 \F_{\W} \\
&H^0 \F\left(m_\infty,k_{\infty}\right) \arrow{ru} & \\
\end{tikzcd}\label{pullbackdiag2}
\]
independent of the choices involved in the creation. 
\end{remark}

\begin{remark}\label{catequiv} \begin{enumerate}[(a)]
\item For all $n \in \Z$, there is an equivalence of $A_{\infty}$ categories \[I_n: \F\left(m_0, k_0\right) \to \F\left(m_0,k_0+n m_0\right)\] for $n \in \Z$.
\item For all $n \in \Z$, there is an equivalence of $A_{\infty}$ categories \[ I_n: \F\left(\left(m_0,k_0\right),\left(m_\infty,k_{\infty} n\right)\right)\to \F\left(\left(m_0,k_0+ n m_0\right),\left(m_{\infty}, k_{\infty}-n m_{\infty}\right)\right).\] Moreover, for each $n$ the following diagram commutes: 
\begin{equation}
\begin{tikzcd} 
& \F\left(\left(m_0,k_0\right),\left(m_\infty,k_{\infty}\right)\right)\arrow{r}{I_n} \arrow[d]  &\F\left(\left(m_0,k_0+ n m_0\right),\left(m_{\infty}, k_{\infty}-n m_{\infty}\right)\right)\arrow{d}\\
& \F\left(m_0, k_0\right) \arrow{r}{I_n} & \F\left(m_0,k_0+n m_0\right) \\ 
 \end{tikzcd}
\end{equation}
where the vertical arrows are the localization functors.
\end{enumerate}The equivalences $I_n$ are induced by the symplectomorphism $Y \to Y$ given by \[\left(t,s,\phi_t, \phi_s\right) \mapsto \left(t, s- n t, \phi_t + n \phi_s, \phi_s\right);\] some slight care must be taken to ensure that the perturbing Hamiltonians used to define the category are sent to perturbing Hamiltonians.
 \end{remark}

\begin{defn} For $\tau \in \C$ with $\im \tau > 0 $, define $H^0 \F_\tau \left(I\right)$ to be the $\C$-linear cohomology category where the morphisms are defined via the map $\ev_\tau: \Lambda_{e^{-2 \pi i \im \tau}} \to \C$. \end{defn}

The following lemma will be essential as we perform calculations in our categories:
\begin{prop}\label{htc} There is a quasi-isomorphism
\[  \widehat{C F}_{\tau}^*\left(\psi_{H}\left(L_{\sigma}\right),\psi_{H'}\left(L_{\sigma'}\right)\right) \cong \Hom_{\F_{\tau} \left(I\right)}^*\left(L_{\sigma}, L_{\sigma}'\right) \]
for any $I$-type $H,H'$ such that the pair $\left(\psi_{H}\left(L_{\sigma}\right),\psi_{H'}\left(L_{\sigma'}\right)\right)$ is in correct position and  $H$ and $H'$ satisfy the conditions of Lemma \ref{qisomcont}. \end{prop} 
\begin{proof}  By the $I$-cofinality of the sequences $\{L_{\sigma}^{\left(n\right)}\}_{n \in \N}$ we can always find $n > 0$ such that the set 
\[\left(\psi_{H_{\sigma, n}}\left(L_{\sigma}\right), \psi_{H}\left(L_{\sigma}\right), \psi_{H'}\left(L_{\sigma'}\right)\right)\] is in correct position. 
Then by Lemma \ref{qisomcont}, multiplication by the continuation elements give quasi-isomorphisms
\[  \widehat{C F}_\tau^*\left(\psi_{H}\left(L_{\sigma}\right),\psi_{H'}\left(L_{\sigma'}\right)\right) \to  \widehat{C F}_\tau^*\left(L_{\sigma}^n,\psi_{H'}\left(L_{\sigma'}\right)\right) \to \whCF_\tau^*\left(L_{\sigma}^n,L_{\sigma'}\right). \]
Moreover, by the same lemma we know that for all $n' > n$, the maps  \[\whCF_\tau^*\left(L_{\sigma}^n,L_{\sigma'}\right) \to \whCF_\tau^*\left(L_{\sigma}^{n'},L_{\sigma'}\right)\] are quasi-isomorphisms. Thus we can use Equation \eqref{isos} to deduce that 
the map \[ \whCF_\tau^*\left(L_{\sigma}^n,L_{\sigma'}\right) \to \Hom_{\F_\tau\left(I\right)}^*\left(L_{\sigma}, L_{\sigma'}\right)\]
is a quasi-isomorphism.
 \end{proof}

\begin{remark}[Stops]\label{stopform} The category $\F\left(\mathbf{A}\right)$ could equivalently be defined in terms of stops \`{a} la Sylvan (\hspace{1sp}\cite{SyPW}) or Ganatra, Pardon and Shende (\hspace{1sp}\cite{GPS2}), after reworking some definitions to allow for non-Liouville manifolds.The``boundary at $\infty$" of $Y$ is given by 
\[\left( \{-1\} \sqcup \{1\}\right) \times \left(S^1\right)^3= \left(S^1\right)^3 \sqcup \left(S^1\right)^3\] with coordinates $\left(s, \phi_t, \phi_s\right)$ at each end. We insert two $\left(S^1\right)^2$ stops given by the set of points 
\[ \{ \left(s, \phi_t, \phi_s\right) \mid  m_0 \phi_t + k_0 \phi_s =0 \}. \]
at the negative end and 
\[ \{ \left(s, \phi_t, \phi_s\right) \mid  -m_\infty \phi_t + k_\infty \phi_s =0\}.  \]
at the positive end. This reflects the Landau-Ginzburg heuristics found in Section \ref{LG}. In this setting, the localization maps can be thought of as ``stop removal" maps.

 \end{remark}

\section{Categories of line bundles}\label{lbcx}
The main theorems of this work exhibit a correspondence between categories of Lagrangian sections in $\F$ and categories of line bundles on the complex surfaces $S_{\tau}(\A)$, $D_{\tau}\left(m,k\right)$, and $\C^* \times E_{\tau}$. For a complex manifold $X$, $\Pic_{\an}\left(X\right)$ can be given the structure of a category by taking $\Hom\left(\L, \L'\right)=  \Ext^0\left(\L, \L'\right)$, where the composition of morphisms is given by the Yoneda product. The categories of line bundles we consider are subcategories of $\Pic_{\an}\left(X\right)$; in the case of the open surfaces we consider the category of all line bundles which can extend to a compactification $S_{\tau}(\A)$, and in the case of the compact surfaces we take our category to be the entire analytic Picard group.  

\subsection{Open surfaces}\label{opsurf} We first investigate line bundles on the open elliptic surfaces $D_{\tau}\left(m, k\right)$ and $\C^* \times E_{\tau}$; we will exhibit admissible Lagrangians mirror to these line bundles. Our methods are inspired by Polischuk and Zaslow's work on mirror symmetry for the elliptic curve (\hspace{1sp}\cite{PZ}), and we largely follow the notation in this work throughout. \footnote{\label{foot}The most notable difference is that we reverse the sign conventions in the mirror symmetry correspondence: where we consider Lagrangian $L$ on $E_{\tau}$ in this work, the authors of \cite{PZ} would consider the Lagrangian $-L$.  We also change notation for the translation factor: our $\tau \eta + \nu$ is their $\tau \alpha+ \beta$.}
Let $\L_{\tau,0}$ be the degree one line bundle on $E_\tau$ called $L$ in \cite{PZ} with total space
\[ \mathrm{Tot}\, \left( \L_{\tau,0}\right) =\left(\C^* \times \C / \left(x,t \right) \sim \left(e^{2\pi i \tau} x, e^{-\pi i \tau} x^{-1} t \right)\right).\] 

\subsubsection{Definition of the line bundles $\L_\tau\left(m,k,d,\xi\right)$} For a pair of coprime integers $\left(m,k\right)$ with $m > 0$, let $g_{m}: D_\tau\left(m,k\right) \to \C \times E_{m\tau}$ denote the degree-$m$ map defined by $\left(v,w\right) \mapsto \left(v^m,w^m\right)$. As in Section \ref{classi}, let $i_{m,k}: \C^* \times E_\tau \to D_\tau\left(m,k\right)$ denote the inclusion map. Now consider the map
\[\left(\pi_{E_{m \tau}} \circ g_m \circ i_{m,k}\right)^*:  \Pic_{\an} E_{m \tau} \to \Pic_{\an} \C^* \times E_{\tau}. \]  
Recall that any degree $d$ line bundle  $\L \in \Pic_{\an} E_{m \tau}$ can be written as  $T_{\xi\inv}^*\left(\L_{m\tau, 0}\right) \L_{m \tau, 0}^{d-1}$ where $T_{\xi\inv}$ denotes the translation map $z \mapsto \xi\inv z$ (written multiplicatively), and $\xi \in \C^*$.
Let $(\eta, \nu)$ be in $\R$. Define
 \begin{align*} \overline{\L_\tau\left(m,k,d, e^{2\pi i (\tau \eta+\nu)}\right)}&= \left(\pi_{E_{m \tau}} \circ g_m \right)^* \left(T_{e^{-2\pi i m(\tau \eta+\nu)}}^*\left(\L_{m \tau, 0}\right)\cdot \L_{m \tau, 0}^{d-1}\right) \in \Pic_{\an} D_\tau\left(m,k\right) \\
\L_\tau\left(m,k,d, e^{2\pi i (\tau \eta+\nu)}\right) &= \left(\pi_{E_{m \tau}} \circ g_m \circ i_{m,k}\right)^*\left( T_{e^{-2\pi im (\tau \eta+\nu)}}^*\left(\L_{m \tau, 0}\right)\cdot \L_{m \tau, 0}^{d-1} \right) \in \Pic_{\an} \C^* \times E_{\tau}. \end{align*}

In coordinates, \[ \left(g_m \circ i_{m,k}\right)\left(z,x\right)= \left(z, z^{-k} x^m\right).\]  
Let $\xi= e^{2\pi i (\tau \eta+\nu)}$.  Letting $i_{z}: E_{\tau} \to \C^* \times E_{\tau}$ denote the inclusion map of the fiber over $z \in \C^*$, we have the following: 
\begin{equation}\label{fiber-wise}i_{z}^*( {\L_\tau\left(m,k,d,\xi\right)})=  T_{\xi\inv z^{k d}}^*\left( \L_{\tau,0} \right)\cdot \L_{\tau,0}^{m d-1}. \end{equation}
We use this fiber-wise description to extend the notation ${\L_\tau\left(m,k,d,\xi\right)}$ to the case where $\left(m,k\right)$ is not primitive in $\Z^2$, so ${\L_\tau\left(m,k,d,\xi\right)}$ describes a line bundle that restricts in this way to each fiber; this is justified with the following lemma:



 \begin{lemma}\label{triv} If a line bundle $\mathcal{L}$ on $\C^* \times E_{\tau}$ satisfies $i^*_z \L= \mathcal{O}_{E}$ for all $z$, then $\mathcal{L} = \mathcal{O}$. \end{lemma}
\begin{proof} This follows from the fact that $\Pic_{\an} \C^* = 0$. Let $\mathcal{L}$ be such a line bundle. Consider the restriction $\mathcal{L}|_{\C^* \times \{0\}}$; let $s: \C^* \times \{0\} \to \mathcal{L}$ be a non-vanishing section. On each fiber there exists a unique section $s_z \in H^0(E, i_z^*\mathcal{L})$ such that $s_z\left(0\right) = s(z)$. Then we can define a global non-vanishing section $\sigma \in H^0(\C* \times E, \mathcal{L})$ by $\sigma\left(z,t\right) = s_z\left(t\right)$. \end{proof}

\begin{remark} \label{ambig} Allowing $\left(m,k\right)$ to be non-coprime is useful, but it introduces a source of non-uniqueness in the labelling beyond that of the choice of $\xi$, namely that $\L\left(m j, k j,d,\xi\right)=\L\left(m ,k ,d j,\xi\right)$ for any integer $j$. {In particular, $\L\left(0, 0, d,\xi\right)=\L\left(m, k, 0,\xi\right)$ for any $(m,k)$.} Note also that this implies that we may always rewrite the label on $\L\left(m j, k j,d,\xi\right)=\L\left(m ,k ,d j,\xi\right)$ so that $\left(m,k\right)$ are co-prime and $m>0$.  \end{remark}

\begin{lemma}\label{mult} Assume that either $m_1 d_1 + m_2 d_2 \neq 0$, or  $m_1 d_1 + m_2 d_2 = k_1 d_1+k_2 d_2 = 0$. Then
\[ \L\left(m_1,k_1,d_1,\xi_1\right) \otimes \L\left(m_2,k_2,d_2,\xi_2\right)=\L\left(m_1 d_1+m_2 d_2,  k_1 d_1+k_2 d_2, 1, \xi_1\xi_2\right) .\]
In the case where $d_1=d_2=d$,
 \[ \L\left(m_1,k_1,d,\xi_1\right) \otimes \L\left(m_2,k_2,d,\xi_2\right)=\L\left(m_1+m_2,k_1+k_2, d, \xi_1 \xi_2\right).\] 
 In the case where $k_1=k_2=k$ and $m_1=m_2=m$, 
  \[ \L\left(m,k,d_1,\xi_1\right) \otimes \L\left(m,k,d_2,\xi_2\right)=\L\left(m,k,d_1+d_2, \xi_1\xi_2\right).\] 
\end{lemma}
 \begin{proof} From Equation \ref{fiber-wise}, we know that \[ \L\left(m_1,k_1,d_1,\xi_1\right) \otimes \L\left(m_2,k_2,d_2,\xi_2\right)\otimes \L\left(k_1 d_1+k_2 d_2,m_1 d_1+m_2 d_2,1, \xi_1\xi_2\right)\inv|_{E_{\tau}} = \O|_{E_\tau};\] the result then follows from  Lemma \ref{triv}. 
 \end{proof}

\begin{prop}\label{extension} 
Let $\mathcal{L}$ be a line bundle on $\C^* \times E_\tau$ which is degree zero when restricted to the fiber. Then $\mathcal{L}$ gives a section of the relative Jacobian $\mathcal{J}\left(\C^* \times E_\tau/ \C^*\right)= \mathcal{J}\left(E_\tau\right) \times \C^*$, i.e. (via the Abel-Jacobi map) a holomorphic map  $f: \C^* \to E_\tau$.  $\mathcal{L}$ extends to a line bundle over $\C \times E_\tau$ if and only if the map $f$ extends across the origin.
\end{prop}
\begin{proof}
The forward direction is clear. Let $C \subset \mathcal{J}\left(\C^* \times E_\tau/ \C^*\right)$ be the graph of $f$. Let $\overline{C} \subset \C \times E_{\tau}$ be the compactification of $C$ determined by the extension of $f$.  Let $\pi_2$ denote the projection \[ \pi_2: \overline{C} \times_{\C} \left(\C \times E_\tau \right) \to \C \times E_\tau.\] Let $\mathcal{P}$ be the Poincar\'{e} line bundle on $\mathcal{J}\left(\C \times E_\tau/ \C\right) \times_{\C} \left(\C \times E_\tau\right)$ .Then the map $\pi_2$ is an isomorphism, so $\pi_{2*}\left( \mathcal{P}|_{\overline{C}}\right)$ is a line bundle on $\C \times E_\tau$. Moreover, $i_z^* \left(\pi_{2*}\left(\mathcal{P}|_{\overline{C}}\right)\right)=i^*_z \mathcal{L}$ by construction, so by Lemma \ref{triv}, $\pi_{2*}\left(\mathcal{P}|_C\right) = \mathcal{L}$. 
\end{proof}


\begin{cor} Parametrize $Pic^d\left(E_\tau\right)$ via the isomorphism $E \to Pic^d\left(E_\tau\right)$ given by $z \mapsto T_{z}^* \L_{\tau,0} \cdot \L_{\tau,0}^d$. Then a line bundle $\L \in \Pic_{\an}\left(\C \times E_\tau^*\right)$ which is degree $d$ on each fiber defines a map $f: \C \to E_{\tau} \cong Pic^d\left(E_{\tau}\right)$. Such a line bundle $\L$ will extend to a line bundle on $\C \times E_{\tau}$ if and only if the function $f$ extends across the origin. \end{cor}

\begin{cor} \label{1683} Let $\mathcal{L}$ be a line bundle on $\C^* \times E_\tau$ of degree $d$ on each fiber. Then there exists a line bundle $\overline{\L}$ on $D_\tau\left(m,k\right)$ with  $i_{(m,k)}^*\overline{\mathcal{L}}=\L$ if and only $m$ divides $d$ and the map $f: \C^* \to E_\tau \cong \Pic^d E_\tau$ has the property that $z^{k d} f\left(z^m\right): \C^* \to \Pic^d \left(E_\tau\right) \cong E_\tau$ defines a holomorphic map which extends across the origin. 
\end{cor}
\begin{proof} Recall that the divisor class of the central fiber is $\frac{1}{m}[F]$, where $[F]$ is the class of the generic fiber; thus any line bundle which extends across the central fiber must have degree divisible by $m$ on the generic fiber.

Assume that $\L$ is degree $d=d' m$ on each fiber. First assume that $\left(i_{m,k}\inv\right)_*\mathcal{L}$ extends to $\overline{\L} \in \Pic_{\an} D_{\tau}\left(m,k\right)$.  Recall that $D_{\tau}\left(m,k\right)= \C \times E_\tau/ \left(\left(v,w\right) \sim \left(e^{2 \pi i/m} v, e^{-2 \pi i k/m} w\right)\right)$; pull back $\overline{\L}$ to $\overline{\L}'\in \Pic_{\an} \C \times E_{\tau}$ via the quotient map. Then we obtain a holomorphic map
 \[h: \C \to E \cong \Pic^d \left(E_\tau\right). \]
On $\C^* \subset \C$, \[ h\left(v\right) = v^{k d} f\left(v^m\right);\]
we conclude that $z^{k d} f\left(z^m\right)$ extends across the origin.

Conversely, assume that $m$ divides $d$ and $h\left(z\right)=z^{k d} f\left(z^m\right)$ extends across the origin. Consider the line bundle $\L'$ of degree $d/m$ on $\C \times E_{m \tau}$ determined by the map $h: \C \to E_{m \tau}$. Again let $g$ denote the $m$-to-one map $g_m: D_\tau\left(m,k\right) \to \C \times E_{m \tau}$ given by $\left(z,t\right) \mapsto \left(z^m, t^m\right)$. Note that $i_{(m,k)}^*\left( g_m^* \L'\right)|_{\C^* \times E_{\tau}} =\mathcal{L}$; thus $g_m^* \L'$ gives the desired extension.
\end{proof}

\begin{cor}\label{extension2} For $d \neq 0$, ${\L\left(m_0,k_0,d,\xi\right)}$ extends to $D\left(m_1, k_1\right)$ if and only if $\frac{k_0}{m_0}=\frac{k_1}{m_1}$. \end{cor}
\begin{proof} The corresponding map $f: \C^* \to \Pic^{d m_0}\left(E\right)$ is given by $z \mapsto \xi z^{-d k_0}$, so ${\L\left(m_0,k_0,d,\xi\right)}$ extends to $D\left(m_1, k_1\right)$ if and only if the map $\xi z^{k_1 d m_0} z^{-d_0 k_0 m_1}$ extends across the origin. The monomial map $z^n$ from $\C^* \to E$ lifts to a map $\C^* \to \C^*$; this map extends to a map $\C \to \C^*$ if and only if $n = 0$. We conclude that $\L\left(m_0,k_0,d,\xi\right)$ extends if and only if
\[ k_1 m_0 = k_0 m_1. \] \end{proof}

\begin{cor}\label{extension3} A line bundle $\L$ on $\C^* \times E_{\tau}$ extends to a compactification $S_{\tau}(\A)$ of $\C^* \times E_{\tau}$ {if and only if $\L=\L_{\tau}\left(m,k,d,\xi\right)$ for some $\left(m,k,d,\xi\right)$ such that $m d=0$ only if $k=0$.} \end{cor}
\begin{proof} When $md \neq0$, $\L=\L_{\tau}\left(m,k,d,\xi\right)$ extends to $\P^1 \times E_{\tau}/\Z/m \Z$, where $\Z/m \Z$ acts by $([z_0: z_1], t) \mapsto (e^{2\pi i/m} [z_0: z_1], e^{-2\pi i k/m} t)$; we recognize  this quotient as $S((m,k),(-m,k))$. When $md=0$,  $\L_{\tau}$ extends to any $S_{\tau}(\A)$ as $\pi_{E_\tau}^*(T_{\xi\inv}^*\L_{\tau, 0} \cdot \L_{\tau, 0})$. 

Now let $\L \in \Pic_{\an}^{d'} \C^* \times E$ be a line bundle which extends to a compactification $S(\A)$. First assume that $d' \neq 0$. Then if $L$ extends to a compactification $S(\A)$, $[E] \in H_2\left(S(\A)\right) \neq 0$ since $c_1\left(\L\right) \cdot [E] = d' \neq 0$. Then $S(\A)$ is algebraic, so $\A=\left(\left(m,k\right),\left(m,-k\right)\right)$. Then since the map $f: \C^* \to E$ has the property that  $z^{k d} f\left(z^m\right)$ extends across the origin, $f$ extends to a map $f: \C^* \to \C^*$, and can be written as a Laurent series
\[ f = \sum_{n \in \Z} f_n z^n. \]
Now apply Corollary \ref{1683}. Since  $z^{k d} f\left(z^m\right)$ extends over the origin, $f_n \neq 0$ for $n = -d k/m$ and $f_n =0$ for $n$ with $k d + m n < 0$; since  $z^{k d} f\left(z^m\right)$ extends over infinity, $f_n=0$ for $n$ with 
$k d + m n > 0$. We conclude that $f = \xi z^{-d/m k}$ for some $\xi \in \C$; then $\L= \L\left(m,k,d,\xi\right)$. 

Now assume that $d=0$, and again consider the corresponding function  $f: \C^* \to E$ and its Laurent series  $f = \sum_{n \in \Z} f_n z^n$. Using Corollary \ref{1683}, we conclude that $f_n=0$ for $n \neq 0$. Thus $f$ is constant, so $\L = \L\left(m_0,k_0,0,f_0\right)$ for any $m_0, k_0$.
\end{proof}

\begin{defn} Let $\mC_\tau$ denote the full subcategory of $\mC \subset \Coh_{\an}{\C^* \times E_\tau}$ generated by the set of line bundles ${\L_\tau\left(m,k,d,\xi\right)}$. Let $\mC_\tau\left(m,k\right)$ denote the full subcategory of $\Coh_{\an}{\C^* \times E_\tau}$ with objects $\overline{\L_\tau\left(m,k,d,\xi\right)}$.
 \end{defn}

\subsubsection{{Morphisms in $\mC_\tau$ and $\mC_\tau \left(m,k\right)$}} \label{labelingsections}
Let $\L, \L' \in  \mC_\tau$. Then 
\[ \Hom_{\mC} \left(\L, \L'\right) = H^0 \left(\C^* \times E_\tau, \L' \otimes \L\inv\right). \] 
So it suffices to compute the global sections $H^0 \left(\C^* \times E,   \L\left(m,k,d,\xi\right)\right).$
Fix a line bundle $\L =  \L\left(m,k,d,\xi\right);$ note that sections exist of $\L$ only if $c_1(\L) . [E] = m d \geq 0$. Thus when $d < 0$,  $H^0 \left(\C^* \times E_{\tau},  \L\right)= 0$. 

Assume then that $d \geq 0$. We express the total space of $\L$ as a quotient:
\[  \mathrm{Tot}\left(\L\left(m,k,d,\xi \right)\right)= \left(\C^* \times \C^*\times \C\right)/\left(\left(z,x,t \right) \sim \left(z,e^{2 \pi i \tau}x, \xi z^{-k d} x^{m d} e^{\pi i \tau md }t \right)\right).\]
 The global sections $\sigma \in  H^0(\C^* \times E_{\tau}, \L)$  can be expressed in terms of convergent power series
\[ \sigma(z,x)= \sum_{j,a \in \Z} C_{j,a} z^{j} x^{a}\]
which are equivariant under this $\Z$ action on $\C^* \times \C^* \times \C$.

When $d=0$, we must have $C_{j,a}=0$ for  $a \tau \not\equiv (\tau \eta + \nu) \, \mod \, \Z$; we conclude that sections exist only when $\eta=N \in \Z$, $\nu \in \Z$. Let $\sigma_{\L,(j,N)}$ denote the monomial $z^j x^N$.
 In this case we can write any global section $\sigma$ as the product of $x^N$ with a globally convergent Laurent series:
\begin{align}\label{mseq1woo} H^0(\C^* \times E, \L)& =\left\{ \left(\sum_{j} c_j z^j \right)x^N \Large\mid \lim_{|j| \to \infty} |c_j|^{1/|j|} =0\right \}\\
                                                        & =   \left\{ \sum_{j} c_j \sigma_{\L,\overline{(j,N)}}\Large\mid \lim_{|j| \to \infty} |c_j|^{1/|j|} =0\right \}. \end{align}

\begin{remark}Any degree zero line bundle with sections is isomorphic to $\O_{E \times \C^*};$ however, it will be convenient to keep the extra data provided by $\xi$ as we write the restriction maps $\Pic S_{\tau}(A) \to D_\tau(m,k) \to \C \times E_\tau$. \end{remark}

When $d \geq 0$, recall that sections of line bundles over the elliptic curve can be expressed in terms of the classical theta functions:
\[ \vartheta[c',c''] \left(q,x\right)= \sum_{\ell \in \Z} q^{\frac{\left(\ell+c'\right)^2}{2}} \left(e^{2 \pi i c''}x\right)^{\left(\ell+c'\right)} \]
The theta function  $\vartheta[0,0] \left(q,x\right)$ is a section of $\L_0$. This observation allows us to find a Fr\'{e}chet space basis for $H^0(\C^* \times E, \L)$. Let $e^{2\pi i (\tau \eta+\nu)}=\xi$, and let
\begin{align}\label{sectionid} \sigma_{\L, \overline{(j,a)}} &=  e^{\frac{ \pi i \tau  \eta^2}{m d}}e^{\frac{2 \pi i \nu \eta}{m d}} z^{(ka/m + j)} \vartheta\left[\frac{a}{m d},0\right]\left(q^{m d}, \xi z^{-k d} x^{m d}\right)\\
&=e^{\frac{ \pi i \tau  \eta^2}{m d}}e^{\frac{2 \pi i  \nu \eta}{m d}} \sum_{\ell \in \Z} q^{m d\left(\ell+\frac{a}{m d}\right)^2/2} \xi^{\ell+\frac{a}{m d}} z^{-k d \ell+j} x^{m d \ell+a}; \end{align}
note that  this power series defines an equivariant function which restricts to a theta function on each fiber. Further note that $\sigma_{\L,\overline{(j,a)}}=\sigma_{\L, \overline{(j+kdn,a-mdn)}}$ for any $n \in \Z$; thus these sections are indexed by the quotient $\Z^2/(kd,-md) \cong \Z \oplus \Z/d$, with the isomorphism realized by $(j,a) \mapsto( m j+ k a, a)$. The  factor $(m j+ k a)$ controls the power of $z$ multiplying the theta function, and any globally convergent equivariant power series $\sigma$ decomposes as 
\[ \sigma = \sum^{md-1}_{a =0} g_a(z)z^{k a/m} \vartheta\left[\frac{a}{m d},0\right]\left(q^{m d}, \xi z^{-k d} x^{m d}\right).\]
Thus we can write again write the space of global sections in terms of everywhere convergent Laurent series:
\begin{equation}\label{mseq2woo} H^0\left(\C^* \times E, {\L\left(m,k,d,\xi\right)}\right) = \left\{ \sum_{\overline{(j,a)}\in \Z^2/(kd,-md)} c_{\overline{(j,a)}}  \sigma_{\L,\overline{(j,a)}}\Large\mid \lim_{ |m j+ k a|  \to \infty} |c_{\overline{(j,a)}}|^{1/ m j+ k a } = 0 \right\}. \end{equation}

\subsubsection{Multiplication of sections}\label{thetaidentities} Let $\L_i =\L\left(m_i,k_i,d_i,\xi_i\right)$, $i=0,1,2$, and let  $e^{2\pi i (\tau \eta_i+\nu_i)}=\xi_i$. We now investigate the Yoneda product
\[ \Ext^0(\L_1,\L_2)\times \Ext^0(\L_0, \L_1)\to \Ext^0(\L_0, \L_2);\]
using the identification $\Ext^0(\L, \L')=H^0(\C^* \times E_{\tau}, \L\inv \L')$, we can identify this with the product  
\[H^0 =H^0(\C^* \times E_{\tau},\L_1\inv \L_2)\otimes H^0(\C^* \times E_{\tau},\L_0\inv \L_1) \to H^0(\C^* \times E_{\tau},\L_0\inv \L_2).\]
We investigate the product  $\sigma_{\L_1\inv \L_2, \overline{(j_1,b)}} \cdot \sigma_{\L_0\inv\L_1,\overline{(j_0,a)}}$. When $m_{i+1} d_{i+1} -m_{i}d_{i} >0$ for $i=0$ or $1$, the line bundle $\L_{i+1}\inv \L_i$ has negative degree, so the space of sections is empty. When one of the bundles is of degree zero, i.e. when $m_i d_i =m_{i+1}d_{i+1}$ for some $i$, the multiplication is simple: 
\begin{equation}\label{35woo} \sigma_{\L_1\inv \L_2, \overline{(j_1,b)}} \cdot\sigma_{\L_0\inv \L_1,\overline{(j_0,a)}} =\sigma_{\L_0\inv\L_2, \overline{(j_0+j_1,a+b)}}.  \end{equation}

Now assume both line bundles are of positive degree. Let $m_1 d_1 - m_2 d_2 =m$ and $k_1 d_1 - k_2 d_2 = k$. When $(m_0 d_0 - m_1 d_1)<0$ and $(m_1 d_1 - m_2 d_2)<0$, we can use  Equation \eqref{sectionid} to write \begin{equation}\label{thetaid3}  \sigma_{\L_0\inv\L_1, \overline{(j_0,a)}}\cdot \sigma_{\L_1\inv \L_2, \overline{(j_1,b)}}= \sum_{n\in \Z}e^{2 \pi i (\tau K_{(j_0+j_1-k n,a+b+m n)}+ \theta_{(j_0+j_1-k n,a+b+mn)})} \sigma_{\L_0\inv \L_2, \overline{(j_0 +j_1 - k n, a+b+m n)}} \end{equation}
where 
\begin{align}\label{46woo} &K_{(j_0+j_1',a+b')} =\frac{1}{2}\left(\frac{(a+\eta_0-\eta_1)^2}{m_1 d_1 - m_0 d_0} +\frac{(b'+\eta_1-\eta_2)^2}{m_2 d_2 - m_1 d_1} - \frac{(a+b'+\eta_0-\eta_2)^2}{m_2 d_2- m_0 d_0}\right);\\
        \begin{split}   \label{47woo} &\theta_{(j_0+j_1',a+b')}\\
         &  =(\nu_0-\nu_1)\left(\frac{a+\eta_0-\eta_1}{m_0 d_0 - m_1 d_1} \right)+ (\nu_1-\nu_2)\left(\frac{b'+ \eta_2-\eta_1}{m_1 d_1 - m_2 d_2} \right)-\left(\nu_2 - \nu_0\right)\left(\frac{a +b' + \eta_0 - \eta_2}{m_2 d_2- m_0 d_0}\right). \end{split}
          \end{align}

 \begin{remark} Note that when $(m_0, k_0)=(m_1, k_2)=(m_2, k_2)$, the image of the homomorphism $\Z \to \Z^2$
 \[ n \mapsto n(-k,m)= (- k_1 (d_1-d_2)  n, m_1 (d_1 - d_2) n) \]
is contained within the torsion part of $\Z^2/( -k (d_0-d_2), m(d_0-d_2))$. Thus the sum above has non-zero coefficients on only finitely many generators; the coefficients on those finitely many generators are theta functions in $z$. This reflects the fact that the sections $\sigma_{\L_0 \inv\L_1  \overline{(j_0,a)}}$ and  $\sigma_{\L_1\inv \L_2, \overline{(j_1, b)}}$ are both \emph{algebraic} sections of  algebraic line bundles on the quasi-projective variety $D_\tau(m,k)$.\end{remark}
 
\subsection{Calculations over $D_{\tau}\left(m,k\right)$} Consider the image of the restriction map 
\[ H^0\left(D_\tau\left(m,k\right), \overline{\L\left(m,k,d,\xi\right)}\right) \to H^0\left(\C^* \times E_\tau,  \L\left(m,k,d,\xi\right)\right);\]
we say a global section extends over $D_\tau\left(m,k\right)$ if it lies in the image of this map.

\begin{prop} $\sigma_{\L, \overline{(j,a)}}$ extends over $D_\tau \left(m,k\right)$ if and only if $m j+ k a \geq 0$. \end{prop}
\begin{proof} Write such a section as a sum of monomials: 
\[ \sigma_{\L, \overline{(j,a)}}= \sum_{\ell \in \Z} C\left(\ell\right) z^{-k d \ell+j} x^{m d \ell+a}: \]
(when $d = 0$, $C\left(\ell\right)=0$ for $\ell \neq 0$; when $d > 0$ the coefficients $C\left(\ell\right)$ are the coefficients of the corresponding $\vartheta$ function). After changing to the $\left(v,w\right)$ coordinates on $\C\times E_\tau$, we write
\[ \sigma_{\L, \overline{(j,a)}}=  v^{m j+ k a}  \sum_{\ell \in \Z} w^{m d \ell + a} : \]
Such section will extend over $v=0$ only if  $m j+ k a > 0$.  Let $g$ denote the generator of $\Z/m \Z$ acting on $\C \times E$; note that  
\[ g \cdot v^{m i+ k a} w^{m d j + a} = e^{2 \pi i k a / m} e^{-2 \pi i k a / m}v^{m i+ k a}  w^{m d j + a} = v^{m i+ k a}  w^{m d j + a} \]
 so any power series of this form is invariant under the action of $\Z/m\Z$, and thus defines a section of $\overline{\mathcal{L}\left(m,k, d, \xi\right)}$. Thus $\sigma_{\L, \overline{(j,a)}}$ is in the image of the restriction map $\iota_(m,k)^*:H^0\left(D_\tau\left(m, k\right), \L\right) \to H^0\left(\C^* \times E_\tau, \L\right)$.  Let $\overline{\sigma}_{\L, \overline{(j,a)}}$ satisfy $\iota_{(m,k)}^*\left(\overline{\sigma}_{\L, \overline{(j,a)}}\right)=\sigma_{\L, \overline{(j,a)}}$; we can then write
\begin{align*} &H^0\left(D_\tau\left(m,k\right),  \L\left(m,k,d,\xi\right)\right) \\
&=  \quad \left\{ \sum_{\overline{(j,a)} \in \Z^2/ \left(-k d, m d\right)} c_{\overline{(j,a)}} \overline{\sigma}_{\overline{(j,a)}}\Large\mid m j+ k a> 0, \lim_{m j + k a \to \infty} |c_{\overline{(j,a)}}|^{|m j +k a|\inv}=0\right\} .\end{align*}
\end{proof}
The theta function identities of the previous subsection govern the multiplications of sections over the $D_\tau \left(m,k\right)$. We conclude
\begin{prop}\label{embedc} \[ i_{(m,k)}^*: \mC_\tau(m,k) \to \mC_{\tau} \] 
is an embedding of categories, i.e., is both faithful and injective on objects. \end{prop}


\subsection{Compact surfaces $S_\tau\left(\mathbf{A}\right)$}\label{csurfaces}
One obtains a description of the Picard groups of the surfaces $S_\tau(\A)$ from the calculation of the fundamental groups of the surfaces in question. Let $\A= \left(\left(m_0,k_0\right), \left(m_{\infty}, k_{\infty}\right)\right)$, and let $n= m_0 k_{\infty} +m_\infty k_{0} \neq 0$; recall that the condition that $n \neq 0$ implies that $S_\tau(A)$ is not algebraic. 

 Recall that 
\[ \pi_1\left(S_{\tau}(\A)\right) = \mathrm{coker} \begin{pmatrix}  m_0 & - m_\infty & 0 \\   k_0 & k_\infty & 0 \\ 0 & 0 & 0\end{pmatrix} : \Z^3 \to \Z^3,  \] 
and $\pi_1\left(S_{\tau }(\A)\right)$ acts freely on the universal cover $\C^2 \setminus \{0\}$ by 
\[ \left(a,b,c\right) \cdot \left(z_1, z_2\right) = \left(\exp\left(\frac{2 \pi i\left(k_{\infty} a + m_{\infty} b+m_\infty c \tau \right)}{n}\right) z_1,  \exp\left(\frac{2 \pi i \left( -k_0 a + m_0 b+ m_0 c \tau\right)}{n}\right) z_2 \right). \]  Let $\Lambda \subset \Z^2$ denote the lattice spanned by the columns of the matrix
 \[\begin{pmatrix} m_0 & - m_\infty \\ k_0 & k_\infty \end{pmatrix}, \] and let $\Lambda^* \subset \Q^2$ denote the dual lattice. Elements of $\Lambda^*$ are those elements $\left(f,g\right) \in \Q^2$ that satisfy
\begin{align}\label{colattice1} m_0 f + k_0 g \in \Z; \\ 
		   \label{colattice2} -m_{\infty} f + k_{\infty}g \in \Z. \end{align}
		     
It is known (see \cite{Ma}) that any line bundle on a Hopf surface  has trivial pullback to the universal cover; thus there is a map $\Hom\left(\pi_1\left(S_{\tau}\left(\mathbf{A}\right)\right), \C^*\right) \to \Pic S_{\tau}(\A)$ described as follows: given an element $\gamma$ of $\Hom\left(\pi_1\left(S_{\tau}\left(\mathbf{A}\right)\right), \C^*\right)$, we  define a  $\pi_1\left(S_{\tau}(\A)\right)$ action on $\left(\C^2 \setminus \{0\}\right) \times \C$ by 
\[ \ell \cdot \left(z_1, z_2, t\right) = \left(\ell\left(z_1,z_2\right), \gamma\inv\left(\ell\right) t\right). \] This defines a holomorphic line bundle $\L_{\gamma}$ on $S_{\tau}(\A)$ with total space
\[ \mathrm{Tot}\, \L_{\gamma} = \left(\left(\C^2 \setminus \{0\}\right) \times \C\right) / \pi_1\left(S_{\tau}(\A)\right). \]
It is also known that this map $\Hom\left(\pi_1\left(S_{\tau}\left(\mathbf{A}\right)\right), \C^*\right) \to \Pic S_{\tau}(\A)$ is an isomorphism.  Note 
\[ \Hom\left(\pi_1\left(S_{\tau}\left(\mathbf{A}\right)\right), \C^*\right)= \Hom(\Z \oplus \Z^2/\Lambda, \C^*) \cong  \C^* \oplus  \left(\Lambda^* / \Z^2\right). \]
The canonical isomorphism is given by $(\xi, f, g) \mapsto \left((a,b,c) \mapsto \xi^c e^{2 \pi i f a} e^{2 \pi i g b}\right).$

Let $\L\left(\xi, f, g\right)$ denote the line bundle corresponding to $\left(\xi, f, g\right) \in \C^*\times  \Lambda^*$. Write $\gamma=\left(\xi, f, g\right)$. The cohomology of the line bundles can be calculated via techniques similar to that in \cite{Ma}, which we summarize below.  The sections of $\L(\gamma)$ are given by holomorphic functions on $\C^2 \setminus \{0\}$ (equivalent, by Hartog's lemma, to holomorphic functions  on $\C^2$) which are equivariant under the action of the fundamental group. The set of such functions is given by the monomials $z_1^{n_1} z_2^{n_2}$, where $n_1$ and $n_2$ are positive integers satisfying
\begin{align}\label{cohoeq1} \exp\left( \frac{2 \pi i \tau (m_{\infty} n_1 + m_0 n_2)} {n}\right) &= \xi;  \\ 
		    \label{cohoeq2}\mathrm{exp}\left(\frac{2 \pi i\left( k_{\infty} n_1 - k_0 n_2\right)}{n}\right) & =    \mathrm{exp}\left({2 \pi i f}\right) ; \\
		       \label{cohoeq3} \mathrm{exp} \left( \frac{2\pi i \left(m_{\infty} n_1 + m_0 n_2\right)}{n}\right) & =    \mathrm{exp}\left({2 \pi i g}\right).
		         \end{align}
Denote the number of such pairs $\left(n_1, n_2\right)$ by $N_{\gamma}.$ From Remark \ref{dualizing}, we know that the dualizing sheaf is  
\[ \mathcal{O}_{S\left(\mathbf{A}\right)}\left(-\left(\frac{1}{m_0}+\frac{1}{m_\infty}\right)[F]\right) = \L\left(\frac{-k_{\infty}+k_0}{n}, \frac{-m_{\infty} - m_0}{n}, e^{2 \pi i \tau \left(\frac{-m_{\infty} - m_0}{n}\right)}\right). \] 
 Write $\gamma_0= \L\left(\frac{-k_{\infty}+k_0}{n}, \frac{-m_{\infty} + m_0}{n}, e^{2 \pi i \tau \left(\frac{-m_{\infty} + m_0}{n}\right)}\right)$; then by Serre duality $h^2\left(\L_\gamma\right)=h^0\left(\L_{\gamma\inv \gamma_0}\right).$ One obtains from the Hirzeburch-Riemann-Roch formula that $\chi\left(\L\right)=0$ for any $\L \in \Pic\left(S_{\tau}(\A)\right)$.  Thus we can deduce 
\begin{equation}\label{cohocalcs} 
h^i\left(\L(\gamma)\right) = \begin{cases} N_{\gamma} & i = 0; \\
										N_{\gamma}+N_{g\inv \gamma_0}  & i=1
									;\\
										N_{\gamma\inv \gamma_0} & i=2. \end{cases} \end{equation}
Note that for all $\L$, $h^0$ and $h^2$ cannot both be non-vanishing.

\subsubsection{Products on global sections}
The product \[ H^0\left(S_{\tau}\left(\mathbf{A}\right), \L(\gamma)\right) \otimes H^0\left(S_{\tau}\left(\mathbf{A}\right), \L\left(g'\right)\right) \to H^0\left(S_{\tau}\left(\mathbf{A}\right), \L\left(g g'\right)\right)\]
is induced by the product on the ring of holomorphic functions $\O_{\C^2}$: 
\[ z_1^{n_1} z_2^{n_2} \cdot z_1^{n_1'} z_2^{n_2'} = z_1^{n_1 + n_1'} z_2^{n_2 + n_2'}. \]

\subsubsection{Restriction maps} The restriction maps $\Pic S_{\tau}(\A) \to \Pic_{\an} {D_{\tau}\left(m_0,k_0\right)} \to \Pic_{\an} ({\C^* \times E_{\tau}})$ induced by the inclusion map $\iota_{m,k}$ can be calculated using the change of coordinates of Equation \eqref{map2}.
\[ \L\left(f,g,e^{2 \pi i (\tau \lambda+\theta)} \right) \mapsto \overline{\L\left(m,k,0, e^{2 \pi i \left(\tau \lambda + \theta- \tau g\right)}\right)} \mapsto \L\left(m,k,0, e^{2 \pi i \left(\tau \lambda + \theta- \tau g\right)}\right). \]
The induced maps on global sections are given by
\begin{equation} \label{cocwoo} z_1^{n_1} z_2^{n_2} \mapsto v^{n_1} w^{\left(m_{\infty} n_1 + m_0 n_2\right)/n} \mapsto z^{\left(k_{\infty} n_1 - k_0 n_2\right)/n} x^{\left(m_{\infty} n_1 + m_0 n_2\right)/n}.\end{equation}

\section{Mirror symmetry for $D_{\tau}(m,k)$ and $\C^* \times E_{\tau}$} \label{mscalcs}
\noindent We now prove mirror symmetry results for the open surfaces $D_{\tau}(m,k)$ and $\C^* \times E_{\tau}$. Let 
\[ M(m,k,d,\alpha) = \begin{pmatrix} -\alpha & d k \\
							  d k & - d m \end{pmatrix}, \]
and  when $m \neq 0$, let 
\[E(m,k, \eta) = \begin{pmatrix}  (k/ m) \eta \\     \eta \end{pmatrix};\]
Let $\eta, \nu \in \R$.  Define the following type-$(m,k)$ Lagrangian section and choice of lift to $\R^4$ by
\[ L_{\left(m,k,d,e^{2 \pi i (\tau \eta + \nu)}\right)} = \left( \{(Y, M(m,k,d, d k^2/m) Y									
												    - E(m,k,\eta)) \Large\mid Y \in \R^2 \}                           ,  \mathrm{d}-   2 \pi i \nu \left( (k/m) \mathrm{d}t+ \mathrm{d} s\right)  \right). \]									    This section is dedicated to the proofs of the following  theorems:

\begin{thmu}
There is a full and faithful functor
\begin{equation*} \Phi_{\W}: \mathcal{C}_{\tau} \to H^0 \F_{\tau, \Wr}\end{equation*}
realized on objects by 
\[\Phi_{(\Wr)}\left( \L_\tau\left(m,k,d,e^{2 \pi i (\tau \eta + \nu)} \right)\right)  = [ L_{\left(m,k,d,e^{2 \pi i (\tau \eta + \nu)}\right)}]. \]
\end{thmu}

 \begin{thmu} 
For each $\left(m, k\right)$ there is a full and faithful functor
\begin{equation*} \Phi_{\left(m,k\right)}: \mathcal{C}_{\tau}\left(m, k\right) \to H^0 \F_\tau\left(m, k\right) \end{equation*} 
realized on objects by 
\[\Phi_{(m,k)}\left( \L_\tau\left(m,k,d,e^{2 \pi i (\tau \eta + \nu)} \right)\right)  =[  L_{\left(m,k,d,e^{2 \pi i (\tau \eta + \nu)}\right)}]. \]
\end{thmu}

\begin{thmu}\label{commute1}
The following diagram commutes:     
 \[\begin{tikzcd}
 \mathcal{C}_{\tau}\left(m, k\right) \arrow{r}{i_{(m,k)}^*}\arrow{d}{\Phi_{(m,k)}}& \mathcal{C}_{\tau} \arrow{d}{\Phi_{(\Wr)}} \\
H^0 \F_\tau\left(m, k\right)\arrow{r}{\iota_{(m,k)}} &H^0 \F_{\tau, \Wr}.
\end{tikzcd}\]
\end{thmu}

\subsection{Calculations in $H^0\F_{\tau, \W}$}
We first calculate in the wrapped category $H^0\F_{\tau, \W}$. 
\subsubsection{Perturbing Hamiltonians}\label{hamform2} For each triple of Lagrangians $(L_0, L_1,L_2)$, choose type-$\W$ Hamiltonians $H_0, H_1, H_2$ satisfying the following conditions:
 \begin{enumerate}[(a)]
 \item The set $(\psi_{H_0}(L_0), \psi_{H_1}(L_1),\psi_{H_2}(L_2))$ is in correct position.
\item $H_i= H_{\alpha_i}(t)+f_i(t,s)$ for some $f(t,s)$ with  $\|f_i\|_{C^2}<1$.
\item $\alpha_i(t)- \alpha_j(t) > K(L_i, L_j)$ for each $i <j$. 
\item For each degree zero generator $y \in CF^0(\psi_{H_i}(L_i), \psi_{H_j} (L_j))$ with $i < j$, $f_i(y)- f_j(y)=0$.
\item When two lifts $\psi_{H_i}\left(\tilde{L}_i\right)$ and  $\psi_{H_j}\left(\tilde{L}_j\right) - \hn$ are not essentially disjoint, so they are not disjoint after perturbing by any compactly supported Hamiltonian, 
$  \left|\psi_{H_i}\left(\tilde{L}_i\right) \cap \psi_{H_j}\left(\tilde{L}_j\right)- \hn\right|=1$ or $2$.
\item $\left|\psi_{H_i}\left(\tilde{L}_i\right) \cap \psi_{H_j}\left(\tilde{L}_j\right)- \hn\right| \subseteq P_i \cap (P_j + \hn)$, where  $P_i$ and $P_j $ are the corresponding planes of the respective Lagrangians, and equality holds when the two planes intersect transversely (i.e., when $[\gamma_{L_0}] \neq [\gamma_{L_1}]$).
\item When $P_i$ and $P_j$ do not intersect transversely, the degree zero generators $y \in \psi_{H_i}\left(\tilde{L}_i\right) \cap\left( \psi_{H_j}\left(\tilde{L}_j\right)- \hn\right)$ have $s(y)\in \N$.
\end{enumerate}
Note that these conditions ensure that $S_{\R^4}(y)=Q_{L_i, L_j}(y)$ for all $y$ since equality holds for all points on the intersection points of the corresponding planes.

By Proposition \ref{htc}, for such choices of Hamiltonian, the natural map 
\[ CF^*(\psi_{H_0}(L_0), \psi_{H_1}(L_1)) \to \Hom_{\F_\Wr}^*(L_0, L_1)\]
is a quasi-isomorphism; thus we can compute the cohomology category $ \Hom_{H^0 \F_\Wr}(L_0, L_1)$ using such choices of Hamiltonians.


\begin{remark}[Notation]  Let $L_0 =  L_{\left(m_0, k_0,d_0,e^{2 \pi i (\tau \eta_0 + \nu_0)}\right)}$, $L_1= L_{\left(m_1,k_1,d_1,e^{2 \pi i (\tau \eta _1+ \nu_1)}\right)}$. \begin{enumerate}[(a)]
\item Let $\mathbf{y}\in CF_{\tau}^*(\psi_{H_0}(L_0), \psi_{H_1}(L_1))$ denote the action-corrected generator $\mathbf{y}=e^{-2 \pi i (A t(y)+\tau S(y))}y$, that is, the image of the action-corrected generator $e^{-2 \pi i  A t(y)} T^{-S(y)}y \in CF^*(\psi_{H_0}(L_0), \psi_{H_1}(L_1))$ under the evaluation map \[\ev_\tau:  CF^*(\psi_{H_0}(L_0), \psi_{H_1}(L_1)) \to CF_{\tau}^*(\psi_{H_0}(L_0), \psi_{H_1}(L_1))\].

\item The stabilizer subgroup $\Gamma(L_0,L_1) \subset \Z^3$ is equal to 
\[\Z \left(\hat{e}_s + k_0 d_0 \hat{e}_{\phi_t} - m_0 d_0 \hat{e}_{\phi_s} \right) \oplus \Z ((k_0 d_0 - k_1 d_1) \hat{e}_{\phi_t} - (m_0 d_0 - m_1 d_1)\hat{e}_{\phi_s}).\] 
Thus the the map $\Z^2 \to \Z^3$ defined by $(j,a) \mapsto j e_{\phi_t} + a e_{\phi_s}$ descends to an isomorphism \[ \Z^2/((k_0 d_0 - k_1 d_1), -(m_0 d_0 - m_1 d_1)) \cong \Z^3/\Gamma(L_0, L_1),\]the indexing group for the topological grading on the Floer complexes.  Let  $\mathbf{y}_{(j,a)}^{H_0, H_1}\in \R^4$ denote a generator of $CF^*\left(\psi_{H_0}(\tilde{L}_0), \left(\psi_{H_1}(\tilde{L}_1)-j e_{\phi_t} -a e_{\phi_s}\right)\right)_{(j,a)}$;  let  \[\mathbf{y}^{H_0,H_1}_{\overline{(j,a)}}=\pi_Y\left(\mathbf{y}_{(j,a)}^{H_0,H_1}\right)\] denote the corresponding generator of the Floer complex. \end{enumerate} \end{remark}

\begin{remark} The conditions on the Hamiltonians ensure that for any \[y= (Y,\Phi) \in \psi_{H_0}(\tilde{L}_0) \cap \left(\psi_{H_1}(\tilde{L}_1)-j e_{\phi_t} -a e_{\phi_s}\right),\] $Y$ satisfies the equation
\begin{equation}\label{27} \begin{pmatrix}-\frac{d_0( k_0)^2}{m_0} +\frac{d_1( k_1)^2}{m_1} + \alpha_0-\alpha_1    & k_0 d_0 - k_1 d_1\\
					k_0 d_0 - k_1 d_1 & -m_0 d_0 + m_1 d_1\end{pmatrix} Y- \begin{pmatrix}j - \frac{k_0}{m_0} (\eta_0)+ \frac{k_1}{m_1} (\eta_1) \\ a -(\eta_0) + \eta_1 \end{pmatrix} =0, \end{equation}
					as this is the equation which describes the intersection of the corresponding planes. \end{remark}

\subsubsection{Morphisms} Let $L_0 =  L_{\left(m_0, k_0,d_0,e^{2 \pi i (\tau \eta_0 + \nu_0)}\right)}$, $L_1= L_{\left(m_1,k_1,d_1,e^{2 \pi i (\tau \eta _1+ \nu_1)}\right)}$. We now calculate
\[ \Hom_{H^0 \F_{\tau, \Wr}}(L_0, L_1) \cong \whHF_{\tau}^0(\psi_{H_0}(L_0), \psi_{H_1}(L_1)). \]
\subsubsection*{Case 1: $L_0$ and $L_1$ represent the same homology class.} The conditions on $H_0$ and $H_1$  ensure that 
\[CF_\tau^i(\psi_{H_0}(L_0), \psi_{H_1}(L_1))_{(j,a)} = \begin{cases} \C \cdot \mathbf{y}_{(j,a)}^{H_0, H_1} & i = 0 \\[5pt]
											  \C \cdot \mathbf{z}_{(j,a)}^{H_0, H_1} & i = 1 \end{cases} \]
when $\psi_{H_0}(\tilde{L}_0)$ and   $\psi_{H_1}(\tilde{L}_1) + j \e_{\phi_t} + a\e_{\phi_s}$
	are not essentially disjoint, and 	
\[CF^*(\psi_{H_0}(L_0), \psi_{H_1}({L_1}))_{(j,a)} = 0 \]					 
otherwise. In the former case, we appeal to Proposition \ref{kingprop} to deduce that   \[HF^0(\psi_{H_0}(L_0), \psi_{H_1}({L_1}))_{\overline{(j,a)}} = \begin{cases}
 																	 \Lambda \mathbf{y}_{\overline{(j,a)}}^{H_0, H_1} &  \nu_0 - \nu_1 \in \Z \\
																	 0 &  \nu_0 - \nu_1 \notin \Z \end{cases}.\]

It now remains to calculate the $(j,a) \in \Z^2$ such that $\psi_{H_0}(\tilde{L}_0)$ and $\psi_{H_1}\tilde{L}_1) -j \hat{e}_{\phi_t} - a \hat{e}_{\phi_s}$ intersect. Note that if $\eta_0 - \eta_1 \notin \Z$,
 \[ \psi_{H_0}(L_0) \cap \left( \psi_{H_1}(L_1)-i \hat{e}_{\phi_t} - a \hat{e}_{\phi_s}\right) = \emptyset\] 
 for all $(j,a) \in \Z$. When $\eta_0 - \eta_1 = N \in \Z$,
  \[|\psi_{H_0}(L_0) \cap \left( \psi_{H_1}(L_1)-i \hat{e}_{\phi_t} - a \hat{e}_{\phi_s}\right) | = 2\] 
  if and only if $a = N$. 

 Let $t_j= t\left(y_{{(j,N)}}^{H_0, H_1}\right)$; from Equation \eqref{27}, we read off that $ t_j = j(\alpha_0-\alpha_1)\inv.$ We conclude that 
  
  \begin{align}\label{28} \Hom_{H^0 \F_{\tau, \Wr}}(L_0, L_1) & \cong \whHF_\tau^0(\psi_{H_0}(L_0), \psi_{H_1}(L_1));\\
  											\label{29} &= \left\{ \sum_{j \in \Z} c_j \mathbf{y}_{\overline{(j,N)}} \mid \lim_{|t_j| \to \infty} |c_j|^{(t_j)\inv} =0 \right\} \\
											\label{30} &=\left\{ \sum_{j \in \Z} c_j \mathbf{y}_{\overline{(j,N)}} \mid \lim_{|j| \to \infty} |c_j|^{j \inv} =0 \right\}.
\end{align}

\subsubsection*{Case 2: $L_0$ and $L_1$ represent different homology classes }Now assume that $[\gamma_{L_0}]\neq [\gamma_{L_1}]$. Let $\alpha (H_i)= \alpha_i +d_i (k_i)^2/m_i$. Let $m= m_1 d_1 - m_0 d_0$, $k=k_1 d_1 - k_0 d_0$, $d=1$.

 Then 
the quadratic form $Q_{\psi_{H_0}(L_0), \psi_{H_1}(L_1)}$ of Definition \ref{qdecomp} is given by \[Q_{\psi_{H_0}(L_0), \psi_{H_1}(L_1)}= x^T M(m,k,d, \alpha_0-\alpha_1) x;\]
its decomposition is given by
\[ Q_{(L_0, L_1)}^\alpha +  Q_{(L_0, L_1)}^\beta=\alpha t^2 + Q_{(L_0, L_1)}^\beta . \] In this case the Floer complex $CF^0 \left(\psi_{H_0}(L_0), \psi_{H_1}(L_1)\right)$ is concentrated in a single degree $i$. $i$ is zero if and only the signature of $Q_{\psi_{H_0}(L_0), \psi_{H_1}(L_1)}$ is $(2,0)$, if and only if $m_0 d_0 < m_1 d_1$. In this case $CF^0 (\psi_{H_0}(L_0), \psi_{H_1}(L_1))_{(j,a)}$ is nonzero for all $(j,a)$ since two transverse planes will always intersect, and generated by a single intersection point. We write
 \[ CF_{\tau}^0 \left(\psi_{H_0}(L_0), \psi_{H_1}(L_1)\right)_{(j,a)} = \C \mathbf{y}_{(j,a)}^{H_0, H_1}. \]
 Write $t_{j,a}= t\left (y_{(j,a)}^{H_0, H_1}\right)$; from Equation \eqref{27}, we read off that  
 \begin{equation}\label{tineq} t\left (y_{(j,a)}^{H_0, H_1}\right)=c_1(m j +m k) + c_2. \end{equation}
 for two constants $c_1$ and $c_2$, with $c_1 > 0$. 
We conclude that
\begin{align} \Hom_{H^0 \F_{\tau, \Wr}}(L_0, L_1) &= \whHF_\tau^0 (\psi_{H_0}(L_0), \psi_{H_1}(L_1))_{(j,a)}  ;\\
  &= \left\{ \sum_{\overline{(j,a)} \in \Z^2/(-k,m)} c_{\overline{(j,a)}} \mathbf{y}_{\overline{(j,a)}}^{H_0, H_1} 
\mid \lim_{\left|t_{j,a}\right| \to \infty} \left| c_{\overline{(j,a)}}\right|^{(t_{j,a})\inv}\right\} =0\\
\label{31} &= \left\{ \sum_{\overline{(j,a)} \in \Z^2/(-k,m)} c_{\overline{(j,a)}} \mathbf{y}_{\overline{(j,a)}}^{H_0, H_1} 
\mid \lim_{|m j+ k a| \to \infty}  \left| c_{\overline{(j,a)}} \right|^{|m j+ ka|\inv}=0\right\}.
\end{align}

\subsubsection{Independence of choices of the Hamiltonian.}
\begin{prop}\label{actionvan} Let $L_0$ and $L_1$ be as above. For any choices of $H_0, H_1, H_0', H_1'$ where $H_0$ and $H_1$ satisfy the conditions above, 
\[\left[\mathbf{y}_{\overline{(j,a)}}^{H_0, H_1}\right]=\left[\mathbf{y}_{\overline{(j,a)}}^{H_0', H_1'}\right] \]
as classes in $\Hom_{H^0 \F_{\tau}}(L_0,L_1)$.
\end{prop}
\begin{proof} For any set of Hamiltonians $H_0, H_1, H_0', H_1'$,  there exists a set of Hamiltonians $H_0'', H_1''$ such that the sets $(\psi_{H_0''}(L_0), \psi_{H_1''}(L_1), \psi_{H_0}, \psi_{H_1})$ and $(\psi_{H_0''}(L_0), \psi_{H_1''}(L_1), \psi_{H_0'}(L_0), \psi_{H_1}(L_1))$ are both in correct position; thus it suffices to show that in the case that $(\psi_{H_0'}(L_0), \psi_{H_1'}(L_1), \psi_{H_0}(L_0), \psi_{H_1'}(L_1))$ is  in correct position
\begin{equation} \label{87} \mu^2\left(c_{H_1 \to H_1'}, y_{(j,a)}^{H_0', H_1'}\right) = \mu^2\left(y_{(j,a)}^{H_0, H_1}, c_{H_0 \to H_0'} \right) \in CF^*(\psi_{H_0'}(L_0), \psi_{H_1}(L_1)). \end{equation}
Write
 \begin{equation}\label{nablad} \nabla_0-\nabla_1= 2 \pi i \left(\left(\nu_0 -\nu_1\right) \mathrm{d}w +A\,\mathrm{d}t\right)\end{equation}
 as in Equation \eqref{nablade}. We have seen that $\mu^2\left(y_{(j,a)}^{H_0,H_1}, c_{H_0 \to H_1'}\right)$ counts one triangle with area 
\[ S_{\R^4}\left(y_{(j,a)}^{H_0, H_1}\right)- S_{\R^4}\left(y_{(j,a)}^{H_0', H_1}\right); \]
the monodromy around this triangle is given by 
\[ \exp\left(2 \pi i \left (B\left( w\left(y_{(j,a)}^{H_0,H_1}\right)-  w\left(y_{(j,a)}^{H_0',H_1}\right)\right)+A\left( t \left(y_{(j,a)}^{H_0,H_1}\right)-  t \left(y_{(j,a)}^{H_0',H_1}\right)\right) \right)\right).\]
Recall that $w\left(y_{(j,a)}^{H_0, H_1}\right)=w\left(y_{(j,a)}^{H_0', H_1}\right)$; thus this quantity is equal to 
\[ \exp\left(2 \pi i A \left( t\left(y_{(j,a)}^{H_0,H_1}\right)- t \left(y_{(j,a)}^{H_0',H_1} \right)\right)\right).\]
Since \[\mathbf{y}_{(j,a)}^{H_0,H_1}=e^{-2 \pi i \left(\tau S\left(y_{(j,a)}^{H_0,H_1}\right)+A t\left(y_{(j,a)}^{H_0,H_1}\right)\right)} y_{(j,a)}^{H_0,H_1},\] and similarly for $\mathbf{y}_{(j,a)}^{H_0',H_1}$, we write
\[ \mu^2\left(\mathbf{y}_{(j,a)}^{H_0,H_1}, c_{H_0 \to H_1'}\right) = e^{ 2\pi i \tau \left(S_{\R^4}\left(y_{(j,a)}^{H_0, H_1}\right)- S\left(y_{(j,a)}^{H_0, H_1}\right)\right)-\left(S_{\R^4}\left(y_{(j,a)}^{H_0', H_1}\right)- S\left(y_{(j,a)}^{H_0', H_1}\right)\right)} \mathbf{y}_{(j,a)}^{H_0',H_1}.\]
When $[\gamma_{L_0}]=[\gamma_{L_1}]$, 
\[ S_{\R^4}\left(y_{(j,a)}^{H_0, H_1}\right)- S\left(y_{(j,a)}^{H_0, H_1}\right)=S_{\R^4}\left(y_{(j,a)}^{H_0', H_1}\right)- S\left(y_{(j,a)}^{H_0', H_1}\right)=0;\] 
when $[\gamma_{L_0}]\neq [\gamma_{L_1}]$ we can apply Lemma \ref{boundondiff} to conclude 
\[ \left(S_{\R^4}\left(y_{(j,a)}^{H_0, H_1}\right)- S\left(y_{(j,a)}^{H_0, H_1}\right)\right)-\left(S_{\R^4}\left(y_{(j,a)}^{H_0', H_1}\right)- S\left(y_{(j,a)}^{H_0', H_1}\right)\right)=0. \] 
Thus 
\[\mu^2\left(\mathbf{y}_{(j,a)}^{H_0,H_1}, c_{H_0 \to H_1'}\right)= \mathbf{y}_{(j,a)}^{H_0',H_1}.\]
We can repeat this process on the left to show that $\mu^2\left(c_{H_1 \to H_1'}, y_{(j,a)}^{H_0', H_1'}\right) = \mathbf{y}_{(j,a)}^{H_0',H_1};$ we conclude that the equality of Equation \eqref{87} holds.

\end{proof}

\begin{defn}Let $\left[\mathbf{y}_{\overline{(j,a)}}^{H_0, H_1}\right]=[y_{\overline{(j,a)}}]\in \Hom_{H^0\F_{\tau, \Wr}}(L_0,L_1).$ We will also denote this generator by $[y_{L_0,L_1, \overline{(j,a)}}]$ when it is useful to record the extra data.  \end{defn}

\subsubsection{$\mu^2$ calculations}
Let  \begin{align*} L_i =& L_{\left(m_i,k_i,d_i,e^{2 \pi i (\tau \eta_i + \nu_i)}\right)}, i=0,1,2;\\
					&\left[x_{\overline{(j_0,a)}}\right]\in \Hom_{H^0 \F_{\tau,\W}}(L_0,L_1); \\
					&\left[z_{\overline{(j_1,b)}}\right]\in \Hom_{H^0\F_{\tau,\W}}(L_1, L_2);\\
				      &\left[y_{\overline{(j_2,c)}}\right]\in \Hom_{H^0\F_{\tau,\W}}(L_0, L_2).\end{align*}		
Let $w_0, w_1,$ and $w_2$ be the coordinates associated to the pairs $(L_0,L_1)$, $(L_1,L_2)$, and $(L_0, L_2)$ (so $w_0 = s - (k_0 d_0 - k_1 d_1)/(m_0 d_0 - m_1 d_1) t$ when $m_0 d_0 - m_1 d_1 \neq 0$, and so on).
Write								 
\[\left[x_{\overline{(j_0,a)}}\right]\cdot\left[z_{\overline{(j_1,b)}}\right]=\sum_{\overline{(j_2,c)}} A_{\overline{(j_2,c)}}\left[y_{\overline{(j_2,c)}}\right]. \]
We now calculate the coefficients in the above equation. Choose representative Hamiltonians $H_0$, $H_1$, $H_2$ and generators $\mathbf{x}_{\overline{(j_0,a)}}, \mathbf{z}_{\overline{(j_1,b)}}$.
First assume that $m_{i} d_{i} - m_{i+1} d_{i+1}=0$ for $i=0,1$; then $\Hom_{H^0 \F_{\tau}}(L_i, L_{i+1})\neq 0$ only if $k_{i} d_{i} - k_{i+1} d_{i+1}=0$ as well. In this case we have seen in the proof of Proposition \ref{conv} that the product $\mu^2\left(\mathbf{z}_{\overline{(j_1,b)}}, \mathbf{x}_{\overline{(j_0,a)}}\right)$ counts only a single $J$-holomorphic map which has vanishing action-corrected area; thus
\begin{equation}\label{mseq3woo} \mu^2\left(\mathbf{z}_{\overline{(j_1,b)}}, \mathbf{x}_{\overline{(j_0,a)}}\right)=\mathbf{y}_{\overline{(j_0 +j_1, {a+b})}}.\end{equation}

Otherwise, fix a lift $x_{(j_0, a_1)} = \tilde{L}_0 \cap \left( \tilde{L}_1-(j_0 \hat{e}_{\phi_t}, a \hat{e}_{\phi_s}) \right)$. Let $k= k_1 d_1 - k_2 d_2$, $m=m_1 d_1 - m_2 d_2$. Using Proposition \ref{onlyone}, write the product as the sum of products corresponding to different choices of the lift of $z_{\overline{(j,a)}}$: let $(-k,m)= (k_1 d_1 - k_2 d_2, m_1 d_1 - m_2 d_2)$, and write
\begin{equation}\label{multeq} \mu^2\left(\mathbf{z}_{\overline{(j_1,b)}}, \mathbf{x}_{\overline{(j_0,a)}}\right) = \sum_{\ell \in \Z }e^{2 \pi i( \tau C_{(j_0+j_1 - k \ell ,a+b+m\ell)}+ \theta_{(j_0+j_1-k\ell,a+b+m\ell)})} \mathbf{y}_{\overline{(j_0+j_1-k \ell ,a+b+m\ell)}} \end{equation}
The coefficients $C_{(j_0+j_1 - k \ell ,a+b+m\ell)}$ and $\theta_{(j_0+j_1-k\ell,a+b+m\ell)}$ record the action-corrected area and monodromy of the unique  J-holomorphic map with the two boundary marked points mapping to the intersection points 
\begin{align*} x_{(j_0,a)}&= \tilde{L}_0 \cap  \left( \tilde{L}_1-(j_0 \hat{e}_{\phi_t}, a \hat{e}_{\phi_s}) \right); \\
		     z_{(j_1',b')}&= \left( \tilde{L}_1-(j_0 \hat{e}_{\phi_t}, a \hat{e}_{\phi_s}) \right)  \cap  \left( \tilde{L}_2-(j_0+j_1') \hat{e}_{\phi_t}- (a+b') \hat{e}_{\phi_s}\right);\\
		    y_{(j_0+j_1',a+b')} &= \tilde{L}_0 \cap \left( \tilde{L}_2-(j_0+j_1') \hat{e}_{\phi_t}- (a+b') \hat{e}_{\phi_s}\right). \end{align*}		   
We conclude that in $H^0 \F_{\tau, \W}$, 
\[ [z_{(j_1,b)}]\cdot [x_{(j_0,a)}]=\sum_{\ell \in \Z} e^{2 \pi i( \tau C_{(j_0+j_1 - k \ell ,a+b+m\ell)}+ \theta_{(j_0+j_1-k\ell,a+b+m\ell)})}[ y_{\overline{(j_0+j_1-k \ell ,a+b+m\ell)}}]. \]
		    
We now find the coefficients on the terms above. Write \begin{align}\label{actionterm} &C_{(j_0,a),(j_1',b')} \\
&=S_{\R^4}\left(x_{(j_0,a)}\right)-S\left(x_{(j_0,a)}\right) + S_{\R^4}\left(z_{{(j_1',b')}}\right) -S\left(z_{(j_1',b')}\right) -S_{\R^4}\left(  y_{(j_0+j_1',a+b')} \right)+S\left(  y_{(j_0+j_1',a+b')}\right)  \\
		&=Q_{L_0, L_1}^{\beta}\left(x_{(j_0,a)}\right)+Q_{L_1, L_2}^{\beta}\left(z_{(j_1,' b')}\right)-Q_{L_0,L_2}^{\beta} \left(y_{(j_0+j_1',a+b')}\right)\\
	          \begin{split} &= \frac{1}{2}\left((m_0 d_0 - m_1 d_1) \left(w_{0}(x_{(j_0,a)})\right)^2
	          +(m_2 d_ 2- m_1 d_1) \left(w_{1}(z_{(j_1',b')})\right)^2\right.\\
	         & \left. \quad -(m_0 d_0-m_2 d_2) \left(w_2\left(y_{(j_0+j_1',a+b')}\right)\right)^2\right).\end{split}
		 \end{align}
 Write $x_{(j_0,a)} =(X, \Phi)$. Let $\alpha =\alpha_0-\alpha_1 +\frac{(k_0 d_0 - k_1 d_1)^2}{m_0 d_0 - m_1 d_1}$.
Note that $X=(x_t, x_s)$ is the solution to the equation 
\begin{align*} 0&=\left(\begin{pmatrix}- \frac{(k_0 d_0 - k_1 d_1)^2}{m_0 d_0 - m_1 d_1} & k_0 d_0 - k_1 d_1\\
					k_0 d_0 - k_1 d_1 & -m_0 d_0 + m_1 d_1\end{pmatrix} + \begin{pmatrix} \alpha & 0\\ 0&0 \end{pmatrix} \right)X- \begin{pmatrix}j_1 - k_0/m_0 (\eta_0)+ k_1/m_1 (\eta_1) \\ a -(\eta_0) + \eta_1 \end{pmatrix} \\
					&=
					 \begin{pmatrix} (k_0 d_0 - k_1 d_1)w_0\left(x_{(j_0,a)}\right) + \alpha t\left(x_{(j_0,a)}\right)\\  -(m_0 d_0 + m_1 d_1)w_0\left(x_{(j_0,a)}\right)\\\end{pmatrix} +\begin{pmatrix}j_1 - k_0/m_0 (\eta_0)+ k_1/m_1 (\eta_1) \\ a - \eta_0 + \eta_1 \end{pmatrix}.
					\end{align*}
From this, we can read off
\[w_{0}\left(x_{(j_0,a)}\right)=-\frac{a - \eta_0 + \eta_1}{m_0 d_0-m_1 d_1}. \]
We can similarly show that 
\begin{align*} w_{1}\left(z_{(j_1', b')}\right)=-\frac{b' - \eta_1 + \eta_2}{m_1 d_1-m_2 d_2};\\
		     w_2\left(y_{(j_0,a)+\overline{(j_1',b')}}\right)=-\frac{a+b' - \eta_0 + \eta_2}{m_2 d_2-m_0 d_0}; \end{align*}
Thus we can write the term in Equation \eqref{actionterm} as
\begin{equation}\label{mseq1} C_{(j_0,a),(j_1',b')} =\frac{1}{2}\left(\frac{(a+\eta_0-\eta_1)^2}{m_1 d_1 - m_0 d_0} +\frac{(b'+\eta_1-\eta_2)^2}{m_2 d_2 - m_1 d_1} - \frac{(a+b'+\eta_0-\eta_2)^2}{m_2 d_2- m_0 d_0}\right) \end{equation}
Using the decomposition of $\nabla$ of Equation \eqref{nablad}, write
\begin{align} \theta &=( \nu_0 - \nu_1)w_{0}\left(x_{(j_0, a)}\right) + ( \nu_1 - \nu_2)w_{1}\left(z_{(j_1', b')}\right)+(\nu_0 - \nu_2)w_2\left(y_{(j_0+j_1',a+b')}\right)\\
 \label{mseq2}&=(\nu_0-\nu_1)\left(\frac{a+\eta_0-\eta_1}{m_0 d_0 - m_1 d_1} \right)+ (\nu_1-\nu_2)\left(\frac{b'+ \eta_2-\eta_1}{m_1 d_1 - m_2 d_2} \right)+\left(\nu_2 - \nu_0\right)\left(\frac{a +b' + \eta_0 - \eta_2}{m_2 d_2- m_0 d_0}\right).\end{align}

\subsection{Calculations in  $H^0\F_{\tau}(m,k)$}We now calculate in the partially wrapped category  $H^0 \F_\tau (m,k)$. Consider the collection $L_{(m,k, d, \eta)}$ for fixed $k$, $m$. For each pair $L_0$ and $L_1$ of such Lagrangians, choose type-$(m,k)$ Hamiltonians $H_0$ and $H_1$with the following properties: 
\begin{enumerate}[(a)] \item The set $(\psi_{H_0}(L_0), \psi_{H_1}(L_1))$ is in correct position. 
\item There are type-$\W$ Hamiltonians $H_0', H_1'$ satisfying  constraints (a)-(f) of  Section \ref{hamform2} so that  $H_0=H_0'$ on $Y_T^+$, where $T$ is such that 
\begin{align*} \pi_{Y}\left(\left\{\tilde{L}_0\cap\left( \tilde{L}_{1} - (j \e_{\phi_s},a\e_{\phi_t}) \right) \mid m j+ k a \geq 0 \right\} \right) \subseteq Y_T^+ \\
		     \pi_{Y}\left( \left\{\tilde{L}_0 \cap\left( \tilde{L}_{1} - (j \e_{\phi_s},a\e_{\phi_t}) \right) \mid m j+ k a < 0 \right\} \right) \subseteq Y_T^-; \end{align*}
we see that such a $T$ is guaranteed to exist by Equation \eqref{27}. We also demand that $H_0(T-t)=H_0(T+t)$.
\item $\psi_{H_0}(L_0)$ and  $\psi_{H_1}(L_1)$ are disjoint on $Y_T^-$ for all $i<j$. 
 \end{enumerate} 
There is a natural inclusion of complexes \[\iota_{i,j}: CF_\tau^*(\psi_{H_0}(L_0), \psi_{H_1}(L_1)) \to CF^*_\tau(\psi_{H_0'}(L_0), \psi_{H_1'}(L_1))\] for $i < j$ induced by the inclusion of generators which respects the actions $S_{\R^4}$ and $S$, thus a natural inclusion \[\iota_{i,j}: \whHF^0_{\tau}(\psi_{H_0}(L_0), \psi_{H_1}(L_1)) \to \whHF^0_\tau(\psi_{H_0'}(L_1), \psi_{H_0'}(L_1)).\] Let $\overline{y}_{\overline{(j,a)}}$ denote the generator of  $CF_\tau^*(\psi_{H_0}(L_0), \psi_{H_1}(L_1)) $ which maps to $y_{\overline{(j,a)}}$; then using Proposition \ref{htc} we can write 
\begin{align*} \Hom_{H^0 \F_{\tau}(m,k)} (L_0, L_1) &\cong \whHF_{\tau}^0(\psi_{H_0}(L_0), \psi_{H_1}(L_1)) \\
									   & \cong\left\{ \sum_{\overline{(j,a)}, m j + k a \geq 0} c_{\overline{(j,a)}} \overline{y}_{\overline{(j,a)}}  \mid \lim_{m j + ka \to \infty} c_{\overline{(j,a)}}^{(m j + ka)\inv}=0\right\}. \end{align*}

 \begin{prop} \label{inclusion}The inclusion of complexes \[\iota_{(0,1)}: CF_\tau^*(\psi_{H_0}(L_0), \psi_{H_1}(L_1)) \to CF^*_\tau(\psi_{H_0'}(L_0), \psi_{H_1'}(L_1'))\] agrees with the map 
 \[ \Hom_{H^0 \F_\tau(m,k)}(L_0, L_1) \to \Hom_{H^0 \F_{\tau, \Wr}} (L_0, L_1).\] 
 \end{prop}\begin{proof} It suffices to show that in the case $H_1 = H_1'=0$, there exists type-$\W$ Hamiltonians $G$ and $G'$ such that $G + H_0 = G' + H_0'$, such that the sets \[(\psi_{G+ H_0} (L_0), \psi_{H_0}(L_0), L_1)\] and \[(\psi_{G'+ H_0'} (L_0), \psi_{H_0'}(L_0), L_1)\] are in correct position, and such that the following diagram commutes: 
  \[\begin{tikzcd}
&CF_\tau^*(\psi_{H_0}(L_0), L_1) \arrow{r}{\iota_0,1}\arrow{d}{\mu^2(c_G, \cdot)} & CF_\tau^*(\psi_{H_0'}(L_0), L_1)\arrow{d}{\mu^2(c_{G'}, \cdot)} \\
 & CF_\tau^*(\psi_{G+H_0}(L_0), L_1) \arrow{r}{=} &CF_\tau^*(\psi_{G'+H_0'}(L_0), L_1). 
\end{tikzcd}\]
 Let $T > 0$ be such that $H_0=H_0'$ on $Y_T^+$. Then defining $G$ and $G'$ by $G(t)=H_0(T+(T-t))+f$ and $G'(t)= H_0(T-|T-t|)+f$, where $f$ is a small function supported on a neighborhood of $T$ such that $H_0(T-|T-t|)+f $ function, provides such $G$ and $G'$. (Here we use the fact that $H_0'$ satisfies $H_0'(T-t)=H_0'(T+t)$.) Note that for $t > T$, $H_0(t) = H_0'(t)$ and $G(t)=G'(t)$. Let $y_{(j,a)}$ denote the unique generator of $CF^0(\psi_{H_0}(L_0), L_1)_{(j,a)}$; let $y^G_{(j,a)}$ denote the unique generator of $CF^0(\psi_{G+H_0}(L_0), L_1)_{(j,a)}$. Then \[\mu^2\left(\mathbf{y}_{(j,a)}, c_{G}\right)=\mu^2\left(\iota_{0,1}\left(y_{(j,a)}\right),  c_{G'}\right) =\mathbf{y}^G_{(j,a)};\] in both cases the areas of the discs contributing to the count defining $\mu^2$ are cancelled by the change in action of the generators. 
 \end{proof}

We conclude the following:  
\begin{prop}\label{embedf}
 \[  \iota_{(m,k)}: H^0 \F_\tau (m,k) \to H^0 \F_{\tau, \Wr}\]
 is an embedding of categories, i.e., is both faithful and injective on objects. 
\end{prop}

 \subsection{Mirror symmetry statements} 
We now prove the theorems in the beginning of the section, which are restated below with proofs; the proofs amount to matching up the structure coefficients.

\begin{thm}\label{ms1}
There is a full and faithful functor
\begin{equation*} \Phi_{\W}: \mathcal{C}_{\tau} \to H^0 \F_{\tau, \Wr}\end{equation*}
realized on objects by 
\[\Phi_{\Wr}\left( \L_\tau\left(m,k,d,e^{2 \pi i (\tau \eta + \nu)} \right)\right)  =  [L_{\left(m,k,d,e^{2 \pi i (\tau \eta + \nu)}\right)}]. \]
\end{thm}\begin{proof}
Let $\L_i = L_{\left(m_i,k_i,d_i,e^{2 \pi i (\tau \eta_i + \nu_i)}\right)}$, $i=0,1$ and  $L_i= L_{\left(m,k,d,e^{2 \pi i (\tau \eta_i + \nu_i)}\right)}$. Having defined $\Phi_{\Wr}(\L_i)=[L_i]$, define for $\sigma_{\L_0 \inv \L_1,\overline{(j,a)}} \in \Ext^0(\L_0, \L_1)$,
\[ \Phi_{\Wr}(\sigma_{\L_0 \inv \L_1, \overline{(j,a)}}) = [y_{L_0,L_1,\overline{(j,a)}}].  \] 

Identifying Equations \eqref{mseq1woo} and Equations \eqref{30}, and \eqref{mseq2woo} and \eqref{31}, shows that this induces an isomorphism 
\[ \Hom_{H^0 \F_\tau}(\Phi_\Wr(L_0),\Phi_{\Wr}(L_1)) =  \Hom_{H^0 \F_\tau}(\Phi_\Wr(L_0),\Phi_{\Wr}(L_1)). \]
for all pairs of Lagrangians. By identifying Equations \eqref{35woo} and \eqref{mseq3woo}, \eqref{46woo} and \eqref{mseq1}, and \eqref{47woo} and \eqref{mseq2}, we see that this  isomorphism respects the composition of morphisms. \end{proof}

By restricting the domain of the functor $\Phi_{W}$ on objects, and using the fact that the localization functors  \[  i_{(m,k)}: H^0 \F_\tau (m,k) \to H^0 \F_{\tau, \Wr};\] 
 \[ \iota_{(m,k)}^*:\mC_{\tau}(m,k) \to \mC_{\tau} \]
 are both embeddings of categories, we obtain the following theorems:
\begin{thm}
There is a full and faithful functor
\begin{equation*} \Phi_{(m,k)}: \mathcal{C}_{\tau}(m,k) \to H^0 \F_{\tau}(m,k)\end{equation*}
realized on objects by 
\[\Phi_{(m,k)}\left( \L_\tau\left(m,k,d,e^{2 \pi i (\tau \eta + \nu)} \right)\right)  = [ L_{\left(m,k,d,e^{2 \pi i (\tau \eta + \nu)}\right)}]. \]
\end{thm}
The functor above takes $\overline{\sigma}_{\L_0\inv \L_1, \overline{(j,a)}}$ to $[\overline{y}_{L_0,L_1,\overline{(j,a)}}]$.

\begin{thm}
The following diagram commutes:     
 \[\begin{tikzcd}
 \mathcal{C}_{\tau}\left(m, k\right) \arrow{r}{i_{(m,k)}^*}\arrow{d}{\Phi_{(m,k)}}& \mathcal{C}_{\tau} \arrow{d}{\Phi_{(\Wr)}} \\
H^0 \F_\tau\left(m, k\right)\arrow{r}{\iota_{(m,k)}} &H^0 \F_{\tau, \Wr}.
\end{tikzcd}\]
\end{thm}

\begin{eg}\label{example1} Consider the line bundles $\L=\L(1,0,1,1)$ and  $\L'=\L(1,-1,1,1)$ on $\C^* \times E_{\tau}$; in less arcane notation, these are, respectively, the line bundles $\pi_{E_{\tau}}^*(\L_{\tau, 0})$ and $f^*\left(\pi_{E_{\tau}}^*(\L_{\tau,0})\right)$, where $f: \C^* \times E_{\tau} \to  \C^* \times E_{\tau}$ is given by $f(z,x) = (z, z x)$. (Note that this is the change of coordinate map by which  $\C^* \times E_{\tau}$ is glued to itself to produce the standard Hopf surface.) There does not exist an algebraic chart on $\C^* \times E_{\tau}$ for which $\L$ and $\L'$ are both algebraic line bundles.) Let $q=e^{2 \pi i \tau}$, and recall that the standard theta function $\vartheta[0,0](q,x)$ is a section of $\L_{\tau, 0}$. Thus
\[ H^0(\C \times E_{\tau},\L)= \left\{ p(z) \cdot \vartheta[0,0](q, x) \mid \,\, p:\C^* \to \C \text{ holomorphic}\right\}.\]
The Fr\'{e}chet basis for $H^0(\C \times E_{\tau}, \L)$ of Equation \eqref{sectionid} is given by  $\left\{\sigma_{\overline{(j,a)}}\right\}_{\overline{(j,a)} \in  \Z^2/(0,1)}$ where 
\begin{align*} \sigma_{\overline{(j,a)}} =  z^{j} \vartheta[a,0]\left(q,x\right)  
					      = \sum_{\ell \in \Z} q^{\left(\ell+a \right)^2/2}z^j x^{\ell+a}. \end{align*}
Pulling back through $f$ gives a basis  $\left\{\sigma'_{\overline{(j,a)}} \right\}_{\overline{(j,a)} \in \Z^2/(1,1)}$ for the global sections of $\L'$:
\begin{align*} \sigma_{\overline{(j,a)}}' &= z^{-a} f^*\left(\sigma_{\overline{(j,a)}}\right)= z^j \vartheta[a,0] \vartheta\left(q, z x\right)
=  \sum_{\ell \in \Z} q^{\left(\ell+a \right)^2/2} z^{\ell+j} x^{\ell+a}. \end{align*}

Let $\L''=\L \otimes \L' = \L(2,-1,1,1)$.  The basis $\left\{\sigma_{\overline{(j,a)}}'' \right\}_{\overline{(j,a)} \in \Z^2/(1,2)}$  for the sections of $\L''$ given by Equation \eqref{sectionid} is \[ \sigma_{\overline{(j,a)}}''=z^{\left(\frac{a}{2}+j\right)} \vartheta\left[a/2,0\right]\left(q^2, z x^2\right) = \sum_{\ell \in \Z} q^{\left(\ell+\frac{a}{2}\right)^2} z^{\ell+j} x^{2 \ell+a}. \]
Standard theta function identities (or the calculations of Subsection \ref{opsurf}) give
\begin{align}\label{productex} \sigma_{\overline{(0,0)}} \cdot \sigma_{\overline{(0,0)}}' & = \vartheta[0,0](q,x) \cdot \vartheta[0,0](q,zx); \\
& = \vartheta[0,0] (q^2,z) \vartheta[0,0] (q^2,z x^2)+\vartheta \left[1/2,0 \right] (q^2, z) \vartheta \left[1/2,0 \right]  (q^2, z x^2);\\
& = \sum_{n \in \Z} q^{\frac{n^2}{4}} \left(\sum_{\ell \in \Z} q^{(\ell+ n/2)^2} z^{k \ell +n} x^{2 \ell+n}\right);\\
& = \sum_{n \in \Z} q^{\frac{n^2}{4}} \sigma_{\overline{(n,n)}}'';
\end{align}
the third equality follows from expanding the power series in $z$.

We now view the multiplication of sections as the product 
$ \Ext^0(\O, \L') \otimes \Ext^0(\L\inv, \O) \to \Ext^0(\L\inv, \L') $
and calculate the product of the mirror elements on the symplectic side. The Lagrangians mirror to the line bundles $\L\inv$, $\O$, and $\L'$ are given by  $L(1,0,-1,1)$, $L(1,0,0,1)$, and $L(1,-1,1,1)$; denote these $L\inv, L_{\mathcal{O}}, L'$ respectively. Let $(L\inv)^{-\alpha}$ and $(L')^{\alpha'}$ denote the image of the Lagrangians under perturbation by the quadratic Hamiltonians $-\alpha t^2/2$ and $\alpha' t^2/2$, where $\alpha, \alpha' > 1$; so 
\begin{align*} (L\inv)^{-\alpha} & =\left\{ \left(Y, \begin{pmatrix} -\alpha & 0 \\
0 & -1 \end{pmatrix}Y \right) \mid Y \in \R^2\right\};  \\
(L')^{\alpha'} &= \left\{ \left(Y, \begin{pmatrix} \alpha' + 1 & 1 \\
1 & 1 \end{pmatrix}Y\right) \mid Y \in \R^2 \right\}. \end{align*}
By Proposition \ref{htc}, we can identify the relevant morphism spaces with the Floer complexes $HF^*\left( (L\inv)^{-\alpha} , L_{\O}\right)$, $HF^*\left(L_{\O},(L')^{\alpha'} \right),$ and  $HF^*\left((L\inv)^{-\alpha} ,(L')^{\alpha'} \right)$. Under the functor of Theorem \ref{ms1}, \begin{align*} \Phi_{\Wr}(\sigma_{\overline{(0,0)}})&=[\mathbf{y}_{\overline{(0,0)}}];\\
\Phi_{\Wr}(\sigma'_{\overline{(0,0)}})&=[\mathbf{y'}_{\overline{(0,0)}}], \end{align*}
where $\mathbf{y}_{\overline{(0,0)}}=e^{-2 \pi i \tau S\left(y_{\overline{(0,0)}}\right)} y_{\overline{(0,0)}}$ and $\mathbf{y'}_{\overline{(0,0)}}=e^{-2 \pi i \tau S\left(y'_{\overline{(0,0)}}\right)} y'_{\overline{(0,0)}}$ are the generators  at $(0,0,0,0) \in Y$ of the Floer complexes, weighted by action. The Floer product $\mu^2\left(\mathbf{y}_{(0,0)}, \mathbf{y'}_{(0,0)}\right)$ counts triangles homotopic in $\tilde{Y}=\R^4$ to the planar triangles shown in Figure \ref{triangle2}, where we have fixed lifts $\tilde{y}_{\overline{(0,0)}}$ and $\tilde{y}'_{\overline{(0,0)}}$ at $(0,0,0,0)$, as well as lifts of the Lagrangians. There is one such triangle $\Delta_n$ for each $n \in \Z$. We can solve for the coordinates of $\tilde{y}_{out,n}$ to obtain
\[ \tilde{y}_{out,n}= \left( n(1+2 (\alpha+\alpha'))\inv, n (\alpha+\alpha') (1+2 (\alpha+\alpha'))\inv, n, n\right); \]
recalling that the generators are indexed by their radial coordinates, we see that $y_{out,n}=y''_{\overline{(n,n)}}$. The action-corrected symplectic area of $\Delta_n$ is equal to 
\begin{align*} B(\Delta_n) &= \omega(\Delta_n)-S\left(y_{\overline{(0,0)}}\right)-S\left(y_{\overline{(0,0)}}\right)+S\left(y''_{\overline{(n,n)}}\right);\\
&= \frac{\alpha n^2 }{2(1+2 \alpha)}+\frac{n^2}{4(1+2 \alpha)}=\frac{n^2}{4}. \end{align*}
We conclude that 
\[ \left[\mathbf{y}_{\overline{(0,0)}}\right] \cdot \left[\mathbf{y'}_{\overline{(0,0)}}\right] = \sum_{n \in \Z} q^{\frac{n^2}{4}} \left[ \mathbf{y''}_{\overline{(n,n)}} \right], \]
and comparing this expression to Equation \eqref{productex}, we verify that the two products agree.
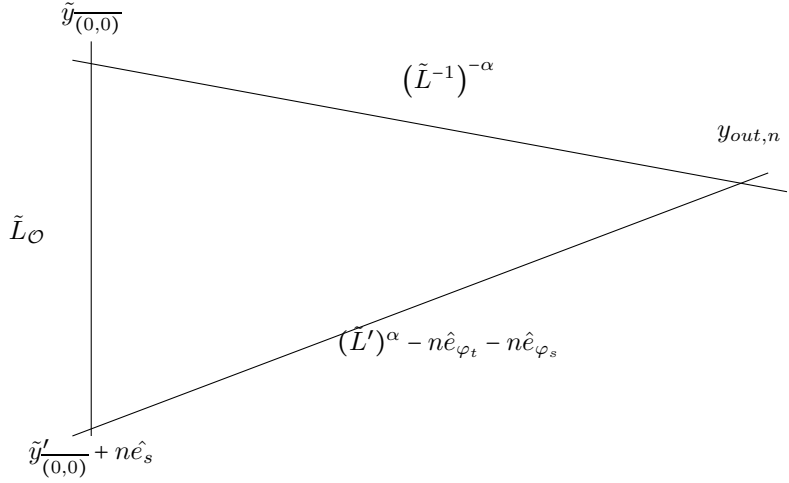
\begin{figure}
\centering
\begin{tikzpicture}
	
		\node (0) at (-2, 3) {};
		\node  (1) at (7.5, 1.25) {};
		\node  (2) at (-1.75, 3.25) {};
		\node  (3) at (-1.75, -2) {};
		\node  (4) at (-2, -2) {};
		\node  (5) at (7.25, 1.5) {};
		\node  (7) at (7, 2) {$y_{out, n}$};
		\node  (8) at (-2.6, .75) { $\tilde{L}_{\O}$};
		\node  (9) at (3, 2.75) {$\left(\tilde{L}\inv\right)^{-\alpha}$};
		\node  (11) at (-1.75, 3.55) {$\tilde{y}_{\overline{(0,0)}}$};
		\node  (12) at (3, -.75) {$(\tilde{L}')^{\alpha}-n \hat{e}_{\phi_t}- n \hat{e}_{\phi_s}$};
		\node  (14) at (-1.75, -2.3) {$\tilde{y}'_{\overline{(0,0)}}+n\hat{e_{s}}$};

		\draw (4.center) to (5.center);
		\draw (0.center) to (1.center);
		\draw (2.center) to (3.center);
	
\end{tikzpicture}\caption{A triangle contributing to  $\mu^2\left(\mathbf{y}_{(0,0)}, \mathbf{y'}_{(0,0)}\right)$ }\label{triangle2} 
\end{figure}




\end{eg}

\section{Mirror symmetry for compact surfaces}\label{cmpct}
\noindent We now prove a mirror symmetry statement for the non-algebraic compact surfaces $S_\tau(\A)$. Let  \begin{align*}
\A&=\left(\left(m_0, k_0\right), \left(m_{\infty},k_{\infty}\right)\right);\\
&A = \begin{pmatrix} m_0 & - m_{\infty} \\ k_0 & k_{\infty} \end{pmatrix};\\
&n = \det A =m_0 k_{\infty}+m_{\infty} k_0. \end{align*}
Let $\left(f, g, e^{2 \pi i \tau \lambda}e^{2 \pi i \theta}\right) \in \Q^2 \times \C^*$ define a representation of $\pi_1\left(S_\tau(\A)\right)$ and thus a line bundle $\L\left(f, g, e^{2 \pi i(\lambda \tau+ \theta)}\right)$ on $S_{\tau}(\A)$ via the correspondence of Section \ref{csurfaces}. Define the following type-$\A$ Lagrangian section: 
\[ L_{\left(f, g, e^{2 \pi i(\lambda \tau+ \theta)}\right)} = \left( \left\{ \begin{pmatrix} t \\ s \end{pmatrix}, \begin{pmatrix} -\lambda \rho_{-k_0/m_0}^{k_{\infty}/m_{\infty}}\left(t\right) \\ - \lambda \end{pmatrix}  +\begin{pmatrix} f \\ g \end{pmatrix} \large\mid \begin{pmatrix} t \\ s \end{pmatrix} \in \R \times S^1 \right\}, d+ 2 \pi i \theta \mathrm{d} s  \right). \] 

The remainder of this section is dedicated to the proof of the following theorem:
\begin{thm}\label{firstthm} For all $S_{\tau}(\A)$ which are not algebraic surfaces, i.e., for such surfaces with $n \neq 0$, there is a full and faithful functor 
\[ \Phi_{\A}: \Pic \left(S_{\tau}(\A)\right) \to H^0 \F_{\tau} \left(\A\right). \]
realized on objects by 
\[ \Phi_{\A}\left(\L\left(f, g, e^{2 \pi i( \tau \lambda + \theta)}\right)\right) = \left[L_{\left(f, g, e^{2 \pi i( \tau  \lambda +\theta})\right)}\right]. \]
 \end{thm}

As in the previous section, we define the map using specific choices of perturbing Hamiltonians. For each triple of Lagrangians $L_0, L_1, L_2$,  with \[L_i = L_{\gamma_i}; \gamma_i=\left(f_i, g_i, e^{2 \pi i \tau \lambda_i}e^{2 \pi i \theta_i}\right), \quad 0 \leq i \leq 2,\] define \[ L_i^{\Wr} = L_{(m_{\infty},-k_{\infty},0, e^{2 \pi i \tau (\lambda_i-g_i)}e^{2 \pi i \theta})}= L_{(m_{0},k_{0}, e^{2 \pi i \tau( \lambda_i-g_i)}e^{2 \pi i \theta})},\] and choose type-$\A$ Hamiltonians $H_i$, $0\leq i\leq2$, such that the following conditions are satisfied:
\begin{enumerate}[(a)]
 \item The set $(\psi_{H_0}(L_0), \psi_{H_1}(L_1),\psi_{H_2}(L_2))$ is in correct position.
 \item The function \[ \frac{\p H(t, s_0)}{ \p t}  -\lambda \rho_{-m_0/k_0}^{m_{\infty}/k_{\infty}}(t)\] controlling the $\phi_t$ coordinate is a monotonic function for all $s_0 \in S^1$.
 \item There exist type-$\W$ Hamiltonians $H_0', H_1', H_2'$ satisfying conditions (a)-(f) of Subsection \ref{hamform2} such that $(\psi_{H_i}(L_i)) \cap ([-1,1]) = \left(\psi_{H_i'}\left(L_i^\Wr\right)\right) \cap ([-1,1])$.
 \item When two lifts $\psi_{H_i}\left(\tilde{L}_i\right)$ and  $\psi_{H_j}\left(\tilde{L}_j \right)- j\hat{e}_{\phi_t}- a \hat{e}_{\phi_s}$ are essentially disjoint, so they are disjoint after perturbing by some compactly supported Hamiltonian, they are disjoint.
\end{enumerate}
These conditions ensure that for $y \in L_i \cap L_j$, \[S_{\R^4}(y)=Q_{\psi_{H_i}\left(L_i^{\Wr}\right), \psi_{H_j}\left(L_j^{\Wr}\right)}(y).\] We use the isomorphism
\[ HF^0\left(\psi_{H_i}\left(L_i\right), \psi_{H_j}\left(L_j\right)\right) \xrightarrow{\cong} \Hom_{H^0 \F_\tau(\A)} \left(L_0, L_1\right)\] 
to compute the cohomology category $ \Hom_{H^0 \F_\tau(\A)}\left(L_0, L_1\right)$.


\begin{lemma}\label{bigclemma} Let $\gamma_0\inv \gamma_1=(f,g, e^{2 \pi i ( \tau \lambda+\theta)})$. There is a bijection between solutions $\left(n_1, n_2\right) \in \N^2$ to Equations \eqref{cohoeq1}, \eqref{cohoeq2} and \eqref{cohoeq3} applied to $(f,g, e^{2 \pi i ( \tau \lambda+\theta)})$, and closed generators of degree zero of $\psi_{H_0}\left(L_0\right)$ and $\psi_{H_1}\left(L_1\right)$, where $H$ is a Hamiltonian of the above form. \end{lemma}
\begin{proof} Note that when $\lambda, \theta \notin \Z$, both sets are empty: if $\lambda \notin \Z$, $\psi_{H_0}\left(L_0\right) \cap \psi_{H_1}\left(L_1\right) = \emptyset$, and if  $\lambda \notin \Z$, $\theta \in \Z$, then there are no closed generators of degree zero by Proposition \ref{kingprop}.
 Otherwise, observe that we have chosen $H_0$ and $H_1$ such that when 
\[ 0> {-\lambda} \left( \frac{k_{\infty}}{m_{\infty}}+ \frac{k_0}{m_0} \right) = -\frac{\lambda n}{m_0 m_{\infty}}, \]
i.e. when $\lambda$ and $n$ have the same sign, there is a bijection 
\[ \{ \text{intersection points }  y_{j,N}  \text{ of degree zero}  \} \leftrightarrow \{ j \in \Z \cap I\}\]
where $I$ is the interval 
\[I=\left[-\frac{\lambda k_\infty}{ m_\infty} -f - \epsilon\left(H_0\right)+\epsilon(H_1), \frac{\lambda k_0}{ m_0} -f + \epsilon\left(H_0\right)-\epsilon(H_1)\right].\]
We claim that the map 
\begin{equation} \label{proposedbij} j \mapsto A \left(\begin{pmatrix} j \\ N \end{pmatrix}  +\begin{pmatrix} f \\ g \end{pmatrix}\right) \end{equation}
realizes the desired bijection. 

 Note that solutions in $\Z^2$ to Equations $\eqref{cohoeq1}$, $\eqref{cohoeq2}$ and \eqref{cohoeq3} are given by elements $\left(\ell_1, \ell_2\right) \in \Z^2$ such that  for some $j \in \Z$, 
\[ \frac{1}{n} \begin{pmatrix} k_{\infty} & -k_0 \\  m_{\infty} &   m_0 \end{pmatrix} \begin{pmatrix}\ell_1 \\ \ell_2 \end{pmatrix} - \begin{pmatrix} f \\ g \end{pmatrix}= - \begin{pmatrix} j \\ N\end{pmatrix}. \]
Denote this set of solutions $\Sigma$. We rewrite the above equation as
\[ A\inv \begin{pmatrix}\ell_1\\ \ell_2 \end{pmatrix} = \begin{pmatrix} f \\ g \end{pmatrix}- \begin{pmatrix} j \\ N \end{pmatrix}. \]
If we omit the requirement that $\ell_1, \ell_2,$ and $j$ are integers, then the set of solutions to this equation in $\R^2$ is given by a line in $\R^2$; the set of solutions $\left(\ell_1, \ell_2\right) \in \left(\R^+\right)^2$  is given by a (potentially empty) segment $S$ of this line.
Let
\[ I' = \left\{ x \in \R \mid A\inv  \begin{pmatrix} \ell_1 \\ \ell_2 \end{pmatrix}  =\begin{pmatrix} f \\ g \end{pmatrix}- \begin{pmatrix} x \\ N \end{pmatrix}, \,\, \begin{pmatrix} \ell_1 \\ \ell_2 \end{pmatrix} \in S.\right\}. \] 
$I'$ is non empty if and only if $\frac{\lambda n}{m_{\infty}} > 0$ and $\frac{\lambda n}{m_{0}}>0$, if and only if $\lambda$ and $n$ have the same sign. In this case the endpoints of $I$  correspond to the solutions $x$ to the equation above when $(\ell_1, \ell_2) =\left(\frac{\lambda n}{m_{\infty}},0\right)$ and $\left(0,\frac{\lambda n}{m_{0}} \right)$; thus 
\[ I'= \left[-\frac{\lambda k_\infty}{ m_\infty} -f , \frac{\lambda k_0}{ m_0} -f \right].\]
We now show that if $j$ is an integer, \[  A \left(\begin{pmatrix} j \\ N \end{pmatrix}  +\begin{pmatrix} f \\ g \end{pmatrix}\right) \in \Z^2. \]  Note that   
\[ A \left(\begin{pmatrix} j \\ N \end{pmatrix}  +\begin{pmatrix} f \\ g \end{pmatrix}\right)  = \begin{pmatrix} m_0 \left(j+f\right) + k_0 \left(g+N\right) \\ -m_\infty \left(j+f\right) + k_\infty \left(g+N\right)\end{pmatrix} .\]
Then the fact that the pair $\left(f,g\right)$ satisfies Equations \eqref{colattice1} and \eqref{colattice2} implies that $\left(\ell_1, \ell_2\right) \in \Z^2$. Thus the map of Equation \eqref{proposedbij} is a bijection from $I' \cap \Z$ to $\Sigma$.

It remains to show only  that $I' \cap \Z = I \cap \Z$. The $\left(m_0, k_0\right)$ boundary conditions ensure that  
\[ m_0 \left( \frac{\lambda k_0}{ m_0} -f\right) +  k N \in \N. \]
Since $N$ is an integer, 
$ \left( \frac{\lambda k_0}{ m_0} -f\right) \in \frac{1}{m_0} \Z.$
Since $\epsilon\left(H\right) < \frac{1}{m_0}$, there are no integers in the interval \[\left(\frac{\lambda k_0}{m_0} -f, \frac{\lambda k_0}{m_0} -f + \epsilon\left(H\right) \right];\] similarly, we deduce that there are no integers in the interval 
\[\left[ \frac{\lambda k_\infty}{ m_\infty} -f - \epsilon\left(H\right), \frac{\lambda k_\infty}{ m_\infty} -f \right).\] We conclude that $I' \cap \Z = I \cap \Z$.

\end{proof}

\begin{cor}\label{bigclemma2} Let  $\L\left(\gamma_1), \L(\gamma_0\right) \in \Pic S_{\tau}(\A)$. There is an isomorphism of $\C$-vector spaces
\[\phi: \Ext^0\left(\L\left(\gamma_0\right), \L(\gamma_1)\right)\cong HF_{\tau} ^0\left(\psi_{H_0}\left(L_{\gamma_0}\right), \psi_{H_1}\left(L_{\gamma_1} \right) \right)\] 
for $H_0, H_1$ as above. 
\end{cor}

\begin{proof}[Proof of Theorem \ref{firstthm}] Define the map on objects by  
\[ \Phi\left(\L(\gamma)\right)=[L_{\gamma}]. \]
Consider bundles $\L\left(\gamma_i\right) \in \Pic\left(S_{\tau}(\A)\right)$,  Lagrangians $L_{\gamma_i},$ and Hamiltonians $H_i$ for $0 \leq i \leq 2$ satisfying the conditions above; we use the isomorphism
\[HF^0(\psi_{H_i}( L_{\gamma_i}), \psi_{H_j}( L_{\gamma_j})) \xrightarrow{\cong}  \Hom_{H^0 \F_{\tau}(\A)} (L_{\gamma_i}, L_{\gamma_j}), i<j. \]

 Let  $\mathbf{s}_{n_1,n_2}^{H_0, H_1}=[e^{-2 \pi i \tau S_{\R^4}\left(\phi(z^{n_1}z^{n_2})\right)}\phi\left (z_1^{n_1} z_2^{n_2} \right)]$ where $\phi$ is the isomorphism defined above. Then define the functor $\Phi_{\A}$ on the generators by
\[ \Phi_{\A}\left(z_1^{n_1}z_2^{n_2} \right)=\mathbf{s}_{n_1,n_2}^{H_0, H_1}. \]
Note that by the assumptions on $H_0$ and $H_1$,  $S_{\R^4}(z_1^{n_1} z_2^{n_2})$ is independent of the choice of lifts of $s$, since it depends only on $t$; using the methods of Proposition \ref{actionvan}, one can also show that this class is independent of the choice of the Hamiltonians $H_0$ and $H_1$ satisfying our assumptions.

We now show that the map $\Phi$ respects the composition of morphisms. We have seen that the Yoneda product 
\[ H^0\left(S_{\tau}(\A),\L\left(\gamma_1\inv \gamma_2\right)\right) \otimes  H^0\left(S_{\tau}(\A), \L\left(\gamma_0\inv \gamma_1\right) \right) \to  H^0\left(S_{\tau}(\A), \L\left(\gamma_0\inv \gamma_2\right)\right), \]
is given by $z_1^{n_1} z_2^{n_2} \cdot z_1^{\ell_1} z_2^{\ell_2} = z_1^{n_1 + \ell_1} z_2^{n_2 + \ell_2}.$
 
 By Lemma \ref{onlyone}, the multiplication map $\mu^2(\phi(z_1^{n_1}z_2^{n_2}), \phi(z_1^{\ell_1}z_2^{\ell_2}))$ is equal to $e^{2 \pi i \Delta}y$, where $y$ is the generator in degree zero  in \[CF^*\left(\psi_{H_0}\left(L_{\gamma_0}\right), \psi_{H_2}\left(L_{\gamma_2}\right)\right)_{(A\inv((n_1+\ell_1, n_2+\ell_2)-(f_0, g_0) +(f_1,g_1)))},\]
 and $\Delta$ encodes the area and monodromy around a planar triangle. Note that this generator is equal to $\phi(z_1^{n_1 + \ell_1} z_2^{n_2 + \ell_2})$.   The action-corrected area and monodromy of the triangle $\Delta$ vanish, so 
 $e^{2 \pi i \Delta}y=\mathbf{s}_{n_1+\ell_1, n_2+\ell_2}^{H_0, H_2}$. Thus
 \begin{align*} [\Phi\left(z_1^{n_1} z_2^{n_2}\right)]\cdot[\Phi\left(z_1^{\ell_1} z_2^{\ell_2} \right)]& = [s_{n_1,n_2}]=\Phi_{\A} \left(z^{n_1}z^{n_2}z_1^{\ell_1}z_2^{\ell_2}\right). \\
 										\end{align*}

\end{proof}

\begin{remark}\label{gradedpiecewoo} Note that $\phi\left(z_1^{n_1} z_2^{n_2}\right)$ is the generator of
\[ CF_{\tau}^0\left(\psi_{H_0}\left(L_{\gamma_0}\right), \psi_{H_1}\left(L_{\gamma_1}\right)\right)_{\left(A\inv\left(\left(n_1, n_2\right)-\left(f, g\right)\right)\right)}=CF_{\tau}^0\left(\psi_{H_0}\left(L_{\gamma_0}\right), \psi_{H_1}\left(L_{\gamma_1}\right)\right)_{(j,N)} \]
for any choice of $H_0$, $H_1$. \end{remark}

\begin{remark} With slightly more work, the proof above can be extended to give an isomorphism on the cohomology categories $H^*\F_{\tau}(\A) \to \Pic S_{\tau}(\A)$, where the Hom spaces of the category $\Pic S_{\tau}(\A)$ are given by the full $\Ext$ groups. \end{remark}

\begin{eg}[Hopf surface]\label{HS}
In the case that $\mathbf{A}=((1,0), (1, 1))$, $S_{\tau}(\A)$ is the classical Hopf surface  $\C^2 \setminus \{0\} / \left(\mathbf{z} \sim e^{2\pi i \tau} \mathbf{z}\right)$; denote this surface by $S$. We have seen above that $\Pic_{\an} S \cong \Hom(\pi_1(S), \C) \cong \C^* $, with the isomorphism realized by $\xi \mapsto \L(\xi)$ where
\[\mathrm{Tot} \, \L(\xi) = \C^2 \setminus \{0\} \times \C/ \left( \left(\mathbf{z}, t\right) \sim \left(e^{2\pi i \tau} \mathbf{z}, \xi t\right) \right).  \]
 Each $\L(\xi)$ has vanishing cohomology unless $\xi = e^{2 \pi i \tau k}$ for $k \in \N$, in which case $\L({e^{2 \pi i \tau k}}) \cong \pi_{\A}^*(\mathcal{O}(k))$ and 
 \[ H^*(S, \L(\xi)) =  H^*(S^1; \C) \otimes H^*(\P^1, \O(k)).\]
 In this case the global sections of $\L(\xi)$ are given by the monomials $z^{n_1}z^{n_2}$ with $n_1 + n_2 = k$; note that these are the pullbacks under $\pi_{\A}$ of the monomials which give a basis for the global sections of $\O(k)$.

The mirror to the Hopf surface is the non-algebraic Landau-Ginzburg model $(Y, \A)$. The Lagrangian mirror to $\L(e^{2 \pi i \tau( k + i \theta)})$ is
\[ \L_{(e^{2 \pi i (\tau k + \theta)})} =  \left( \{ \left(t,s, -k \rho\left(t\right), -k\right) \mid \left(t,s\right) \in \R \times S^1 \}, \mathrm{d} + e^{2 \pi i \theta} \mathrm{d} s \right). \] 
where $\rho: \R \to [0,1]$ is a smooth increasing function interpolating between $0$ and $1$. Denote the mirror to the structure sheaf by $L_{\O}$.

We now calculate the morphism spaces $\Hom_{H^* \F_{\tau} (\A)}\left(L_{\O}, L_{(q^{k} e^{2 \pi i \theta})}\right)$, which we expect to be isomorphic to the cohomology spaces $H^*\left(S, \L\left(q^{k} e^{2 \pi i \theta}\right)\right)$. Recall that there is an isomorphism  
 \[\Hom_{H^* \F_{\tau} (\A)}\left(L_{\O}, L_{(q^{k} e^{2 \pi i \theta})}\right) \cong HF^* \left(\psi_{H}(L_{\O}), L_{(q^{k} e^{2 \pi i \theta})}\right)\] for a perturbing type-$\A$ Hamiltonian $H: Y \to \R$. When $k \notin \Z$, $H$ can be chosen such that the two Lagrangians, which are disjoint before perturbation, remain disjoint. Otherwise, we can choose $H$ such that the calculations of the Floer complexes and Floer product essentially reduce to calculations involving product Lagrangians in $T^*S^1 \times T^2$; see as pictured in Figure \ref{hopfex}. From this we see that the differential $\mu^1$ on the complexes is given by $\mu^1(\mathbf{y}_{i})=(e^{2 \pi i \theta}-1) \mathbf{z}_i$ for $0 \leq i \leq k$ and we conclude that 
\[ \Hom_{H^* \F_{\tau} (\A)}(L_{\O}, L_{\xi} ) = \begin{cases} H^*(S^1; \C) \otimes H^*(\P^1, \O(k)) & \xi = q^k, k \in \Z \\
													     0 & \text{otherwise}. \end{cases} \]
													     
When $k \geq 0$, the mirror functor $\Phi_{\A}$ sends $z^{n_1}z^{n_2}$ to the action-corrected generator $\mathbf{y}_{n_1}$. To check that the functor intertwines the Floer and Yoneda products, we note that computation of the Floer product reduces to a calculation in $T^* S^1$, and is given by 
\begin{align*} \mu^2\left(\Phi_{\A}\left(z^{n_1}z^{n_2}\right), \Phi_{\A}\left(z^{n_1'}z^{n_2'}\right)\right)&= \mu^2\left(\mathbf{y}_{n_1}, \mathbf{y}_{n_1'}\right)=\mathbf{y}_{n_1+n_1'}=\Phi_{\A}\left(z^{n_1+n_1'}z^{n_2+n_2'}\right).
\end{align*}

\begin{figure}
\includegraphics[width=300pt]{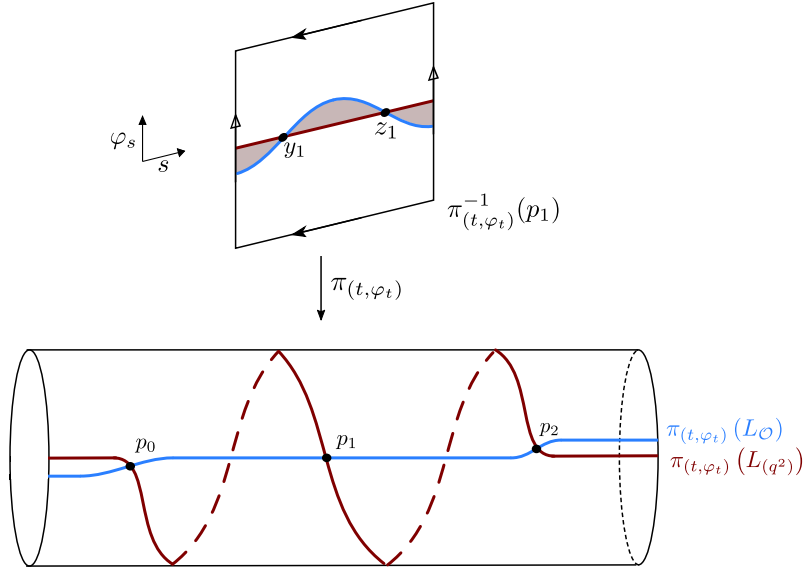}
\caption{We calculate the complex $CF^*(\psi_{H}(L_\O), L_{(\xi)})$ by choosing a perturbing Hamiltonian $H$ such that the two Lagrangians split as product Lagrangians in the region containing all intersection points and holomorphic curves contributing to the differential. The case $\xi=q^2$ is pictured; in this case there are six intersection points after perturbation. The differential counts two discs contained in the fiber of the projection map for each intersection point $y_i$ with $\ind y_i = \ind z_i - 1$. Here the two discs contributing to $\mu^1(y_1)$, which cancel, are shown.}
\label{hopfex}
\end{figure}

.

\end{eg}

\section{Diagrams of categories} \label{diagsec}
The following theorem shows that our mirror correspondences commute with localization. 

\begin{thm} \label{thediag2} For any pair $\left(\mathbf{A}\right)=\left(\left(m_0,k_0\right),\left(m_\infty,k_\infty\right)\right)$ such that $\left(m_0, k_0\right) \neq \left(m_\infty,-k_\infty\right)$, the following diagram commutes:

\begin{equation}
    \begin{tikzcd}
   \Pic S_{\tau}\left(\mathbf{A}\right)  \arrow[rr, near start, "1"] \arrow[dr, "9"] \arrow[dd,swap,"2"] &&
    \mathcal{C}_{\tau}\left(m_0, k_0\right) \arrow[dd, near start, "3" ] \arrow[dr, "11"] \\
    & H^0 \F_{\tau}\left(\mathbf{A}\right) \arrow[rr, near start, "5"] \arrow[dd, near start, "6"] &&
    H^0 \F_{\tau}\left(m_0, k_0\right) \arrow[dd,"7"] \\
    \mathcal{C}_\tau \left(m_\infty,k_\infty\right) \arrow[rr, near start, "4"] \arrow[dr, "10"] && \mathcal{C}_\tau \arrow[dr,"12"] \\
    & H^0 \F_{\tau}\left(m_\infty,k_\infty\right) \arrow[rr,"8"] && H^0 \F_{\tau, \W}
    \end{tikzcd}
\end{equation}
where: \begin{itemize}
\item the arrows 1, 2, 3, and 4 are restriction functors induced by the covering of $S(\A)$ by $D_{\tau}(m_0, k_0)$ and $D_{\tau}(m_{\infty}, k_{\infty})$;
\item the arrows 5 and 7 are the localization functors of Section \ref{inv}; 
\item the arrows 6 and 8 are the localization functors composed with the symplectomorphism $\Psi: Y\to Y$ given by $(s,t,\phi_s, \phi_t)\mapsto (s,-t,\phi_s, -\phi_t)$;
\item the arrows 9, 10, 11, and 12 are the mirror functors.
\end{itemize}.
\end{thm}

\begin{proof} The commutativity of the $(1,2,3,4)$ and $(5,6,7,8)$ squares is immediate; the commutativity of the $(2,6,9,10)$ and $(3,7,11,12)$ squares follows from Theorem \ref{commute1} and Remark \ref{inversion}. Thus it remains only to check that the $(2,6,9,10)$ and $(1,5,9,11)$ squares commute. Note that the maps $3,4,7,$ and $8$ are embeddings of categories; thus to show that these remaining squares commute it suffices to show that the diagram
\[ \begin{tikzcd}  &\Pic S_{\tau}\left(\mathbf{A}\right)\arrow{r}{\Phi_{\A}}\arrow{d}{i_{\A}^*} & H^0 \F_\tau(\A)\arrow{d}{\iota_{(m_{0},k_0)}  \circ \iota_{\A}}\\
&\mC_{\tau} \arrow{r}{\Phi_{\Wr}}  &H^0 \F_{\tau, \W}. \end{tikzcd}\]
commutes.

Let $\L_i= \L(\gamma_0)=\L(f_i, g_i, \xi_i)$ $i=0,1$, be line bundles in $\Pic S_\tau(\A)$. Let $\L = \L_0\inv \L_1=\L(f,g,\xi)$.  Let $L_i=L_{(\gamma_i)}$ denote the corresponding type-$\A$ Lagrangians. Recall that $\iota^*_{\A}(\L_i)=\L(m_0, k_0, 0, \xi e^{-2 \pi i g})$; let \[L_i^{\Wr}=L_{\left(m_i, k_i, 0, \xi_i e^{-2 \pi i g_i}\right)}.\] Let $H_0$ and $H_1$ be type-$\A$ Hamiltonians which satisfy the constraints above, and $H_0'$ and $H_1'$ type-$\W$ Hamiltonians with  $(\psi_{H_i}(L_i)) \cap ([-1,1]) = \left(\psi_{H_i'}\left(L_i^\Wr\right)\right) \cap ([-1,1])$. Let $G_i: \R \times S^1$ be small functions with  support within $[-1,1]\times S^1$ such that the functions $H_i'+G_i$ and $H_i ' - H_i+G_i$  are smooth and such that $\psi_{H_i ' - H_i+G_i}\left(\psi_{H_i}\left(L_i\right)\right)$ and $\psi_{H_i}\left(L_i\right)$ intersect transversely; note that this implies that $ H_i'+G_i$ and $H_i ' - H_i+G_i$ are type-$\W$ perturbing Hamiltonians.  Then note that 
\[ \psi_{H_i ' - H_i+G_i}\left(\psi_{H_i}\left(L_i\right)\right)=\psi_{H_i'+G_i}\left(L_i^{\Wr}\right);\] thus $L_i$ and $L_i^{\Wr}$ represent the same object in $H^0 \F_{\Wr}$. 
We conclude that on objects, the diagram reads
\[ \begin{tikzcd} & \L_i \arrow[mapsto]{r} \arrow[mapsto]{d} & \left[L_i \right]\arrow[mapsto]{d}\\
& i_{\A}^*\L_i \arrow[mapsto]{r} & \left[L_{i}\right]. \end{tikzcd}\]

Now let $z_1^{n_1} z_2^{n_2} \in \Ext^{i}(\L_0, \L_1)\cong H^0(S_{\tau}(\A), \L)$. Then there exists $N \in \N$ such that $\xi e^{-2 \pi i g}=e^{2 \pi i \tau N}$. Using the change of coordinates formula of Equation \eqref{cocwoo}, we know that  \[ i^*_{\A}(z_1^{n_1} z_2^{n_2})=z^{(k_{\infty} n_1 - k_0 n_2)} x^N=\sigma_{i^*_A(\L),\overline{(j,N)}}.\]

 Let $[s_{n_1, n_2}] \in \Hom_{H^0 \F_\tau(\A)}$ be represented by $\mathbf{s}_{n_1, n_2}^{H_0, H_1} \in CF_{\tau}^0\left(\psi_{H_0}\left(L_0\right), \psi_{H_1}\left(L_1\right)\right)$. By Remark \ref{gradedpiecewoo}, we know that \[\mathbf{s}_{n_1, n_2}^{H_0, H_1} \in CF_{\tau}^0\left(\psi_{H_0}\left(L_0\right), \psi_{H_1}\left(L_1\right)\right)_{\overline{(j, N)}}.\] Using the same methods as in the proof of Lemma \ref{inclusion}, we can show that 
\[ \mu^2\left(c_{H_i \to H_i ' - H_i+G_i},\mathbf{s}_{n_1, n_2}^{H_0, H_1}\right) = \mathbf{y}_{\overline{(j,N)}}^{H_0, H_1}, \]
 as both are the action-corrected generators of degree zero in the $\overline{(j,N)}$ graded piece of the Floer complex. Thus on morphisms, the diagram reads
\[ \begin{tikzcd} & z_1^{n_1} z_2^{n_2} \arrow[mapsto]{d}\arrow[mapsto]{r} & \left[s_{n_1, n_2}\right]\arrow[mapsto]{d}\\
&\sigma_{\L, \overline{(j,N)}} \arrow[mapsto]{r} & \left[y_{L_0, L_1, \overline{(j,N)}}\right]. \end{tikzcd}\]
				
We conclude that the diagram in the statement of the theorem commutes.

\end{proof}

\appendix
\section{Lagrangians mirror to torsion sheaves}\label{appendix}
The heuristics that guide our mirror construction can also be used to find Lagrangians mirror to sheaves which are supported on finitely many elliptic fibers on the Hopf surfaces we consider. These novel Lagrangians are objects of an enlargement of the Fukaya category constructed in the main section of this paper.

Recall that the mirror space $Y$ is the product of the SYZ mirrors to $\C^*$ and the elliptic curve, and decompose $Y$ as \[ Y= (\R \times S^1, d t \wedge d \phi_t) \times  (S^1 \times S^1, d s \wedge d \phi_s).\] Consider a Fukaya category consisting of Lagrangians equipped with higher rank bundles with flat connections, like those in the work of Polischuk and Zaslow on mirror symmetry for the elliptic curve  (\hspace{1sp}\cite{PZ}). We expect that the product decomposition of $Y$ commutes with the mirror functor in the following sense: Given a pair  $(L_1, V_1, \nabla_1)$ and $(L_2, V_2, \nabla_2)$ of Lagrangians equipped with vector bundles with flat connections, such that $L_1$ is a Lagrangian submanifold $T^* S^1$ and $L_2$ is a Lagrangian submanifold of $T^2$, we can form a product Lagrangian $(L_1 \times L_2, V_1 \boxtimes V_2, \nabla_1 + \nabla_2)$ in $Y$. We expect that given two coherent sheaves $\F$ and $\G$ on $\C^*$ and $E$ respectively, with mirror Lagrangians $L_{\F}$ and $L_{\G}$, the sheaf $\F \boxtimes \G$ on $\C^* \times E$ is mirror to the Lagrangian $L_{\F} \times L_{\G}$.  This is made precise in the remainder of this section.

\begin{remark}[Notation] Let $E^{\tau}$ denote the torus with complex K\"ahler parameter $\tau$.
Let $CF^*_{E^{\tau}}(L, L')$ denote the Floer complexes over $\C$ associated to an elliptic curve with complex K\"{a}hler parameter $\tau$ that appear in \cite{PZ}. We use ``Lagrangian" to mean the data of a Lagrangian equipped with a vector bundle and a flat connection. Let $\F(E^{\tau})$ denote the Fukaya category of such Lagrangians, as in \cite{PZ}. Let $\Phi_{E_{\tau}}$ denote the mirror functor of  \cite{PZ} that interchanges coherent sheaves on $E_{\tau}$ with Lagrangians on $E^{\tau}$, with the sign conventions reversed (see Footnote \ref{foot}). \end{remark} 

\subsubsection{Sheaves supported on interior fibers}
Let $z= \exp(R_0+2 \pi i \theta_0) \in \C^*$. The SYZ mirror to the skyscraper sheaf $\O_z$ is the Lagrangian \[ L_z = \left(\{ (R_0, \phi_t) \mid \phi_t \in S^1\}, V_0, \mathrm{d}+ 2 \pi i \theta_0 \mathrm{d} \phi_t \right)\] consisting of the radial circle with radius $R_0$ and the flat connection with monodromy $e^{2 \pi i \theta_0}$ on the trivial rank-one vector bundle $V_0$. Let $\F$ be a coherent sheaf on $E_{\tau}$ and let $\Phi_{E_{\tau}}(\F)=\L_F$. Let $i_z: E_{\tau} \to \C^* \times E_{\tau}$ denote the inclusion map which sends $E_{\tau}$ to $\{z\} \times E_{\tau}$. Then 
\[ i_{z*}(\F) = \O_{z} \boxtimes \F \]
so we expect that $i_{z*} \F $ is mirror to $L_z \times L_{\F}$. Let $L_z \times L_{\F} = I_z(L_{\F})$.

 We provide partial verification of this expectation with the following proposition:
 \begin{prop}\label{resiso} Let $\G$ be a line bundle on on a surface $S_{\tau}(\A)$ mirror to a Lagrangian $L_{\G}$. Then for all such $z \in \C^*$ and $\F \in \Coh (E_{\tau})$, there is an isomorphism
 \[ \Ext^*(i_{z*} \F, \G) \cong HF_{\tau}^*( I_z(L_{\F}), L_{\G}). \] 
  We allow the case where $S_{\tau}( \A)$ is algebraic. \end{prop}
 
 We need the following lemma:
 \begin{lemma}[\hspace{1pt}\cite{PZ}]\label{PZadjunction}\begin{enumerate}[(a)]
 \item
 Let $\pi_m: E_{\tau} \to E_{m \tau}$ denote the isogeny of elliptic curves, and let $\Pi_{m}: E^{m \tau} \to E^{\tau}$ denote the $m$-fold covering of tori which sends $(x, \phi_x)$ to $(x, m \phi_x)$. The induced functor $\Pi_{m}^*: \F(E^{\tau}) \to \F(E^{m \tau})$ that sends a Lagrangian to its preimage satisfies  
 \[ \Phi_{m \tau} \circ \pi_{m*} = \Pi_m^* \circ \Phi_{\tau}. \]
 \item  Write $E^{\tau}$ as $(S^1)^2$ with coordinates $(x, \phi_x)$.  Let $T^B_{e^{2 \pi i (\tau \eta + \nu )}}$ denote the translation functor on the Fukaya category of $E^{\tau}$: \[(L, V, \nabla) \mapsto (L+ \eta \hat{e}_x, V, \nabla + 2 \pi i \nu \Id_V d \phi_x).\] Let $T^A_{\xi}$ denote the translation functor on  $E_{\tau}$ given by $z \mapsto \xi z$. Then \[T^B_{\xi} \circ \Phi_{E_{\tau}} = \Phi_{E_{\tau}} \circ T^{A*}_{\xi}.\] \end{enumerate} 
 \end{lemma}
Part (a) is adapted from Proposition 4 of \cite{PZ} (with $\pi_m$ denoting the dual isogeny to that which appears there); part (b) is evident from the mirror correspondence in Section 5.1 of the same paper.
 
 \begin{proof}[Proof of Proposition \ref{resiso}] By Corollary \ref{extension3}, we can write $i_{\A}^*\G= \L_{\tau}(m_0,k_0,d,\xi)$ for some $d \in \Z$, $\xi \in \C$. Let $i_z^*\G= \G_z \in \Pic^{dm_0}_{E_{\tau}}$, and note  \[ \G_z = i_{z}^* \L_{\tau}(m_0,k_0,d,\xi) =  T_{\xi\inv z^{k d}}^*\left( \L_{\tau,0} \right)\cdot \L_{\tau,0}^{m_0 d-1} .\] Let $\G_0=i_0^*\G \in \Pic^{d}_{E_{m_0 \tau}}$, and note that $\pi_{m_0}^*(T^{B*}_{z^{k_0/m_0}} \G_0)=\G_{z}$.
 
 Consider the map $f: Y \to \R \times S^1$ given by $y \mapsto (t(y), m_0 \phi_t(y) + k_0 \phi_s(y))$; note that 
 \begin{align*} f \left( L_{\G}\right)&= \gamma; \\
 f \left( I_z(L_{\F})\right) &= \{R_0\} \times S^1,
\end{align*}
where $\gamma \subset \R \times S^1$ is a curve which intersects the circle $f \left( I_z(L_{\F})\right)$ at one point $p$. Let $Y_p = f\inv(p)$. The map  $\wt{Y_{p}} \to (\R/\Z)\times (\R/m_0 \Z)$  defined by $y \mapsto (s(y), \phi_s(y))$ descends to a symplectomorphism \[(Y_{p}, \omega|_{Y_p}) \to ((\R/\Z)\times (\R/m_0\Z), d x \wedge d \phi_x).\] 
Then one can verify the following identities for the intersections of the Lagrangians with $Y_p$: 
\begin{align*} I_z(L_{\F}) \cap Y_{p} &=  T_{e^{-2 \pi k i \theta_0}} \circ \Pi_{m_0}^* (\Phi_{E_{\tau}}(\F));\\
L_{\G} \cap Y_p &= T_{e^{-2 \pi k i \theta_0}} \circ \Phi_{E_{m_0 \tau}}(T^{B*}_{z^{k_0/m_0}} \G_0).
   \end{align*}
By choosing an almost-complex structure $\J$ for which $f$ is holomorphic, we can guarantee that all holomorphic curves contributing to the differential on the Floer complex are contained in the fiber $Y_{p}$. Then there is a canonical isomorphism
\[ CF^*_{\tau}(I_z(L_{\F}), \L_{G}) \cong CF^*_{E^{m_0 \tau}}(I_z(L_{\F}) \cap Y_p,  \L_{G} \cap Y_p). \]
Then the Hom space is given by 
\begin{align*}  HF^*_{E^{m_0 \tau}}(I_z(L_{\F}) \cap Y_p,  \L_{G} \cap Y_p) &\cong HF^*_{E^{m_0 \tau}} \left(  T_{e^{-2 \pi k_0 i \theta_0}} \circ \Pi_{m_0}^* (\Phi_{E_{\tau}}(\F)),  T_{e^{-2 \pi k_0 i \theta_0}} \circ \Phi_{E_{m_0 \tau}}( T^{B*}_{z^{k_0/m_0}} \G_0)\right)\\
 & \cong HF^*_{ E^{m_0 \tau}} \left( \Pi_{m_0}^* (\Phi_{E_{\tau}}(\F)), \Phi_{E_{m_0 \tau}} \left(T^{B*}_{z^{k_0/m_0}} \G_0\right) \right) \\
& \cong \Ext^*(\pi_{m_0*} \F, T^{B*}_{z^{k_0/m_0}} \G_0 ) \\
& \cong \Ext^*(\F, \G_z) \\
& \cong \Ext^*(i_{z*} \F, \G). \end{align*}
The second isomorphism is the mirror functor and Lemma \ref{PZadjunction}; the third follows from the $(\pi_{m_0^*}, \pi_{m_0}^*)$ adjunction.
\end{proof}

\subsubsection{Sheaves supported on boundary fibers}We now consider sheaves supported on the boundary fiber $\pi_{\A} \inv(0)$. Recall that $\pi_{\A} \inv(0)$ is the genus-one curve $E_{m_{0} \tau}$ isogenous to the interior fibers $E_{\tau}$. Let $\F$ be a coherent sheaf on $E_{m_{0} \tau}$, and consider the sheaf $i_{0*} \F$. We claim that $i_{0*} \F$ is mirror to a ``U-shaped" Lagrangian constructed in the following way: Let $f: Y \to \R \times S^1$ be the map given above. Choose a U-shaped arc $\gamma \in \R \times S^1$ which extends to infinity in the negative direction as pictured in Figure \ref{ushaped}, with the ends of the arc asymptotic to the line $m_0 \phi_t + k_0 \phi_s=0$. For a point $p \in \gamma$, parametrize $f\inv(p)=Y_p$ using coordinates adapted to the $(m_0,k_0)$ boundary conditions: map $\wt{Y}_p \to (\R/\Z) \times (\R/m_0 \Z)$ by \[y \mapsto (s - (k_0/m_0) t, \phi_s).\] Let $L_\F = \Phi_{E_{m_0 \tau}}(\F)$ be mirror to $\F$, and construct a Lagrangian $I_0(L_{\F})$ in $Y$ which fibers over $\gamma$ and has the property that $I_0(L_{\F}) \cap Y_{p}= L_{\F}$ by symplectic parallel transport. In the chosen coordinates, the parallel transport is the identity map. 

\begin{remark} The Lagrangians of the form  $I_{0}(L_\F)$, and their cousins of the form $I_{\infty}(L_\F)$, are objects of a Fukaya-Seidel category consisting of Lagrangians that, at each end, fiber over curves  that asymptotically approach the lines $m_0 \phi_t + k_0 \phi_s=0$ and $m_{\infty} \phi_t - k_{\infty} \phi_s=0$. The morphisms in this proposed category are computed via the same procedure as in the main text, and in particular the Floer complex $CF^*(\psi_{H}(L), L')$ for $H: Y \to \R$ a type-$\A$ Hamiltonian should be quasi-isomorphic to the Hom space between two Lagrangians $L$ and $L'$; this motivates the following proposition. \end{remark}
 
\begin{prop}\label{resiso} Let $\G$ be a line bundle on on a surface $S_{\tau}(\A)$ mirror to a Lagrangian $L_{\G}$. Then for all $\F \in \Coh \, E_{m_0 \tau}$, there is an isomorphism
 \[ \Ext^*(i_{0*} \F, \G) \cong HF_{\tau} ^*(\psi_{H}(I_0(L_{\F})), L_{\G}) \] 
 for $H$ a type-$\A$ Hamiltonian such that  $\psi_{H}(I_0(L_{\F}))$ and $L_{\G}$ intersect transversely.  \end{prop}

\begin{proof}After possibly composing with a compactly supported Hamiltonian isotopy, we can ensure that there is a unique $p \in \psi_{H} (\gamma)  \cap f(L_{\G})$, and 
\begin{align*} \psi_{H}( I_0(L_{\F})) \cap Y_p &= L_{\F} ;\\
		 L_{\G} \cap Y_p &= \Phi_{E_{m_0 \tau}} (\G_0). \end{align*}
Thus again there is a sequence of isomorphisms:
\[ CF^*_{\tau}(\psi_{H}(I_0(L_{\F})), L_{\G}) \cong  CF^*_{E^{m_0 \tau}}(\Phi_{E_{m_0 \tau}} (\F), \Phi_{E_{m_0 \tau}} (\G_0)) \cong \Ext^*(\F, \G_0) \cong \Ext^*(i_{0*} \F, \G). \]
The situation is pictured in Figure \ref{ushaped}. 
\end{proof}
 
\begin{figure}
\includegraphics[height=180pt]{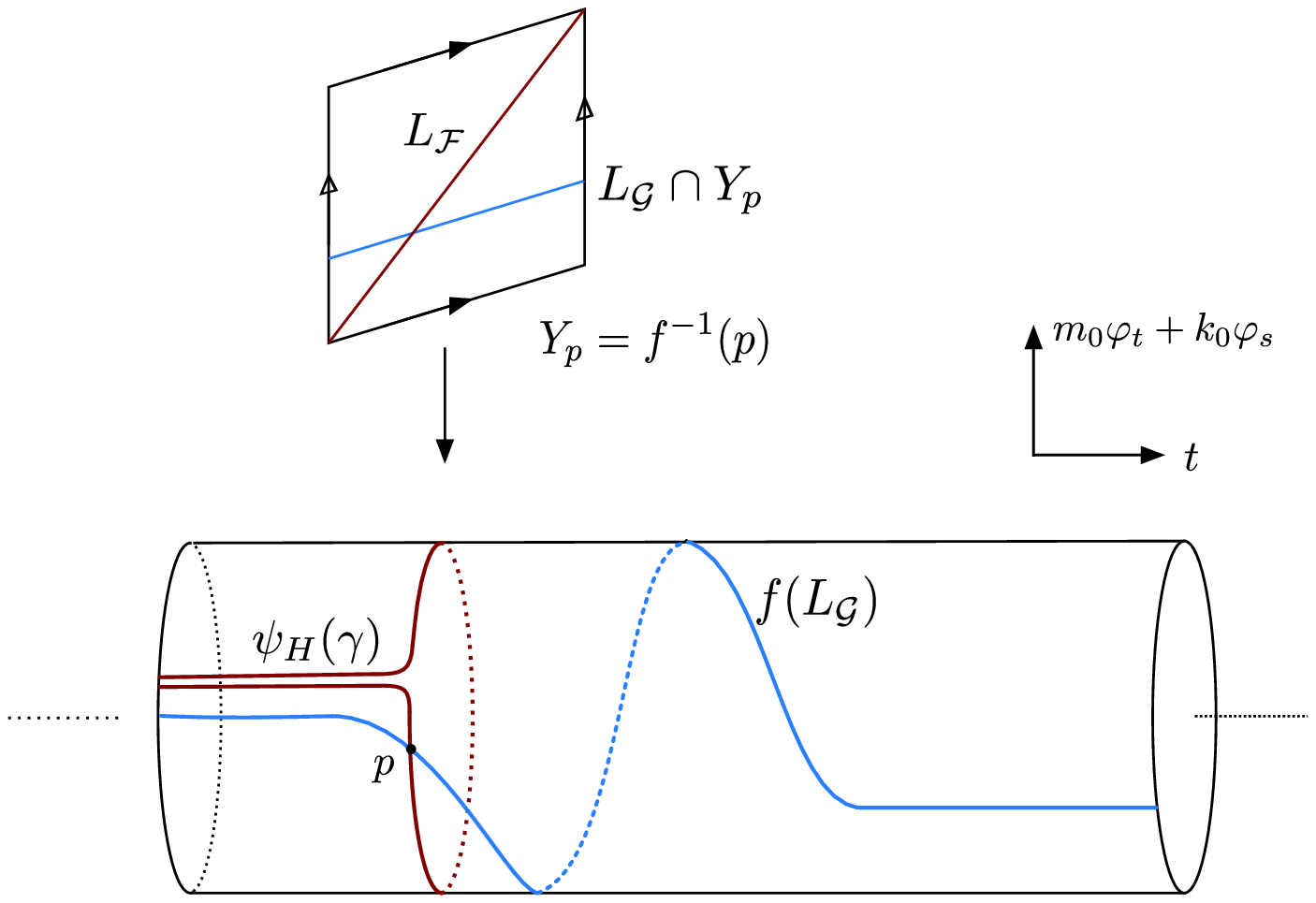}
\caption{}
\label{ushaped}
\end{figure}
 
\begin{remark} The functor $I_0: \F(E^{m_0 \tau}) \to \F (\A)$ is the Orlov functor of Abouzaid and Ganatra  (\hspace{1sp}\cite{AbGa}, \cite{Sylvanfunctor}), which in general produces Lagrangians mirror to Lagrangians supported on the fiber of a Landau-Ginzburg model  (\hspace{1sp}\cite{Jeffs}). In fact this is a very simple instance of this phenomenon. Consider the case where $(m_0, k_0)=(1,0)$, so $D_{\tau} (m_0,k_0)=\C \times E_{\tau}$. The mirror to $\C$ is the manifold $T^* S^1$ stopped at one point in the boundary at infinity, and the linking arc of the stop is mirror to the structure sheaf of the point $\{0\}$  (\hspace{1sp}\cite{GPS2}). Denote the arc linking the stop by $L_{0}$. Then note that for a Lagrangian $L_{\F}$ mirror to $\F \in \Coh(E_{\tau})$, the conjectured mirror to $i_{0*} \F$ is the product 
\[ I_{0}(L_\F) = L_{0} \times L_{\F}  \]
as expected.

All $U$-shaped Lagrangians mirror to sheaves supported on the fiber over $0$ take this form after a different choice of symplectic splitting of (a finite cover of) $Y$, which corresponds to a choice of algebraic coordinates on the surface $D(m_0, k_0)$.
 \end{remark}

\subsubsection{Further work} With some further work, one could prove that the inclusion functors constructed above intertwine the Floer and Yoneda products, and so the domain of the  mirror functors that appear in this work could be expanded to a subcategory of $\Coh_{\an} S_{\tau}(\A) $ which includes torsion sheaves. However, for non-algebraic $S_{\tau} (\A)$ the subcategory of $D^b \Coh_{\an} \, S_{\tau}(\A)$ generated by torsion sheaves and line bundles is a strict subcategory: for example, the rank-two vector bundles on the Hopf surface constructed in \cite{Mo} are not generated by these objects. (One can show that the objects of the subcategory in question restrict to the same object of $D^b \Coh\, E_{\tau}$ on all but finitely many fibers on the open subset $\C^* \times E_{\tau} \subset S_{\tau}$ fibering over $\C^*$, a property not enjoyed by these rank-two bundles.) We therefore expect that any proof of Conjecture \ref{expected}, i.e., of an equivalence of triangulated categories, would take the form sketched in Subsection \ref{pproof} and proceed via Theorem \ref{thediag2}.

\bibliographystyle{amsplain}
\bibliography{citations}

\end{document}